\documentclass[11pt]{article}
\usepackage{setspace}
\usepackage{cprotect}
\usepackage{amsmath,amssymb}
\usepackage{amsthm}
\usepackage[noend]{algorithmic}
\usepackage[ruled,vlined]{algorithm2e}
\usepackage{url}
\usepackage{fullpage}
\usepackage{makeidx}
\usepackage{enumerate}
\usepackage{fullpage}
\usepackage{graphicx,float,psfrag,epsfig,caption,subcaption}
\usepackage{epstopdf}
\usepackage{color}
\usepackage{hyperref}
\usepackage{xcolor}
\hypersetup{
    colorlinks,
    linkcolor={blue!80!black},
    citecolor={green!50!black},
}
\colorlet{linkequation}{blue}
\usepackage[title]{appendix}
\usepackage{cases}
\usepackage{footmisc}
\usepackage{lmodern}

\usepackage[mathscr]{euscript}

\DeclareSymbolFont{rsfs}{U}{rsfs}{m}{n}
\DeclareSymbolFontAlphabet{\mathscrsfs}{rsfs}

\numberwithin{equation}{section}

\newtheoremstyle{myexample} 
    {\topsep}                    
    {\topsep}                    
    {\rm }                   
    {}                           
    {\bf }                   
    {.}                          
    {.5em}                       
    {}  

\newtheoremstyle{myremark} 
    {\topsep}                    
    {\topsep}                    
    {\rm}                        
    {}                           
    {\bf}                        
    {.}                          
    {.5em}                       
    {}  

\newtheorem{claim}{Claim}[section]
\newtheorem{lemma}[claim]{Lemma}
\newtheorem{fact}[claim]{Fact}

\newtheorem{theorem}{Theorem}
\newtheorem{proposition}[claim]{Proposition}

\newtheorem{definition}[claim]{Definition}

\theoremstyle{myremark}
\newtheorem{remark}{Remark}[section]

\theoremstyle{myremark}

\theoremstyle{myexample}
\newtheorem{example}[remark]{Example}

\makeatletter
\newcommand\footnoteref[1]{\protected@xdef\@thefnmark{\ref{#1}}\@footnotemark}
\makeatother

\mathchardef\mhyphen="2D

\def\<{\langle}
\def\>{\rangle}
\def\P{{\mathbb P}}

\def\integers{{\mathbb Z}}

\def\E{{\mathbb E}} 

\def\normal{{\sf N}}

\newcommand\myeqref[1]{{Eq.\,\eqref{#1}}}

\def\reals{\mathbb{R}}

\def\normal{{\sf N}}

\def\cB{{\mathcal{B}}}
\def\cC{{\mathcal{C}}}
\def\cD{{\mathcal{D}}}
\def\cF{{\mathcal{F}}}

\def\cI{{\mathcal{I}}}

\def\cP{{\mathcal{P}}}
\def\cR{{\mathcal{R}}}

\def\cL{{\mathcal{L}}}

\def\cPR{{\mathcal{PR}}}

\def\frf{{\mathfrak{f}}}
\def\frg{{\mathfrak{g}}}

\def\frA{{\mathfrak{A}}}
\def\frD{{\mathfrak{D}}}
\def\frF{{\mathfrak{F}}}
\def\frR{{\mathfrak{R}}}

\def\bzero{{\boldsymbol 0}}

\def\bg{{\boldsymbol g}}

\def\by{{\boldsymbol y}}
\def\bu{{\boldsymbol u}}
\def\br{{\boldsymbol r}}

\def\bI{{\boldsymbol I}}

\def\bx{{\boldsymbol x}}
\def\bz{{\boldsymbol z}}

\def\bX{{\boldsymbol X}}

\def\bv{{\boldsymbol v}}
\def\bu{{\boldsymbol u}}
\def\bw{{\boldsymbol w}}

\def\btheta{{\boldsymbol \theta}}

\def\bb{{\boldsymbol b}}

\def\bh{{\boldsymbol h}}

\def\bg{{\boldsymbol g}}

\def\blambda{{\boldsymbol \lambda}}

\def\U{{\rm U}}
\def\U1{{\rm U}(1)}

\title{Approximate separability of symmetrically penalized least squares in high dimensions: characterization and consequences}
\author{Michael Celentano\thanks{Department of Statistics, Stanford University} \thanks{Supported in part by NSF Grant DGE -- 1656519}}
\date{\today}

\begin{document}

\maketitle

\begin{abstract}
    We show that the high-dimensional behavior of symmetrically penalized least squares with a possibly non-separable, symmetric, convex penalty in both \emph{(i)} the Gaussian sequence model and \emph{(ii)} the linear model with uncorrelated Gaussian designs nearly matches the behavior of least squares with an appropriately chosen separable penalty in these same models.
    The similarity in behavior is precisely quantified by a finite-sample concentration inequality in both cases.
    Our results help clarify the role non-separability can play in high-dimensional M-estimation.
    In particular, if the empirical distribution of the coordinates of the parameter is known --exactly or approximately-- there are at most limited advantages to using non-separable, symmetric penalties over separable ones.
    In contrast, if the empirical distribution of the coordinates of the parameter is unknown, we argue that non-separable, symmetric penalties automatically implement an adaptive procedure which we characterize.
    We also provide a partial converse which characterizes adaptive procedures which can be implemented in this way.
\end{abstract}


\section{Introduction}

In this paper, we consider estimation in two closely related statistical models.
First, we consider the Gaussian sequence model
\begin{equation}\label{gaussian-sequence-model}
    \by = \btheta + \tau \bz,
\end{equation}
where $\btheta \in \reals^p$ is a parameter vector, possibly fixed or random, $\tau \geq 0$, and $\bz \sim \normal(0,I_p)$.
In the sequence model, the statistician observes $\by$ and estimates $\btheta$.
Second, we consider the linear model 
\begin{equation}\label{linear-model}
    \by = \bX \btheta + \bw,
\end{equation}
where $\btheta \in \reals^p$, $\bw \in \reals^n$, and $X_{ij} \stackrel{\mathrm{iid}}\sim \normal(0,1/n)$.
In the linear model, the statistician observes $\by$ and $\bX$ and wishes to estimate $\btheta$.
The noise $\bw$ may or may not be random, but we require that it be independent of the design $\bX$.

M-estimation is a popular approach to estimation which involves solving a data-depending optimization problem. 
The M-estimators we consider minimize the sum of a least squares loss and a convex regularization penalty.
In the sequence model \eqref{gaussian-sequence-model}, such M-estimators are commonly known as proximal operators.
In particular, if $f_p:\reals^p \rightarrow \bar\reals$ is a lower semi-continuous (lsc), proper, convex function,
the proximal operator $\mathsf{prox}[f_p]:\reals^p \rightarrow \reals^p$ is defined via
\begin{equation}\label{eqdef:finite-p-prox}
    \mathsf{prox}[f_p](\by) = \arg\min_{\bx\in \reals^p} \left\{\frac12\|\by - \bx\|^2 + f_p(\bx)\right\}.
\end{equation}
In the linear model \eqref{linear-model}, such M-estimators are defined to satisfy
\begin{equation}\label{lm-cvx-estimator}
    \widehat \btheta \in \arg\min_{\bb \in \reals^p} \left\{\frac1{2n}\|\by - \bX \bb\|^2 + f_p(\bb)\right\},
\end{equation}
where we specify set membership rather than equality because the minimizing set in \eqref{lm-cvx-estimator}, unlike that in \eqref{eqdef:finite-p-prox}, is not necessarily a singleton.
When the statistician only has access to (or chooses only to exploit) structural information on the empirical distribution of the coordinates of $\btheta$ (and not the order in which these coordinates appear), it is natural to restrict attention to $f_p$ which are invariant to permuting the coordinates of its argument.
Such $f_p$ we will call \emph{symmetric}.

In this paper, we provide results which precisely characterize for any symmetric, convex $f_p$ the behavior of the estimators \eqref{eqdef:finite-p-prox} and \eqref{lm-cvx-estimator} in the models \eqref{gaussian-sequence-model} and \eqref{linear-model}, respectively.
Our results have several consequences on the design and potential use of such estimators, which we also describe.
We begin by summarizing some of these results.


\subsection{Penalized least squares in the sequence model}\label{sec:penalized-ls-in-seq-model}
Consider that $f_p$ is both symmetric and separable; that is,
\begin{equation}\label{separable-functions}
    f_p(\bx) = \sum_{j=1}^p \rho(x_j),
\end{equation}
where $\rho:\reals \rightarrow \bar \reals$ is convex. 
Then the proximal operator \eqref{eqdef:finite-p-prox} satisfies
\begin{equation}\label{separable-prox}
    \mathsf{prox}[f_p](\by)_j = \mathsf{prox}[\rho](y_j),
\end{equation}
for all $j$.
Thus, observation $i$ does not affect estimate $j$ for $i \neq j$.
This separability allows us to study statistical properties of the proximal operator in the sequence model by studying statistical properties of the scalar proximal operator $\mathsf{prox}[\rho]$ in the model $y = \theta + \tau z$ where $\theta \in \reals$ and $z \sim \mathsf{N}(0,1)$.

A main insight of this paper is that in high-dimensions, all symmetric proximal operators are ``approximately separable'' in a sense which we make precise in Section \ref{sec:main-results}.
As a demonstration, consider the penalty 
\begin{align}
    f_p(\bx)&= \frac12 \min_{\eta \in \reals_+^p} \sum_{j=1}^p \left(\frac{w_j^2}{\eta_j} + \lambda_j \eta_{(j)}\right),\label{sowl}
\end{align}
where $\lambda_1 \geq \lambda_2 \geq \cdots \geq \lambda_p$ and $\eta_{(j)}$ denotes the $j^\text{th}$ decreasing order statistic of $\eta$.
This penalty is convex, symmetric, and non-separable.
It is a member of a large class of convex relaxations to submodular combinatorial penalties that are studied in detail by \cite{obozinskiBach2012}.
It has been referred to as a \emph{smoothed ordered weighed $\ell_1$ (OWL) norm} by \cite{sankaranBachBhattacharya2017}.
One derivation of \eqref{sowl} is as the tightest positively homogeneous convex lower bound of the combinatorial penalty
\begin{equation}\label{submodular-combinatorial-penalty}
    \bx \mapsto \frac12 \left(F(|\mathsf{supp}(\bx)|) + \|\bx\|_2^2\right),
\end{equation}
where $F(j) = \sum_{i=1}^j \lambda_j$ is a sub-modular function of the number of non-zero coordinates of $\bx$ \cite[Corollary1, Lemma 8]{obozinskiBach2012}.
The penalty \eqref{sowl} was further studied by \cite{sankaranBachBhattacharya2017} in the context of sparse linear regression with correlated designs. 
We refer the reader to \cite{obozinskiBach2012,sankaranBachBhattacharya2017} and references therein for a discussion of this penalty and its uses.

\begin{figure}
\centerline{\phantom{A}\hspace{-1cm}\includegraphics[width=0.65\textwidth]{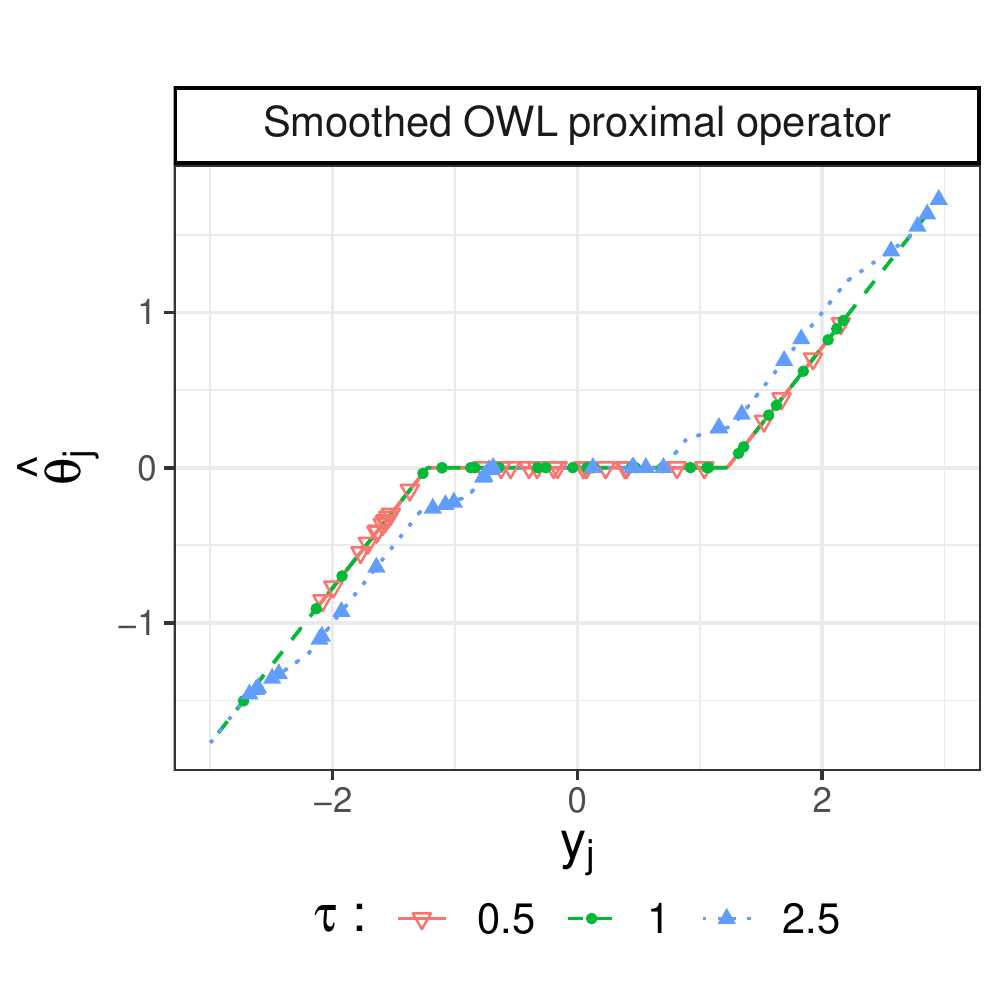}}
\caption{Plots of $\widehat \theta_j$ vs.\ $y_j$ with penalty \eqref{sowl} in model \eqref{gaussian-sequence-model} at three different noise levels $\tau = .5,1$, and $2.5$. Dimension $p = 1000$; parameter distribution $\mu_{\btheta} = \frac1{20} \mu_{-1} + \frac1{10} \mu_0 + \frac1{20} \mu_1$; regularization parameters $\lambda_j = 2$ for $j \leq 333$, $\lambda_j = 1$ for $334 \leq j \leq 667$, and $\lambda_j = .5$ for $668 \leq j \leq 1000$. 
Also shown are curves computed based on the theory developed in the paper on which $(y_j,\widehat \theta_j)$ are predicted to approximately lie.}\label{fig:smoothed-OWL-sim}
\end{figure}

In Figure \ref{fig:smoothed-OWL-sim}, we report the results of a simulation study in which we applied the proximal operator for penalty \eqref{sowl} in the sequence model \eqref{gaussian-sequence-model} when the dimension $p = 1000$; the parameter $\btheta$ has 50 coordinates equal to $-1$, 50 coordinates equal to 1, and 900 coordinates equal to 0; $\lambda_j = 2$ for $j \leq 333$, $\lambda_j = 1$ for $334 \leq j \leq 667$, and $\lambda_j = .5$ for $668 \leq j \leq 1000$; and $\tau$ is either $.5,1$, or $2.5$. 
For each $\tau = .5,1,2.5$, we randomly select some of the indices $j$ and display $(y_j,\widehat \theta_j)$ using red open triangles, green circles, and blue filled triangles, respectively.
For each $\tau$, we also plot a curve which we have computed before the realization of the noise and on which our theory predicts all such pairs $(y_j,\widehat \theta_j)$ should approximately lie.
We see that for a fixed value of $\tau$, each pair $(y_j,\widehat \theta_j)$ does indeed approximately lie on the corresponding pre-computed curve. 

These observations suggest --and we will show-- that conditional on \emph{(i)} the empirical parameter distribution and \emph{(ii)} the noise level, the estimate $\widehat \theta_j$ is approximately only a function of $y_j$. This function does not depend upon $j$ and can be determined prior to the realization of the noise.
Moreover, we will show that this function can be written as $\mathsf{prox}[\rho]$ for some scalar convex function $\rho$.
In this sense, for the purposes of estimation in the Gaussian sequence model with symmetrically penalized least squares, all penalties behave like separable penalties in high-dimensions.
For statistical purposes, they are ``approximately separable.''
We make these results precise with a finite sample concentration inequality which holds uniformly over choices of symmetric penalty $f_p$.
These results suggest that there are limited advantages to non-separability in the sequence model when the empirical parameter distribution is approximately known.

A recent paper established this phenomenon for a particular class of symmetric penalties: the \emph{ordered weighted $\ell_1$ (OWL) norms}, which induce an estimation scheme also refered to as \emph{sorted $\ell_1$ penalized estimation (SLOPE)} \cite{huLu2019}. 
These penalties have been proposed for the purposes of adaptation to sparsity and greater stability in sparse estimation with highly correlated designs \cite{Bogdan2015SLOPE---AdaptiveOptimization,Su2016SLOPEMinimax,sankaranBachBhattacharya2017}. 
They take the form
\begin{align}\label{slope}
    f_p(\bx) = \sum_{j=1}^p \lambda_j |x|_{(j)}
\end{align}
where $\lambda_1 \geq \cdots \geq \lambda_p \geq 0$ are appropriately chosen regularization parameters and $|x|_{(j)}$ are the decreasing absolute order statistics of $\bx$.
We establish approximate separability of symmetric proximal operators much more generally because our results hold for any symmetric penalty.
For example, the theorems in  \cite{huLu2019} do not address the approximate separability of the penalty \eqref{sowl}.

\subsection{Penalized least squares in the linear model with Gaussian designs}

In the past several years, there have been many works characterizing the distribution of M-estimators in high-dimensional linear models with Gaussian design matrices (see e.g.\ \cite{BayatiMontanariLASSO,ElKaroui2013OnPredictors.,thrampoulidis2015regularized,thrampoulidis2018precise,Donoho2016HighPassing,miolane2018distribution}).
To facilitate the discussion to come, we present a recent result of this type for the LASSO.

\begin{theorem}[Adapted from Theorem 3.1 of \cite{miolane2018distribution}]\label{thm:miolane-montanari}
    Consider model \eqref{linear-model} with $\bw \sim \mathsf{N}(\bzero,\sigma^2\bI_n)$ for $\sigma > 0$. Let $f_p(\bx) = \xi \|\bx\|_1$.
    Then there exist $\tau,\lambda,C,c$ depending only on the empirical distribution of the parameters $\mu_{\btheta}:= \frac1p\sum_{j=1}^p \delta_{\theta_j}$, the measurement rate $n/p$, the noise variance $\sigma^2$, and the regularization parameter $\xi$ such that for all $\epsilon \in (0,1/2)$
    \begin{equation}\label{miolane-montanari-concentration}
        \P\left(W_2(\mu_{(\widehat \btheta,\btheta)},\mu_*)^2 \geq \epsilon\right)\leq C\epsilon^{-3}\exp\left(-cp\epsilon^3\log(1/\epsilon)^{-2}\right),
    \end{equation}
    where $\mu_{(\widehat \btheta,\btheta)} = \frac1p \sum_{j=1}^p \delta_{(\widehat \theta_j,\theta_j)}$ is the empirical joint distribution of the true parameters and their corresponding estimates; 
    $\mu^*$ is the joint distribution of $(\eta_{\mathrm{soft}}(\theta + \tau z;\lambda\xi),\theta)$ when $\theta \sim \mu_{\btheta}$, $z \sim \mathsf{N}(0,1)$ independent of $\theta$, and $\eta_{\mathrm{soft}}$ is soft-thresholding with threshold $\lambda \xi$;\footnote{That is, $\eta_{\mathrm{soft}}(y;\lambda\xi) = \max(\min(y+\lambda\xi,0),y-\lambda\xi)$.}
     and $W_2$ is the Wasserstein distance of order 2 on the space of probability distributions with finite second moment (which we define in Section \ref{sec:notations}).
\end{theorem}

\noindent Coordinate-wise soft-thresholding with threshold $\lambda \xi$ is the proximal operator of the separable penalty $f_p(\bx) = \lambda \xi \|\bx\|_1$. 
Thus, the concentration inequality of Theorem \ref{thm:miolane-montanari} says that the estimate $\widehat \btheta$ of parameter $\btheta$ behaves as if it resulted from applying a separable proximal operator to observations in the Gaussian sequence model \eqref{gaussian-sequence-model} at a particular noise variance $\tau^2$.
Once we understand the dependence of $\tau,\lambda$ on the parameters of the problem (which \cite{miolane2018distribution} describes), Theorem \ref{thm:miolane-montanari} reduces the study of the LASSO in the linear model with Gaussian designs to the study of soft-thresholding in the scalar statistical model $y = \theta + \tau z$ with $\theta \sim \mu_{\btheta}$ and $z \sim \mathsf{N}(0,1)$ independent of $\theta$.
For example, Theorem \ref{thm:miolane-montanari} implies that the realized $\ell_2$-loss of the LASSO estimate, $\frac1p \|\widehat \btheta - \btheta\|^2$, concentrates on the $\ell_2$-risk of soft-threshdolding in the scalar model, $\E_{\theta \sim \mu_{\btheta},z \sim \mathsf{N}(0,1)}[(\eta_{\mathsf{soft}}(\theta + \tau z;\lambda \xi)-\theta)^2]$.

In this paper, we show that a concentration inequality of the form \eqref{miolane-montanari-concentration} holds for any symmetric penalty $f_p$.
In particular, the estimator \eqref{lm-cvx-estimator} in the linear model \eqref{linear-model} behaves as if it resulted from applying a certain separable proximal operator to observations in the Gaussian sequence model at a particular noise variance.
One consequence of this result is that for a fixed empirical parameter distribution $\mu_{\btheta}$, aspect ratio $n/p$, noise variance $\sigma^2$, and symmetric penalty $f_p$, there exists a separable penalty which produces estimates in model \eqref{linear-model} with nearly the same statistical properties.

\subsection{Adaptive estimation}

The preceding two sections suggest that in the models \eqref{gaussian-sequence-model} and \eqref{linear-model} and for a fixed empirical parameter distribution, non-separable symmetrically penalized least-squares behaves like  separable symmetrically penalized least-squares with appropriately chosen penalty.
What advantage, then, do non-separable penalties have for estimation?

We propose that one potential advantage of non-separability emerges when the empirical distribution of the parameters is unknown. 
Figure \ref{fig:smoothed-OWL-sim} suggests, and we show, that in the sequence model, non-separable symmetrically penalized least squares approximately applies to each coordinate an estimator which depends upon the underlying empirical parameter distribution.
In this sense, symmetric proximal operators in the sequence model with unknown empirical distribution of the parameter automatically implement  adaptive procedures.
In particular, they select in a data-driven way a scalar estimator to (approximately) apply coordinate-wise.
For any given penalty $f_p$, we characterize this adaptive procedure. Moreover, we provide a sufficient condition which given an adaptive procedure guarantees it is implemented by some symmetric, convex penalty.

Ideally, we could design a non-separable penalty which, for each underlying parameter distribution of interest, chose the ``right'' coordinate-wise estimator according to some criterion. 
While we do not explore the design of such penalties in this paper, we hope this work opens the door to principled approaches to designing non-separable M-estimators like the square-root LASSO and SLOPE \cite{belloniChernozhukovWang2011,Bogdan2015SLOPE---AdaptiveOptimization} which adapt to a pre-specified nuisance parameter.

\subsection{Summary of contributions}

To summarize and add to the preceding three sections, we list our primary contributions.

\begin{enumerate}
    
    \item 
    We prove a finite sample concentration inequality which quantifies in a precise sense the ``asymptotic separability'' of symmetric proximal operators in the Gaussian sequence model. 
    This concentration inequality holds uniformly over all choices of penalty.
    This result is substantially more general and more precise than a similar result of \cite{huLu2019}: more general because it applies to arbitrary symmetric penalties rather than only to penalties of the form \eqref{slope}; more precise because it holds in finite samples rather than asymptotically.

    \item 
    We prove a finite sample concentration inequality which --given a solution to a certain system of equations involving model and estimator parameters-- characterizes in a precise sense the behavior of symmetrically penalized least squares in the linear model.
    The characterization involves comparison to a particular scalar estimation model. 
    For the same reasons as above, this result is substantially more general and more precise than similar results in \cite{huLu2019} which apply only to SLOPE and only asymptotically.
    Moreover, for each fixed empirical parameter distribution, aspect ratio $n/p$, noise variance $\sigma^2$, and symmetric penalty $f_p$, our results establish the same concentration inequality (with the same constants) for penalized least squares with a particular separable penalty. 
    Thus,  we reveal a near equivalence in behavior between the symmetrically and the separably penalized estimators (contingent on the existence of solutions to a certain system of equations). 

    \item 
    A recent paper established a lower bound on the asymptotic squared error of symmetrically penalized least squares in the linear model with Gaussian design matrices with independent entries  \cite{celentanoMontanari2019}. The arguments of that paper easily imply an analagous lower bound for separably penalized least squares.
    We show that the two lower bounds agree, which is significant because the lower bounds are expected to be generally tight. 
    This and contribution 2 provide strong evidence that, in high dimensions, symmetrically penalized least squares cannot outperform separably penalized least squares in the linear model with Gaussian designs when the empirical parameter distribution is approximately known. 

    \item 
    We argue that symmetric proximal operators in the sequence model with unknown empirical parameter distribution automatically implement adaptive procedures.
    We describe these adaptive procedures and provide a sufficient condition which given an adaptive procedure guarantees it is implemented by some symmetric, convex penalty.

    \item 
    We develop a theory of symmetric, convex penalties and proximal operators based on the tools of optimal transport theory. 
    This theory is the basis of the contributions listed above.
    In the sequence model, this theory can generate results which generalize contribution 1 beyond the case of Gaussian noise. 
    While the results we present hold with high-probability, the framework we develop can easily generate results which hold in other senses, like in expectation.

\end{enumerate}

\subsection{Related literature}\label{sec:related-lit}

Several recent works have developed precise characterizations of M-estimators in linear models with Gaussian designs. 
One approach, and the one we adopt, uses Gaussian comparison inequalities. 
This approach was first developed in \cite{stojnic2013framework} and was developed further by \cite{thrampoulidis2015regularized,thrampoulidis2018precise}. 
Finite-sample concentration inequalities for the LASSO using this technique were developed by \cite{miolane2018distribution}. 
Sharp asymptotics for SLOPE in linear regression were provided by similar means in \cite{huLu2019}. 

An alternative approach to the precise characterization of M-estimators in high-dimensional linear models uses Approximate Message Passing (AMP) algorithms \cite{BayatiMontanariLASSO,Donoho2016HighPassing,sur2018modern}.
The precise characterization of AMP algorithms via the so-called \emph{state evolution} was first developed in \cite{bolthausen2014iterative,BM-MPCS-2011,javanmard2013state}.
Recent work establishes the validity of the state evolution for such algorithms which use nonseparable non-linearities \cite{Berthier2017StateFunctions}.
This work was used in \cite{celentanoMontanari2019} to develop lower bounds on the asymptotic performance of symmetrically (and possibly nonseparably) penalized least squares.

A third approach to the precise characterization of M-estimators in the high-dimensional linear model uses leave-one-out techniques.
This approach has been successfully applied to ridge regression and to schemes which penalize residuals using a non-quadratic but separable loss \cite{ElKaroui2013OnPredictors.,karoui2013asymptotic}.

The recent paper \cite{huLu2019} identified the asymptotic separability of the SLOPE penalty, and our work is a natural extension of that paper.

The adaptive potential of M-estimation has been previously observed through specific examples.
These include the square-root LASSO, which adapts to unknown noise-level by applying the loss $\|\by - \bX \btheta\|/\sqrt{n}$ to the residuals \cite{belloniChernozhukovWang2011} and SLOPE, which adapts to unknown sparsity by applying a LASSO-like penalty which penalizes parameter estimates more or less strongly based on rank statistics \cite{Bogdan2015SLOPE---AdaptiveOptimization,Su2016SLOPEMinimax}. 

Some of the results we develop regarding symmetric, convex functions are not entirely new. The papers 
\cite{horsley1988,day2973} and the references therein develop several results regarding the structure of symmetric, convex functions in Banach spaces and their subdifferentials.
We use the tools of optimal transport theory to the study of symmetric, convex penalties, which we believe is particularly fruitful in generating statistical insight into the behavior of M-estimation with symmetric penalties.

\subsection{Notations}\label{sec:notations}

For a vector $\bx \in \reals^p$, we denote by $\mu_{\bx}$ the empirical distribution of its coordinates; that is, $\mu_{\bx} = \frac1p \sum_{j=1}^p \delta_{x_j}$. 
Similarly, for a random variable $X$, we denote by $\mu_X$ its law.
We denote by $\bar \reals = \reals \cup\{\infty\}$ the real line completed at $\infty$. 
For two measures $\mu,\mu' \in \cP_2(\reals)$, we denote by $\Pi(\mu,\mu')$ the collection of couplings between $\mu$ and $\mu'$.
That is, $\nu$ a probability measure on $\reals^2$ is in  $\Pi(\mu,\mu')$ if its first marginal is $\mu$ and its second marginal is $\mu'$.
The set $\Pi(\mu,\mu')$ is non-empty because it contains, in particular, the product meausure $\mu \otimes \mu'$.
It is well known that
\begin{equation}
    W_2(\mu,\mu')^2 := \inf_{\pi \in \Pi(\mu,\mu')} \E_{(X,X') \sim \pi} \E[(X-X')^2]
\end{equation}
defines a metric $W_2(\mu,\mu')$ on the space of probability measures with bounded second moment \cite[Definition 6.4]{Villani2008}.
We denote by $\cP_2(\reals)$ the space of probability measures with finite second moment endowed with this metric. 
The space $\cP_2(\reals)$ is refered to as the Wasserstein space of order 2 on $\reals$. 
The optimal coupling between $\mu$ and $\mu'$ is denoted $\pi_{\mathrm{opt}}(\mu,\mu')$. 
For any $\mu \in \cP_2(\reals)$ and $\tau \geq 0$, we denote by $\mu^{*\tau} = \mu * \normal(0,\tau^2)$, that is, the distribution of $\Theta + \tau Z$ when $\Theta \sim \mu,\, Z \sim \normal(0,1)$ independent.
For a vector $\bx \in \reals^p$, we denote $\|\bx\| = \sqrt{\sum_{j=1}^px_j^2}$ the standard Euclidean norm.
We denote the space of square-integrable, Borel-measureable random variables on the unit interval with Lebesgue measure by $L_2(0,1)$. 
For $X,Y \in L_2(0,1)$, we denote $d_2(X,Y) = \sqrt{\E[(X-Y)^2]}$ the standard Hilbert-space metric, and $\|X\|_2 = \sqrt{\E[X^2]}$ the standard Hilbert-space norm.
The space $\cP_2(\reals^k)$, the Wasserstein space of order 2 on $\reals^k$, is defined similarly to $\cP_2(\reals)$, replacing probability measures with finite second moment on $\reals$ with those with finite second moment on $\reals^k$ and replacing squared distance $\E[(X-X')^2]$ with $\E[\|\bX - \bX'\|^2]$.

\subsection{Organization}

In Section \ref{sec:main-results}, we present and discuss our main results.
In Section \ref{sec:symm-fnc-opt-transport}, we develop a theory which uses optimal transport to study symmetric convex functions. 
In Section \ref{sec:proofs}, we prove our main results using the theory developed in Section \ref{sec:symm-fnc-opt-transport}. 
In fact, in Section \ref{sec:proofs} we  prove slightly more general versions of the theorems which appear in Section \ref{sec:main-results}.
We defer these more general statements to Section \ref{sec:proofs} because their meaning and statistical relevance are less readily apparent.
Nevertheless, they may serve as a more appropriate starting point for further extensions or applications to alternative models. 
Some technical details and additional simulations we defer to the appendices.

\section{Main results}\label{sec:main-results}

\subsection{Limited advantages to non-separability in the Gaussian sequence model}

In this section, we argue that estimators of the form \eqref{eqdef:finite-p-prox} with symmetric, convex $f_p$ in the sequence model \eqref{gaussian-sequence-model} behave as if they were defined using separable $f_p$.
By \eqref{separable-prox}, M-estimation with symmetric, separable penalties in the sequence model constructs an estimate by applying to each coordinate the same scalar estimator from the collection
\begin{equation}\label{scalar-prox-collection}
    \mathcal{PR}_1 := \{\mathsf{prox}[\rho]\mid \rho:\reals \rightarrow \bar\reals\text{ lsc, proper, convex}\}.
\end{equation}
This collection does not contain all scalar functions or even all scalar, non-decrasing functions.
\begin{fact}\label{fact:scalar-prox-characterization}
    The collection $\cPR_1 $ contains exactly those functions $\reals \rightarrow \reals$ which are non-decreasing and $1$-Lipschitz.
\end{fact}

\noindent Fact \ref{fact:scalar-prox-characterization} is proved in Appendix \ref{app:proof-of-fact-scalar-prox-characterization}. 
Our main result for estimation in the Gaussian sequence model states that all symmetric proximal operators construct an estimate by approximately applying to each coordinate the same scalar estimator from the collection $\cPR_1$.
\begin{theorem}\label{thm:symm-like-sep}
    There exist \emph{(i)} universal functions $\mathsf{c},\mathsf{C}:\reals_{>0} \rightarrow \reals_{>0}$, non-increasing and non-decreasing respectively, and \emph{(ii)} for each $p$, a collection of mappings $\frA_{f_p}:\cP_2(\reals) \rightarrow \cPR_1$ indexed by lsc, proper, symmetric, convex functions $f_p: \reals^p \rightarrow \bar\reals$ such that the following is true. 

    For any $p$, $\btheta \in \reals^p$, $\tau > 0$, and $\epsilon \in (0,1/2]$, we have
    \begin{align}
        &\P\left(\frac1p\|\mathsf{prox}[f_p](\btheta + \tau \bz) - \frA_{f_p}(\mu^{*\tau})(\btheta + \tau \bz)\|^2 > 4(W_2(\mu,\mu_{\btheta}) + \tau \sqrt\epsilon)^2\right)\nonumber\\
        &\qquad\qquad\qquad\qquad\qquad\qquad \qquad \qquad \qquad \qquad \qquad \leq C\epsilon^{-1} \exp(-cp\epsilon^3\log(\epsilon)^{-2}),\label{prox-to-scalar-concentration}
    \end{align}
    where $\mu^{*\tau} := \mu * \normal(0,\tau^2)$, $C = \mathsf{C}(\mathsf{snr})$ and $c = \mathsf{c}(\mathsf{snr})$ with $\mathsf{snr} := \frac{\|\btheta\|^2}{p\tau^2}$, the probability is over $\bz \sim \normal(0,I_p)$, and it is understood that $\frA_{f_p}(\mu^{*\tau})$ is applied to $\btheta + \tau \bz$ coordinate-wise. 
    If $\tau = 0$, we may replace the upper bound on the right-hand side by 0. 

    If $f_p(\bx)$ is separable as in \eqref{separable-functions}, we may take $\frA_{f_p}(\mu) = \mathsf{prox}[\rho]$ for all $\mu$.
\end{theorem}

\noindent Theorem \ref{thm:symm-like-sep} is proved in Section \ref{sec:proof-of-thm-symm-like-sep}. 
A detailed description of how to construct such a collection of mappings $\frA_{\cdot}$ is deferred to Section \ref{sec:symm-fnc-opt-transport}, and in particular, Section \ref{sec:finite-d-pen}.
Importantly, the scalar estimator $\frA_{f_p}(\mu^{*\tau})$ appearing in Theorem \ref{thm:symm-like-sep} does not depend upon the realization of the noise $\bz$. 
Rather, it depends only on the distribution $\mu^{*\tau}$, so that it can in principle be determined by the statistician in advance of any observations.
Moreover, it can often be efficiently computed, as we briefly discuss in Section \ref{app:simulations}.

We should think of the $\mu$ in Theorem \ref{thm:symm-like-sep} as a good approximation to the empirical distribution of the coordinates of $\btheta$. 
For example, we may take $\mu = \mu_{\btheta}$, in which case \eqref{prox-to-scalar-concentration} simplifies to 
\begin{equation}\label{prox-to-scalar-concentration-no-mu-approx}
    \P\left(\frac1p\|\mathsf{prox}[f](\btheta + \tau \bz) -  \frA_{f_p}(\mu_{\btheta}^{*\tau})(\btheta + \tau \bz)\|^2 > \epsilon \right) \leq C\epsilon^{-1} \exp(-cp\epsilon^3\log(\epsilon)^{-2}).
\end{equation}
Keep in mind that \eqref{prox-to-scalar-concentration-no-mu-approx} differs from \eqref{prox-to-scalar-concentration} not only in the threshold which appears on the right-hand side of the inequality inside the probability, but also potentially in the scalar estimator: $ \frA_{f_p}(\mu_{\btheta}^{*\tau})$ rather than $\frA_{f_p}(\mu^{*\tau})$.

We state the more general inequality \eqref{prox-to-scalar-concentration}, which applies to choices other than $\mu = \mu_{\btheta}$, because we may not always wish to apply Theorem \ref{thm:symm-like-sep} with this particular choice. 
For example, if we allow that $\btheta$ be random with coordinates drawn iid from $\mu$, then we may wish to take $\mu$ as the population rather than empiricial distribution of the cordinates of $\btheta$.
With this choice,  we may take any $M > 0$ and view Eq.~\eqref{prox-to-scalar-concentration} as a bounded on the conditional probability conditioned on $\|\btheta\|^2/p \leq M$ with $C = \mathsf{C}(M/\tau^2)$ and $c = \mathsf{c}(M/\tau^2)$, with $W_2(\mu,\mu_{\btheta})$ now also random.
Here we have used critically the monotonicity of $\mathsf{c},\mathsf{C}$. 
For $M$ larger than the second moment of $\mu$, the event $\|\btheta\|^2/p \leq M$ will hold with high probability.
By quantifying this probability, the bound \eqref{prox-to-scalar-concentration}, which in this context holds only conditionally, can lead to an unconditional bound.
We can often also control the probability that $W_2(\mu,\mu_{\btheta})$ is large (see e.g.~\cite{Fournier2015}), so that with some work  we can control the probability that $\frac1p\|\mathsf{prox}[f](\btheta + \tau \bz) -  \frA_{f_p}(\mu^{*\tau})(\btheta + \tau \bz)\|^2$ exceeds a parameter-independent threshold.
Importantly, in this case the scalar estimator $ \frA_{f_p}(\mu^{*\tau})$ does not depend upon the realization of $\btheta$ or $\bz$.
The distribution $\mu^{*\tau} = \mu * \normal(0,\tau^2)$ is interpreted as the \emph{population measurement distribution}.

Theorem \ref{thm:symm-like-sep} says that for the purposes of estimation in the Gaussian sequence model \eqref{gaussian-sequence-model} with $\mu_{\btheta}$ approximately known, non-separable, symmetric M-estimation behaves almost equivalently to separable, symmetric M-estimation with appropriately chosen convex penalty $\rho:\reals \rightarrow \bar \reals$.
Importantly, because the functions $\mathsf{c},\mathsf{C}$ are universal, the rate of concentration we establish is uniform over choices of penalty $f_p$.
That is, separable M-estimation approximates non-separable M-estimation uniformly well over such choices.

We remark that Theorem \ref{thm:symm-like-sep} follows from a more general theorem whose statement can be found in Section \ref{sec:proof-of-thm-symm-like-sep}.
For random $\btheta$, this more general theorem may yield tighter results than the approach outlined in the paragraph above.
Moreover, this more general theorem is not specific to the model \eqref{gaussian-sequence-model}.
From it, we can derive results analogous to Theorem \ref{thm:symm-like-sep} in different statistical models on $\reals^p$.
As we will see, the concentration we establish will always be uniform over the choice of $f_p$.
In fact, it will only depend upon the rate of concentration of the empirical distribution of observations in Wasserstein space, a property of the statistical model and not the estimator we choose.
Thus, Theorem \ref{thm:symm-like-sep} and its more general statement in Section \ref{sec:proof-of-thm-symm-like-sep} separate the analysis of the penalty, used to determine $\frA_{f_p}$, and the analysis of the statistical model, used to determine the rate of concentration.

\subsection{No first-order asymptotic advantage to non-separability in the linear model}

A recent paper \cite{celentanoMontanari2019} establishes an asymptotic lower-bound on the $\ell_2$-risk of the estimator \eqref{lm-cvx-estimator} in a certain high-dimensional limit in which the empirical distribution of the coordinates of $\btheta$ appropriately converges. 
Here we prove that the lower bound established there over the collection of symmetric penalties agrees with the corresponding lower bound over the much smaller collection of separable penalties.
Thus, we confirm a conjecture stated in a footnote and in Appendix Q of \cite{celentanoMontanari2019}.
This equivalence is significant because these lower bounds are expected to be generally tight.

First we describe the lower-bound of \cite{celentanoMontanari2019}.\footnote{The formulas differ slightly here because we adopt an slightly different convention of normalization.}
Fix prior $\mu \in \cP_2(\reals)$.
Consider a sequence $p \rightarrow \infty$.
Let 
\begin{equation}
    \cC = \left\{\{f_p:\reals^p \rightarrow \bar \reals\}: f_p\text{ is lsc, proper, symmetric, and convex $\forall p$}\right\}
\end{equation}   
be the collection of all sequences of lsc, proper, symmetric, convex functions.
Define the optimal per-coordinate $\ell_2$-risk of symmetric, penalized least squares in the Gaussian sequence model by
\begin{equation}\label{R-symm-def}
    \mathsf{R_{symm}^{opt}}(\tau;\mu) = \inf_{\{f_p\} \in \cC} \liminf_{p \rightarrow \infty}\frac1p\E\left[\|\mathsf{prox}[f_p](\btheta + \tau \bz) - \btheta\|^2\right],
\end{equation}
where $\theta_j \stackrel{\mathrm{iid}}\sim \mu$ and $\bz \sim \normal(0,I_p)$ independent of $\btheta$.
Let $\sigma^2 \geq 0$.
Define 
\begin{equation}
    \tau_{\mathsf{symm}}^2 = \sup\left\{\tau^2 \Biggm\vert \delta(\tau^2 - \sigma^2) < \mathsf{R_{symm}^{opt}}(\tau;\mu)\right\}.
\end{equation}
The lower-bound is given in the following proposition, which we copy from \cite{celentanoMontanari2019}.
\begin{theorem}[Theorem 1 of \cite{celentanoMontanari2019}]\label{thm:cvx-lower-bound}
    Consider a sequence of models \eqref{linear-model} with $n,p\rightarrow \infty$, $n/p \rightarrow \delta \in (0,\infty)$. 
    Assume that $\bw$ is independent of $\bX$ and almost surely
    \begin{equation}\label{signal-noise-convergence}
        \mu_{\btheta} \stackrel{\mathrm{W}}\rightarrow \mu,\qquad \frac1n\|\bw\|^2 \rightarrow \sigma^2,
    \end{equation}
    (in particular, we consider both models in which $\btheta,\bw$ are random and models in which they are deterministic).
    If we adopt the convention that $\|\widehat \btheta - \btheta\|^2$ is infinite whenever the minimizing set in \eqref{lm-cvx-estimator} is empty, then
    \begin{equation}\label{cvx-lower-bound}
        \inf_{\{f_p\} \in \cC_{\delta,\mu}} \liminf_{p\rightarrow\infty}^\mathrm{p} \frac1p \|\widehat \btheta - \btheta \|^2 \geq \delta(\tau_{\mathsf{symm}}^2 - \sigma^2),
    \end{equation}
    where a sequence of lsc, proper, symmetric, convex functions $\{f_p\}$ is in $\cC_{\delta,\mu}$ if 
    \begin{equation}\label{delta-bounded-width}
        \begin{gathered}
            \text{for all compact $T\subset (0,\infty)$, there exists $\bar \lambda = \bar\lambda(T) < \infty$ such that}\\
            \limsup_{p \rightarrow \infty} \sup_{\lambda > \bar \lambda,\tau \in T} \frac1{\tau p} \E\left[\langle \bz , \mathsf{prox}[\lambda f_p](\btheta + \tau \bz)\right],
        \end{gathered}
    \end{equation}
    where in the expectation we take $\theta_j \stackrel{\mathrm{iid}}\sim \mu$ and $\bz \sim \normal(0,I_p)$ independent of $\btheta$.
\end{theorem}

\noindent Theorem \ref{thm:cvx-lower-bound} establishes a lower bound on the realized $\ell_2$-loss of symmetrically penalized least-squares asymptotically for sequences of penalties belonging to the collection $\cC_{\delta,\mu}$. 
Using Fatou's lemma, it is straightforward to extend the lower bound \eqref{cvx-lower-bound} to a lower-bound on the asymptotic risk of the estimator \eqref{lm-cvx-estimator} (that is, where we take an expectation in \eqref{cvx-lower-bound}. See \cite[Lemma I.1]{celentanoMontanari2019}).
The collection $\cC_{\delta,\mu}$ is extensively discussed in \cite{celentanoMontanari2019}, see for example Section 3 and Appendix C in that paper.
It is argued there that the condition \eqref{delta-bounded-width}, though difficult to understand, is not very restrictive.
The authors of \cite{celentanoMontanari2019} refer to sequences which satisfy \eqref{delta-bounded-width} as sequences with $\delta$-\emph{bounded width}, a terminology which we adopt as well.

Now we consider symmetric, separable penalties.
Let $\cF_1$ be the collection of all univariate, lsc, proper, convex functions.
Define 
\begin{equation}\label{R-sep-def}
    \mathsf{R_{sep}^{opt}}(\tau;\mu) = \inf_{f_1 \in \cF_1}\E\left[(\mathsf{prox}[f_1](\theta + \tau z) - \theta)^2\right],
\end{equation}
where $\theta \sim \mu$, $z \sim \normal(0,1)$ independent, and the infimum is taken over all lsc, proper, convex functions.
Define 
\begin{equation}
    \tau_{\mathsf{sep}}^2 = \sup\left\{\tau^2 \Biggm\vert \delta(\tau^2 - \sigma^2) < \mathsf{R_{sep}^{opt}}(\tau;\mu)\right\}.
\end{equation}
We claim that under the conditions of Theorem \ref{thm:cvx-lower-bound}, we have that 
\begin{equation}\label{cvx-sep-lower-bound}
    \inf_{\substack{f_1 \in \cF_1\\ \left\{f_p(\bx)= \sum_{j=1}^p f_1(x_j)\right\} \in \cC_{\delta,\mu} } } \liminf_{p\rightarrow\infty}^\mathrm{p} \frac1p \|\widehat \btheta - \btheta \|^2 \geq \delta(\tau_{\mathsf{sep}}^2 - \sigma^2),
\end{equation}
where the reader should have in mind that for each $p$ we make the same choice $f_1$, though this is not required.
Indeed, it is straightforward (though perhaps tedious) to check that the proofs in \cite{celentanoMontanari2019} all go through if we instead consider estimation using separable penalties in the class $\cC_{\delta,\mu}$ with the lower bound \eqref{cvx-sep-lower-bound} (which includes, for example, all $f_1$ with unique minimizer \cite[Claim 3.5]{celentanoMontanari2019}).

Our next result establishes that the right-hand sides of \eqref{cvx-lower-bound} and \eqref{cvx-sep-lower-bound} agree.
\begin{theorem}\label{thm:symm-lb-is-sep-lb}
    Let $\mu \in \cP_2(\reals)$ and $\tau \geq 0$.
    Then
    \begin{equation}
        \mathsf{R_{symm}^{opt}}(\tau;\mu) = \mathsf{R_{sep}^{opt}}(\tau;\mu).
    \end{equation}
    In particular, the lower-bound \eqref{cvx-lower-bound} on the optimal $\ell_2$-loss over sequences of symmetric penalties with $\delta$-bounded width agrees with the lower bound \eqref{cvx-sep-lower-bound} over sequences of symmetric, separable penalties with $\delta$-bounded width.
\end{theorem}
\noindent Theorem \ref{thm:symm-lb-is-sep-lb} is proved in Section \ref{sec:proof-of-thm-symm-lb-is-sep-lb}.
Theorem \ref{thm:symm-lb-is-sep-lb} should not be surprising in light of Theorem \ref{thm:symm-like-sep}. 
Indeed, $\mathsf{R_{symm}^{opt}}(\tau;\mu)$ and $\mathsf{R_{sep}^{opt}}(\tau;\mu)$ are defined in terms of the performance of M-estimation in the Gaussian sequence model.
Theorem \ref{thm:symm-lb-is-sep-lb} suggests that asymptotically there is no advantage to non-separability in symmetrically penalized least squares for model \eqref{linear-model}.
Confirming this claim requires establishing that the lower-bounds \eqref{cvx-lower-bound} and \eqref{cvx-sep-lower-bound} are tight. 
While we believe this to be the case, we do not prove so here.
In the following section we present further evidence of this fact.
In particular, given a solution to a certain system of equations, we establish a finite-sample concentration inequality which characterizes the behavior of symmetrically penalized least-squares in the model \eqref{linear-model} which also holds for a particular choice of separable, symmetric penalty.

Finally, we remark that Theorem \ref{thm:symm-lb-is-sep-lb} allows us to compute $\mathsf{R_{symm}^{opt}}(\tau;\mu)$ and the symmetric lower bound by computing $\mathsf{R_{sep}^{opt}}(\tau;\mu)$ and the separable lower bound.
Because the quantities appearing in \eqref{R-sep-def} involve only functions of a scalar random variables, they can be computed efficiently numerically (see \cite[Appendix Q]{celentanoMontanari2019}), whereas the quantities in \eqref{R-symm-def} may not be.
The theoretical lower bounds presented in Table 1 and Figure 1 of \cite{celentanoMontanari2019} are computed via \eqref{R-sep-def}, and Theorem \ref{thm:cvx-lower-bound} rigorously justifies this computation.

\subsection{Finite-sample behavior of symmetric M-estimators in the linear model}\label{sec:finite-sample-linear-model}

We now present a concentration inequality analogous to that in Theorem \ref{thm:symm-like-sep} except applied to the linear model \eqref{linear-model}.
Our concentration inequality in the linear model uses a solution to a certain system of equations involving scalar estimators (and in particular, assumes such a solution exists).

\begin{theorem}\label{thm:finite-sample-lm-concentration}
    There exist \emph{(i)} universal functions $\mathsf{c_1},\mathsf{c_2},\mathsf{c_3},\mathsf{C}:\reals_{>0}^4 \rightarrow \reals_{>0}$ and \emph{(ii)} for each $p$, a collection of mappings $\frA_{f_p'}:\cP_2(\reals) \rightarrow \cPR_1$ indexed by lsc, proper, symmetric, convex functions $f_p': \reals^p \rightarrow \bar\reals$ such that the following is true.

    Fix any lsc, proper, symmetric, convex function $f_p:\reals^p \rightarrow \bar \reals$.
    Assume $\btheta \in \reals^p$ with $\|\btheta\|^2/p \leq M$ for some $M$.  
    Consider model \eqref{linear-model} with $n/p = \delta$ and $\bw$ independent of $\bX$.
    Let $\widehat \btheta$ be any random vector satisfying \eqref{lm-cvx-estimator} almost surely on the event that the minimizing set is non-empty.
    Consider random variables $\Theta^* \sim \mu_{\btheta}$ and $G^* \sim \normal(0,1)$ independent and $\tau^*,\sigma^* > 0$. Let $Y^* = \Theta^* + \tau^* G^*$.
    Assume that $\tau^*,\lambda^*$ solve
    \begin{subequations}\label{fixed-pt-eqns}
        \begin{align}
            {\tau^*}^2 &= {\sigma^*}^2 + \frac1\delta \E\left[(\frA_{\lambda^* f_p}(\mu^{*\tau^*}) \circ Y^* - \Theta^*)^2\right],\\
            \delta &= \lambda^* \left(1 - \frac1{\delta\tau^*} \E[G^*(\frA_{\lambda^* f_p}(\mu^{*\tau^*})\circ Y^*)]\right),
        \end{align}
    \end{subequations}
    where for any function $\eta:\reals \rightarrow \reals$ and random variable $X$, the expression $\eta \circ X$ denotes the random variable constructed by applying $\eta$ to the realized value of $X$.\footnote{We adopt this notation because we will frequently consider function $f:L_2(0,1)\rightarrow \bar\reals$, and we do not wish to confuse a real-valued function of a random variables with the application of a real-valued function of a real number to the realized value of the random variable.} 
    Let $c_j = \mathsf{c_j}(M,\tau^*,\sigma^*,\delta)$, $j = 1,2,3$ and $C = \mathsf{C}(M,\tau^*,\sigma^*,\delta)$.
    Then for all $0 < \epsilon < c_1$,
    \begin{align}
        &\P\left(W_2(\widehat \mu_{(\widehat \btheta, \btheta)} , \mu_{(\frA_{\lambda^* f_p}(\mu^{*\tau^*}) \circ Y^*,\Theta^*)})^2 > {\tau^*}^2 \epsilon\right) \nonumber\\
        &\qquad\qquad \leq 2C\epsilon^{-1}\exp\left(-c_2p\epsilon^3\log(1/\epsilon)^{-2}\right) + 4 \P\left(\left|\frac{\|\bw\|/\sqrt n}{\sigma^*} - 1\right|> c_3 \epsilon\right).\label{linear-regression-concentration}
    \end{align}
   
    If $f_p'(\bx)$ is separable as in \eqref{separable-functions}, we may take $\frA_{f_p'}(\mu) = \mathsf{prox}[\rho]$ for all $\mu$.
\end{theorem}

\noindent Theorem \ref{thm:finite-sample-lm-concentration} is proved in Section \ref{sec:proof-of-thm-finite-sample-lm-concentration}. 
The dependence of $\mathsf{c_1}, \mathsf{c_2}, \mathsf{c_3}, \mathsf{C}$ on the parameters of the problem is careful tracked in its proof (and, in particular, in the part of the proof deferred to the Appendices).
The collections $\{\frA_{f_p} \mid f_p \text{ lsc, proper, convex on $\reals^p$}\}$ for which Theorem \ref{thm:symm-like-sep} and Theorem \ref{thm:finite-sample-lm-concentration} hold can be taken to be the same. 
We will describe these collections in Section \ref{sec:effective-scalar-estimators}.
Theorem \ref{thm:finite-sample-lm-concentration} supports the discussion following Theorem \ref{thm:symm-lb-is-sep-lb} that the equivalence between the symmetric and separable lower bounds is not an artefact of known proof techniques but rather is fundamental.
Indeed, the constants $c_1,c_2,c_3,C$ do not depend upon the choice of penalty $f_p$.
In particular, the same constants apply to the separable penalty \eqref{separable-functions} which takes $\rho$ such that $\mathsf{prox}[\lambda^*\rho] = \frA_{\lambda^* f_p}(\mu^{*\tau^*})$.
Theorem \ref{thm:finite-sample-lm-concentration} does not establish this equivalence completely because it requires assuming a solution to \eqref{linear-regression-concentration} exists, and it is difficult to argue that an M-estimator for which no solution exists could not perform better.
In fact, much of the technical work in proving Theorem \ref{thm:cvx-lower-bound} goes into establishing the existence of a solution to a system like \eqref{fixed-pt-eqns} \cite{celentanoMontanari2019}. 
Nevertheless, solutions to \eqref{fixed-pt-eqns} should usually exist (for example, they always exist for the LASSO, see \cite{miolane2018distribution}), and when they do not, it often results from the estimator being severely ill-defined (for example, taking $f_p = 0$ when $\delta < 1$).

\subsection{Advantages to non-separability: adaptive estimation in the sequence model}\label{sec:non-sep-adaptive}

The mapping $\frA_{f_p}$ which appears in Theorems \ref{thm:symm-like-sep} and \ref{thm:finite-sample-lm-concentration} has a natural adaptive interpretation.
Empirical Bayes methods of estimation are based upon the insight that in certain high-dimensional statistical models the prior which generated the data can be consistenly estimated from the empirical measurement distribution, whence the Bayes estimator can be approximated and applied to estimate each parameter \cite{robbins1956,brown2009,jiang2009,Efron2011TweediesBias}.\footnote{Under certain interpretations, we need not view the underlying parameters as random. Rather, the prior can be viewed as a description of the fixed empirical distribution of the truth.} Thus, Bayes performance can be approximately achieved without knowledge of the prior.
We may view such full empirical Bayes methods as selecting a coordinate-wise estimator from the collection of all measurable estimators in a data-dependent way.
We may also consider restricted empirical Bayes, in which the statistician uses the data to select an estimator from a restricted class rather than from the collection of all estimators.
In these settings, for certain underlying priors the statisticians fail to consistently select the Bayes estimator. 
James-Stein estimation and procedures which adaptively choose a soft-thresholding parameter in sparse estimation are instances of the use of such techniques. See, for example, \cite{efronMorris1973,donohoJohnstone1995,xieKouBrown2012}.  

Theorems \ref{thm:symm-like-sep} and \ref{thm:finite-sample-lm-concentration} indicate that symmetric penalties automatically implement exactly this type of program. 
For example, consider the sequence model \eqref{gaussian-sequence-model} and proximal operators \eqref{eqdef:finite-p-prox}.
For $f_p$ fixed, we should think of the mappings $\frA_{f_p}$ as a population adaptive mapping: it takes the true population measurement distribution (which is unknown) and chooses a scalar estimator to apply coordinate-wise based on that distribution.
Theorem \ref{thm:symm-like-sep} says that symmetric proximal operators approximately implements the population adaptive mapping  with errors bounded by \eqref{prox-to-scalar-concentration}.
Some such error is inevitable.
In finite samples, we cannot exactly infer the population measurement distribution, so we cannot exactly choose a scalar estimator to apply coordinate-wise based on $\frA_{f_p}$.
Whether inequality \eqref{prox-to-scalar-concentration} captures the correct or optimal error rate is not something we address.
Conveniently, however, the error in inequality \eqref{prox-to-scalar-concentration} is uniform over choices of symmetric penalty $f_p$.

The adaptive interpretation of M-estimators \eqref{lm-cvx-estimator} in the linear model \eqref{linear-model} is more complicated because the scalar estimator $\frA_{\lambda^* f_p}(\mu^{*\tau^*})$ depends upon $\lambda^*$, which is determined in a complicated manner via \eqref{fixed-pt-eqns} (which, in turn, depends upon $\frA_{\lambda^* f_p}(\mu^{*\tau^*})$).
Fixing the penalty $f_p$ and measurement rate $\delta$, the population adaptive mapping takes the parameter empirical distribution $\mu_{\btheta}$ and approximate noise level ${\sigma^*}^2$ and chooses a scalar estimator $\frA_{\lambda^*f_p}$ and noise level $\tau^*$ by finding a soultion to \eqref{fixed-pt-eqns}. 
Eq.\ \eqref{fixed-pt-eqns} involves both the collection $\{\frA_{\lambda f_p}\}_{\lambda > 0}$ and the parameters $\mu_{\btheta},\sigma^*$ to which it adapts.
Theorem \ref{thm:finite-sample-lm-concentration} says that symmetric least squares behaves approximately as if the scalar estimator $\frA_{\lambda^*f_p}$ selected by the population adaptive mapping were applied to measurements in the sequence model \eqref{gaussian-sequence-model} at noise level $\tau^*$.

The adaptive potential of symmetric M-estimation should not be surprising.
Recently, SLOPE was introduced to achieve FDR control and to adaptively achieve minimax rates of estimation over sparisty levels in both the Gaussian sequence model and the linear model \cite{Bogdan2015SLOPE---AdaptiveOptimization,Su2016SLOPEMinimax,bellec2018}. 
The current paper establishes in a much more general way the adaptive potential of symmetric M-estimation.

It is natural to ask which adaptive procedures can be implemented --exactly or approximately-- by symmetrically penalized least squares in either the Gaussian sequence model or the linear model, and whether there is a principled design process which automates the discovery of an adaptive symmetric penalty for a particular task.
In this paper, we characterize which adaptive procedures can be exactly implemented in the Gaussian sequence model and leave such a characterization in the linear model for future work.

More preciesly, consider a collection $\cD \subset \cP_2(\reals)$ of distributions with non-trivial Guassian component. 
That is, for each $\nu \in \cD$, we may write $\nu = \mu * \normal(0,\tau^2)$ for some $\mu \in \cP_2(\reals)$ and $\tau > 0$. 
We should think of $\cD$ as a collection population measurement distributions in a collection of sequence models of the form \eqref{gaussian-sequence-model}.
Consider an ideal adaptive procedure defined by a mapping $\frA:\cD \rightarrow \cPR_1$.
When does there exist a symmetric, convex $f_p$ such that we may take $\frA_{f_p}$ in Theorem \ref{thm:symm-like-sep} to agree with $\frA$ on $\cD$?
We provide an explicit characterization of the $\frA$ for which this is possible, which relies on the following notion.
\begin{definition}[Joint cyclic monotonicity]\label{def:joint-cyc-mon}
    A subset $\frR \subset \cP_2(\reals^2)$ is \emph{jointly cyclically-monotone} if for all finite $n$, all sets $\{\pi_j\}_{j=1}^n \subset \frR$, all random vectors $(X_1,G_1,\ldots,X_n,G_n)$ with $(X_j,G_j) \sim \pi_j$ for all $j$, and all permutations $\sigma:[n] \rightarrow [n]$, we have
    \begin{equation}\label{join-cycl-mono-characterization}
        \sum_{j=1}^n \E[X_jG_j] \geq \sum_{j=1}^n \E[X_j G_{\sigma(j)}].
    \end{equation}
\end{definition}
\noindent The implementability of an adaptive procedure via a symmetric proximal operator depends upon the joint cyclic monotonicity of a certain set.
\begin{theorem}\label{thm:adaptive-char}
    There exist universal functions $\mathsf{c},\mathsf{C}:\reals_{>0} \rightarrow \reals_{>0}$, non-increasing and non-decreasing respectively, such that the following is true.

    Consider $\cD \subset \cP_2(\reals)$ a collection of distributions, each with a non-trivial Gaussian component.
    Let $\frA:\cD \rightarrow \cPR_1$.
    If $\{\mu_{(\frA(\mu_Y)\circ Y,Y-\frA(\mu_Y)\circ Y)} \mid \mu_Y \in \cD\}$ is jointly cyclically monotone, then for each $p$ there exists an lsc, proper, symmetric, convex function $f_p:\reals^p \rightarrow \bar\reals$ such that for all $\btheta \in \reals^p$, all $\mu \in \cP_2(\reals)$ and $\tau > 0$ with $\mu^{*\tau} \in \cD$, and all $\epsilon \in (0, 1/2]$, we have
    \begin{align}
        &\P\left(\frac1p\|\mathsf{prox}[f_p](\btheta + \tau \bz) - \frA(\mu^{*\tau})(\btheta + \tau \bz)\|^2 > (W_2(\mu,\mu_{\btheta}) + \tau \sqrt\epsilon)^2\right)\nonumber\\
        &\qquad\qquad\qquad\qquad\qquad\qquad \qquad \qquad \qquad \qquad \qquad \leq C\epsilon^{-1} \exp(-cp\epsilon^3\log(\epsilon)^{-2}),\label{prox-to-scalar-adaptive-concentration}
    \end{align} 
    where $\mu^{*\tau} = \mu * \normal(0,\tau^2)$, $C = \mathsf{C}(\mathsf{snr})$ and $c = \mathsf{c}(\mathsf{snr})$ with $\mathsf{snr} = \frac{\|\btheta\|^2}{p\tau^2}$, the probability is over $\bz \sim \normal(\bzero,\bI_p)$, and it is understood that $\frA(\mu^{*\tau})$ is applied to $\btheta + \tau \bz$ coordinate-wise.
\end{theorem}
\noindent Theorem \ref{thm:adaptive-char} is proved in Section \ref{sec:proof-of-thm-adaptive-char}.
The condition of Theorem \ref{thm:adaptive-char} expresses the implementability of an adaptive procedure with a symmetric proximal operator through a condition that is intrinsic to the adaptive procedure itself.
Unfortunately, the condition as stated is difficult to work with. 
We help make the statement more concrete with two examples. 

\begin{example}[Non-adaptive procedures]\label{ex:bayes}
If $\cD$ is a singleton, then the condition of Theorem \ref{thm:adaptive-char} holds if and only if $\frA(\mu)$ is non-decreasing and 1-Lipschitz for the unique $\mu \in \cD$. 
Indeed, in this case the separable penalty $f_p(\bx) = \sum_{j=1}^p \rho(x_j)$ where $\mathsf{prox}[\rho] = \frA(\mu)$ implements the procedure $\frA$. 
This example is a non-adaptive example, so that its implementability by separable, symmetric M-estimation in unsurprising. 
If $\mu = \mu' * \normal(0,\tau^2)$, we can implement the Bayes estimator $\frA(\mu)(y) = \E[\theta | \theta + \tau z = y]$, where $\theta \sim \mu'$ and $z \sim \normal(0,1)$ independent of $\theta$, if and only if the Bayes estimator is 1-Lipschitz (it is non-decreasing automatically).
\end{example}

\begin{example}[Full-empirical Bayes]\label{ex:emp-bayes}
Fix $\tau > 0$ and let $\cD = \{\mu^{*\tau} \mid \mu \in \cP_2(\reals)\}$. 
Let $\frA(\mu^{*\tau})(y) = \E[\theta| \theta + \tau z  = y]$, where on the right-hand side $\theta \sim \mu'$ and $z \sim \normal(0,1)$ independent of $\theta$.
The condition of Theorem \ref{thm:adaptive-char} fails for this $\frA$.
Indeed, there exist $\mu$ such that the Bayes estimator is not 1-Lipschitz, and we cannot construct $f_p$ such that $\frA_{f_p}(\mu^{*\tau})$ is the Bayes estimator at such $\mu$ by the preceding example.
\end{example}

\noindent The reader may wonder why adaptive procedures which can be implemented via symmetric M-estimation deserve special attention.
We suggest several reasons.
First, convex M-estimators should typically be easy to compute. Thus, we may expect that identifying adaptive procedures implemented by convex M-estimation will also generate procedures which are computationally feasible.
Second, M-estimators designed for adaptation in one model may continue to exhibit appealing adaptive qualities in alternative models.
For example, an important paper of Abramovich et al.\ \cite{abramovich2006} demonstrated an intriguing connection between adaptive estimation and FDR control in the  Gaussian sequence model: by only estimating those means selected by Benjamini-Hochberg at sufficiently low target FDR, one could achieve asymptotic minimaxity in estimation  across sparsity levels with respect to several losses.
Unfortunately, it is not obvious how to generalize their procedure to linear models.
In contrast, SLOPE, which achieves FDR control and adaptive minimaxity in the sequence model, has an obvious generalization to the linear model \eqref{linear-model}. 
Thus, the identification of a penalty which behaves well in one model may more easily generate candidate procedures in alternative models that we may hope retain good properties.
Indeed, Theorem \ref{thm:finite-sample-lm-concentration} may serve as a starting point for understanding the use of a particular penalty in a linear model with Gaussian designs, but the qualitative behavior identified by such an analysis may generalize to less restrictive design assumptions.
Third, the concentration inequalities we have established in Theorems \ref{thm:symm-like-sep} and \ref{thm:finite-sample-lm-concentration} hold for any choice of convex $f_p$.
A major challenge the statistician faces in designing adaptive procedures is controlling selection bias.
Somehow, the convexity of the estimators controls selection bias in a manner which is uniform over choices of $f_p$. 
Theorems \ref{thm:symm-like-sep} and \ref{thm:finite-sample-lm-concentration} permit the automatic control of bias arising from adaptivity without requiring a separate analysis for each penalty.

We conclude this section by providing several open questions which we believe may prove fruitful.
First, which natural adaptive procedures beyond those of examples \ref{ex:bayes} and \ref{ex:emp-bayes} do or do not satisfy the condition of Theorem \ref{thm:adaptive-char}?
Second, is there a simpler characterization of the implementability of an adaptive procedure than that found in Theorem \ref{thm:adaptive-char}?
Perhaps a weaker sufficient condition than that in Theorem \ref{thm:adaptive-char} exists which  guarantees its the result while being more interpretable and still widely applicable.
Third, is there a notion of ``approximate implementability'' which captures when there exists an $f_p$ such that \eqref{prox-to-scalar-concentration} holds with an $\frA' \approx \frA$ in an appropriate sense?
Insisting on exact implementability of a pre-specified procedure may be too restrictive and conceal the existence of high-quality adaptive penalties.
Fourth, given a certain adaptive goal (adapting to sparsity in the Gaussian sequence model, for example), is there a principled design process by which we might automate the discovery of adaptive symmetric penalties implementing --exactly or approximately-- a particular adaptive procedure or achieving certain minimax rates adaptive to a certain structural parameter? 
Finally, can we prove a theorem analagous to Theorem \ref{thm:adaptive-char} for the linear model \eqref{linear-model}?

\section{Symmetric functions and optimal transport}\label{sec:symm-fnc-opt-transport}

The main strategy towards establishing the results in Section \ref{sec:main-results} is to view lsc, proper, symmetric, convex function on $\reals^p$ as the restrictions of lsc, proper, symmetric, convex function on $L_2(0,1)$. 
Such a viewpoint has been developed in \cite{horsley1988,day2973}, for example.
We develop this viewpoint from a different perspective by drawing on the tools of optimal transport theory, which we believe is particularly natural in statistical applications.

\subsection{Two function spaces and their equivalence}

We consider $\bar\reals$-valued functions defined on $L_2(0,1)$ and on $\cP_2(\reals)$.
Functions defined on $L_2(0,1)$ will be denoted by standard font $f,g,$ etc., and those defined on $\cP_2(\reals)$ will be denoted by fraktur font $\frf,\frg,$ etc.

\begin{definition}[Symmetric functions]\label{def:symmetric-functions}
    A function $f:L_2(0,1) \rightarrow \bar\reals$ is \emph{symmetric} if it is constant on the equivalence classes defined by the equivalence relations
    \begin{equation}\label{equiv-relation}
        X \sim X' \text{ if } \mu_{X} = \mu_{X'}.
    \end{equation}
    Equivalently, $f$ is symmetric if and only if there exists a function $\frf:\cP_2(\reals) \rightarrow \bar\reals$ 
    \begin{equation}\label{eqdef:symmetric-functions}
        f(X) = \frf(\mu_X).
    \end{equation}
\end{definition}

\noindent The structure of a symmetric function $f$ is reflected in the structure of the function $\frf$. 
In particular,
\begin{proposition}\label{prop:lsc-and-convexity}

    Consider a symmetric function $f:L_2(0,1) \rightarrow \bar\reals$.  
    Then, there exists a unique function $\frf: \cP_2(\reals) \rightarrow \bar\reals$ satisfying (\ref{eqdef:symmetric-functions}).
    Moreover, 
    \begin{enumerate}[(a)]

        \item 
        $f$ is proper (i.e.\ not everywhere infinite) if and only if $\frf$ is proper.

        \item 
        $f$ is lower semi-continuous if and only if $\frf$ is lower semi-continuous.

        \item 
        $f$ is convex if and only if
        for all $\mu,\mu' \in \cP_2(\reals)$, any $\pi \in \Pi(\mu,\mu')$, and any $\alpha \in [0,1]$, we have
        \begin{equation}\label{eqdef:wasserstein-convexity}
            \frf(\alpha \mu \oplus_\pi (1-\alpha) \mu' ) \leq \alpha \frf(\alpha) + (1-\alpha) \frf(\mu'),
        \end{equation}
        where $\alpha \mu \oplus_\pi (1-\alpha) \mu'  \in \cP_2(\reals)$ is defined a follows: construct $(X,X') \sim \pi$ (see Lemma \ref{lem:embeddings} below) and define $\alpha \mu \oplus_\pi (1-\alpha) \mu' := \mu_{\alpha X + (1-\alpha)X'}$.

    \end{enumerate}

\end{proposition}
\noindent Throughout the paper, we will always use the topology on $\cP_2(\reals)$ induced by the Wasserstein metric $W_2$. 
Proposition \ref{prop:lsc-and-convexity} justifies the following definition.
\begin{definition}[Convexity on Wasserstein Space]\label{def:wasserstein-convexity}
    A function $\frf: \cP_2(\reals) \rightarrow \bar\reals $ is \emph{convex} if for any $\mu,\mu' \in \cP_2(\reals)$, coupling $\pi \in \Pi(\mu,\mu')$, and $\alpha \in [0,1]$, \myeqref{eqdef:wasserstein-convexity} holds. 
\end{definition}

The proof of Proposition \ref{prop:lsc-and-convexity} relies on the following key embedding lemma, which will also be used repeatedly in the following sections.
It allows us to realize multivariate probability distributions as the joint distributions of collections of random variables in $L_2(0,1)$.
\begin{lemma}\label{lem:embeddings}
    We have the following.
    \begin{enumerate}[(a)]

        \item 
        For any $k \geq 1$ and $\pi \in \cP_2(\reals^k)$, there exists $X_1,X_2,\ldots,X_k \in L_2(0,1)$ with $(X_1,X_2,\ldots,X_k) \sim \pi$.

        \item 
        For any $\mu,\mu_1,\mu_2,\ldots \in \cP_2(\reals)$ and $\pi_p \in \Pi(\mu,\mu_p)$ for all $p$, there exists $X,X_1,X_2,\ldots \in L_2(0,1)$ such that $(X,X_p) \sim \pi_p$ for all $p$.

        \item 
        For any sequence $\{\mu_p\} \subset \cP_2(\reals)$ with $\mu_p \stackrel{\mathrm{W}}\rightarrow \mu \in \cP_2(\reals)$, there exists $X_p,X \in L_2(0,1)$ with $\mu_{X_p} = \mu_p$ for all $p$, $\mu_X = \mu$, and $X_p \stackrel{L_2}\rightarrow X$.   

    \end{enumerate}

\end{lemma}

\noindent Such embedding results are common \cite{Kallenberg2002}. We provide a proof in Appendix \ref{sec:proof-of-lem-embeddings} for the reader's convenience.
We can now prove Proposition \ref{prop:lsc-and-convexity}.

\begin{proof}[Proof of Proposition \ref{prop:lsc-and-convexity}]
    Uniqueness holds because by Lemma \ref{lem:embeddings}.(a), for all $\mu \in \cP_2(\reals)$, there exists a random variable $X \in L_2(0,1)$ with $\mu_X = \mu$. 
    Thus, $\frf(\mu)$ is dicated by $f(X)$, and vice-versa.

    \begin{enumerate}[(a)]
        \item 
        (Proper)
        $f$ is proper if and only if for some $X \in L_2(0,1)$ we have $f(X) < \infty$, which by the embedding Lemma \ref{lem:embeddings}.(a) occurs exactly when for some $\mu \in \cP_2(\reals)$ we have $\frf(\mu) < \infty$.
 
        \item 
        (Lower semi-continuous)
        Assume $f$ is lsc. Consider $\mu_p \stackrel{\mathrm{W}}\rightarrow \mu$.
        Taking $X,X_p$ as in Lemma \ref{lem:embeddings}.(b), we have $\liminf_{p \rightarrow \infty} \frf(\mu_p) = \liminf_{p \rightarrow \infty} f(X_p) \geq f(X) = \frf(\mu)$, whence $\frf$ is lsc.
        Conversely, assume $\frf$ is lsc. Consider $X_p \stackrel{L_2}\rightarrow X$. By definition, we have $W_2(\mu_{X_p},\mu_X) \leq d_2(X_p,X)$, whence $\mu_{X_p} \stackrel{\mathrm{W}}\rightarrow \mu_X$. Thus, $\liminf_{p\rightarrow \infty} f(X_p) = \liminf_{p\rightarrow \infty} \frf(\mu_{X_p}) \geq \frf(\mu_X) = f(X)$, whence $f$ is lsc. 

        \item 
        (Convex) Assume $f$ is convex. Consider $\mu,\mu',\pi,\alpha$ as in the statement of the proposition. 
        Taking $X,X'$ as in Lemma \ref{lem:embeddings}.(a), we have $\frf(\alpha \mu \oplus_\pi (1-\alpha)\mu') = \frf(\mu_{\alpha X + (1-\alpha)X'}) = f(\alpha X + (1-\alpha)X') \leq \alpha f(X) + (1-\alpha)f(X') = \alpha \frf(\mu_X) + (1-\alpha)\frf(\mu_{X'})$, whence we conclude \eqref{eqdef:wasserstein-convexity}. 
        The converse follows via the same logic but by first taking arbitrary $X,X' \in L_2(0,1)$, and then considering $\mu_X,\mu_{X'}$, and $\pi = \mu_{(X,X')}$.
    \end{enumerate}
    The proof is complete.
\end{proof}

\noindent Denote by $\cF$ to the space of lsc, proper, symmetric, and convex functions on $L_2(0,1)$ and $\frF$ the space of lsc, proper, convex functions on $\cP_2(\reals)$. 
Proposition \ref{prop:lsc-and-convexity} establishes a bijection between $\cF$ and $\frF$ via \eqref{eqdef:symmetric-functions}.  

\subsection{Proximal operators on Wasserstein space}

We define the proximal operator of $f \in \cF$ as
\begin{align}\label{eqdef:l2-prox}
    \mathsf{prox}[f](Y) = \arg\min_{X \in L_2(0,1)} \left\{\frac12 \E[(Y-X)^2] + f(X)\right\}.
\end{align}
Because the objective in \eqref{eqdef:l2-prox} is strongly convex and proper, its minimizer exists and is unique, so that $\mathsf{prox}[f]$ is well-defined.
From the symmetry of $f$, one might naturally expect that the joint distribution of $(Y,\mathsf{prox}[f](Y))$ depends only on the distribution of $Y$, so that the proximal operator ``inherits'' the symmetry of $f$.
While this intuition ends up being correct, its proof is not immediate.
The primary difficulty is indicated by the following counter-example.
\begin{example}\label{counter-example}
    There exist $U_1,U_2 \in L_2(0,1)$ such that $U_1,U_2\stackrel{\mathrm{iid}}\sim \mathsf{Unif}(0,1)$.
    This follows from Lemma \ref{lem:embeddings}.(a).
    Nevertheless, there exists $U \in L_2(0,1)$ with $U \sim \mathsf{Unif}(0,1)$ such that $U$ is only independent of constant random variables.
    Indeed, if we take $U(\omega) = \omega$ for $\omega \in (0,1)$, then the sigma-algebra generated by $U$ is the Borel $\sigma$-algebra on $(0,1)$, which is only independent of the trivial $\sigma$-algebra $\{\emptyset,(0,1)\}$.\footnote{This example is related to the potential non-invertibility of measure-preserving maps on measure spaces, and previous authors have observed the challenge this poses in studying symmetric functions in infinite dimensional spaces. See, e.g.\ \cite[pg.~463]{horsley1988}.}
\end{example}
\noindent Example \ref{counter-example} indicates that the geometry of a random variable in $L_2(0,1)$ in relation to the rest of the space does not depend only upon its distribution.
In particular, imagine that for a particular $Y \in L_2(0,1)$, the joint distribution of $(Y,\mathsf{prox}[f](Y))$ is $\pi$.
Example \ref{counter-example} shows that, a priori, it could be the case that for some $\tilde Y \stackrel{\mathrm{d}}= Y$, there exists no $X \in L_2(0,1)$ for which $(Y,X) \sim \pi$, in which case $(\tilde Y,\mathsf{prox}[f](\tilde Y))$ could not have distribution $\pi$. 

This scenario does not occur, and the reason is that the joint distribution between $Y$ and $\mathsf{prox}[f](Y)$ satisfies certain structural properties induced by the proximal minimization.
We identify these structural properties by first studying proximal operators on Wasserstein space and later ``lifting'' results about such proximal operators to $L_2(0,1)$.
We define the proximal operator of a function $\frf \in \frF$ as
\begin{align}\label{eqdef:wass-prox}
    \mathsf{prox}[\frf](\mu) &:= \arg\min_{\nu \in \cP_2(\reals)} \left\{\frac12 W_2(\mu,\nu)^2 + \frf(\nu)\right\}.
\end{align}

\noindent In Appendix \ref{app:prox-well-defined}, we show that when $\frf$ is lsc, proper, and convex, the minimum on the right-hand side of \eqref{eqdef:wass-prox} exists and is unique.
Unlike $\mathsf{prox}[f]$, the proximal operator on Wasserstein space is automatically symmetric: by definition it depends only on the distribution $\mu$. 
This is one of the main benefits of developing first the theory for proximal operators on Wasserstein space before ``lifting'' it to $L_2(0,1)$.

An important object is $\pi_{\mathrm{opt}}(\mu,\mathsf{prox}[\frf](\mu))$, the optimal coupling between $\mu$ and the $\nu$ which minimizes the objective in \eqref{eqdef:wass-prox}.
For any $\pi \in \reals^2$, we denote by $\mathsf{spt}(\pi)$ the support of $\pi$, defined to be the intersection of all closed sets with measure 1 according to $\pi$. 
In particular, $\mathsf{spt}(\pi)$ is a closed set with measure 1 and is the smallest such set.
A standard fact from optimal transport theory is that $\pi$ is the optimal coupling between its marginals if and only if $\mathsf{spt}(\pi) \subset \reals^2$ satisfies the following property \cite[Theorem 5.10.(ii)]{Villani2008}.

\begin{definition}[Cyclic monotonicity on $\reals \times \reals$]\label{def:R-cyc-mon}
    A set $\Gamma \subset \reals \times \reals$ is said to be cyclically monotone if for every $k \in \integers_{>0}$, every permutation $\sigma:[k]\rightarrow[k]$, and every finite family of points $(x_1,y_1),\ldots,(x_k,y_k) \in \Gamma$ we have
    \begin{equation}\label{eqdef:cyclic-monotonicity}
    \sum_{i=1}^k (x_i-y_i)^2 \leq \sum_{i=1}^k (x_i-y_{\sigma(i)})^2.
    \end{equation}
    Equivalently,
    \begin{equation}\label{eqdef:cyclic-monotonicity-inner-product-form}
    \sum_{i=1}^k x_iy_i \geq \sum_{i=1}^k x_iy_{\sigma(i)}.
    \end{equation}
    Equivalently, $\{x_j\}$ and $\{y_j\}$ have the same ordering. That is, there are no $i,j$ such that $x_i < x_j$ and $y_i > y_j$.
\end{definition}

\noindent By \cite[Theorem 5.10.(ii)]{Villani2008}, $\mathsf{spt}(\pi_{\mathrm{opt}}(\mu,\mathsf{prox}[\frf](\mu)))$ is cyclically monotone. 
The proximal minimization \eqref{eqdef:wass-prox} imposes the following additional structure.

\begin{proposition}\label{prop:prox-properties}
    Fix $\mu \in \cP_2(\reals)$.
    Let $\pi_{\mathrm{res}} = \mu_{(Y-B^*,B^*)}$ where $(Y,B^*) \sim \pi_{\mathrm{opt}}(\mu,\mathsf{prox}[\frf](\mu))$.
    Then $\pi_{\mathrm{res}}$ is the optimal coupling between $\mu_{Y-B^*}$ and $\mu_{B^*}$. 
    In particular, both $\mathsf{spt}(\pi_{\mathrm{opt}})$ and $\mathsf{spt}(\pi_{\mathrm{res}})$ are cyclically monotone. 

\end{proposition}

\begin{proof}[Proof of Proposition \ref{prop:prox-properties}]
 
    By \cite[Theorem 5.10.(ii)]{Villani2008}, the cyclic monotonicity of the given sets holds once we establish the optimality of the coupling $\pi_{\mathrm{res}}$. 

    Let $\rho_{\mathrm{opt}}$ be the optimal coupling between $\mu_{Y-B^*}$ and $\mu_{B^*}$. 
    By Lemma \ref{lem:embeddings}.(b), we can construct $(R,B,B')$ random variables on $L_2(0,1)$ such that $(R,B) \sim \pi_{\mathrm{res}}$ and $(R,B') \sim \rho_{\mathrm{opt}}$. 
    Define $Y' = B+R$. 
    Then $(Y',B) \sim \mu_{(Y,B^*)} = \pi_{\mathrm{opt}}$.
    For all $\alpha \in (0,1)$, we have
    \begin{align*}
        \frac12\E\left[(Y' - (\alpha B' + (1-\alpha)B))^2\right] + \frf(\mu_{B}) &= \frac12\E\left[(Y' - (\alpha B' + (1-\alpha) B))^2\right] + \alpha \frf(\mu_{B'}) + (1-\alpha)\frf(\mu_{B})\\
        &\geq \frac12 W_2(\mu_{Y'},\mu_{\alpha B' + (1-\alpha)B}) + \frf(\mu_{tB' + (1-\alpha)B}) \\
        &\geq \frac12 W_2(\mu_{Y'},\mu_{B}) + \frf(\mu_{B})\\
        &= \frac12 \E[(Y' - B)^2] + \frf(\mu_{B}),
    \end{align*}
    where in the first equality we use $\mu_B = \mu_{B'}$, in the first inequality we use \eqref{eqdef:wasserstein-convexity} and the definition of the Wasserstein distance, in the the second we use \eqref{eqdef:wass-prox}, and in the final equality we use the optimality of $\pi_{\mathrm{opt}}$.
    Taking $\alpha \downarrow 0$ gives $\E[R(B' - B)] = \E[(Y' - B)(B' - B)] \leq 0$,
    whence $\E[(R - B)^2] = \E[(R - B')^2] + 2 \E[R(B' - B)] \leq \E[(R - B')^2]$, where the first equality uses $\E[B^2] = \E[{B'}^2]$.
    That is, $\pi_{\mathrm{res}}$ is the optimal coupling between $\mu_{Y - B^*}$ and $\mu_{B^*}$, as desired.
\end{proof}

\subsection{Effective scalar estimators}\label{sec:effective-scalar-estimators}

One important consequence of Proposition \ref{prop:prox-properties} is that the optimal coupling that solves the proximal minimization \eqref{eqdef:wass-prox} is implemented by a deterministic map.
In later sections, we will see that this deterministic map, for a particular choice of $\frf$, is the $\frA_{f_p}$ which appears in Theorems \ref{thm:symm-like-sep}, \ref{thm:finite-sample-lm-concentration}, and \ref{thm:adaptive-char}.

\begin{proposition}[Effective scalar estimators]\label{prop:effective-scalar-estimators}
    Let $\frf \in \frF$. Fix $\mu \in L_2(0,1)$. 
    Then there exists $\eta \in \cPR_1$ such that if $Y \sim \mu$, then $(Y,\eta \circ Y)\sim \pi_{\mathrm{opt}}(\mu,\mathsf{prox}[\frf](\mu))$. 
\end{proposition}

\noindent Note that in the theory of optimal transport, there exist many pairs of probability distributions whose optimal coupling is non-deterministic (see, e.g.\ \cite[pg.\ 6]{Villani2008}); that is, there exist no functions $\eta$ for which the preceding proposition holds. Moreover, when such functions exists, they may --by necessity-- be discontinuous, and if continuous, need not be 1-Lipschitz.
Thus, the optimization \eqref{eqdef:wass-prox} imposes substantial additional structure.

Proposition \ref{prop:effective-scalar-estimators} does not state that the $\eta$ which implements the optimal coupling is unique.
Indeed, if $\mu$ does not have full support, it will in general not be unique because it will typically not be determined outside of the support. 
A central object will be mappings of the following type.
\begin{definition}[Effective scalar representation]\label{def:effective-scalar-representation}
    A mapping $\frA: \cP_2(\reals) \rightarrow \cPR_1$ is an \emph{effective scalar representation} of $\frf \in \frF$ if for each $\mu \in L_2(0,1)$ we have 
    \begin{equation}\label{frA-def}
    (Y, \frA(\mu_Y) \circ Y) \sim \pi_{\mathrm{opt}}(\mu,\mathsf{prox}[\frf](\mu)).
    \end{equation}
    If $f\in \cF$ is related to $\frf$ by \eqref{eqdef:symmetric-functions}, then we say that $\frA$ is an effective scalar representation of $f$.
\end{definition}
\noindent For all $\frf \in \frF$, Proposition \ref{prop:effective-scalar-estimators} guarantees the existence of an effective scalar representation of $\frf$.

\begin{proof}[Proof of Proposition \ref{prop:effective-scalar-estimators}]
    Define $\pi_{\mathrm{res}}$ as in Proposition \ref{prop:prox-properties}. 
    Because $(y,x) \mapsto (y-x,x)$ is a homeomorphism, we have 
    \begin{equation}\label{spt-opt-to-spt-res}
        \mathsf{spt}(\pi_{\mathrm{res}}) = \{(y-x,x) \mid (y,x) \in \mathsf{spt}(\pi_{\mathrm{opt}})\}.
    \end{equation}
    By the cyclic monotonicity of $\mathsf{spt}(\pi_{\mathrm{res}})$ and (\ref{spt-opt-to-spt-res}), we also have that there exist no $(y_0,x_0),(y_1,x_1) \in \mathsf{spt}(\pi_{\mathrm{opt}})$ such that $x_0 < x_1$ and $y_0 - x_0 > y_1 - x_1$.
    Thus, if $x_0 < x_1$, then $y_0 \leq y_1 + x_0 - x_1 < y_1$.
    In particular, for each $y$ for which there exists $x$ with $(y,x) \in \mathsf{spt}(\pi_{\mathrm{opt}})$, that $x$ is unique, so we may define $\eta(y) = x$ for such $y$ and get
    \begin{equation}\label{spt-image-under-eta}
        \mathsf{spt}(\pi_{\mathrm{opt}}) = \{(y,\eta(y)) \mid \exists x \text{ with } (y,x) \in \mathsf{spt}(\pi_{\mathrm{opt}})\}.
    \end{equation}

    By the cyclic monotonicity of $\mathsf{spt}(\pi_{\mathrm{opt}})$, there exists no $(y_0,x_0),(y_1,x_1) \in \mathsf{spt}(\pi_{\mathrm{opt}})$ such that $y_0 < y_1$ and $x_0 > x_1$, whence $\eta$ is non-decreasing on its domain.
    Further, by the cyclic monotonicity of $\mathsf{spt}(\pi_{\mathrm{res}})$, there exists no $(y_0,x_0),(y_1,x_1) \in \mathsf{spt}(\pi_{\mathrm{opt}})$ such that $y_0 < y_1$ and $x_1 - x_0 > y_1 - y_0$.
    Thus, $\eta$ is 1-Lipschitz on its domain. We may extend $\eta$ to a non-decreasing, 1-Lipschitz function on all of $\reals$ by the Lipschitz extension theorem (see, e.g.\ \cite{Evans2015MeasureFunctions}). 
    Eq.~\eqref{spt-image-under-eta} still holds for the extended $\eta$.\footnote{This is the only place where non-uniqueness occurs. Observe, non-uniqueness occurs only if $\mu$ does not have full support.}
    Now, for any $Y \sim \mu$, we have $(Y,\eta \circ Y) \sim \pi_{\mathrm{opt}}$, as desired.
\end{proof}

\subsection{Proximal operators on Hilbert space}

We are now ready to establish the symmetry of proximal operators on $L_2(0,1)$, defined in \eqref{eqdef:l2-prox}.
In addition, the following proposition gives us a convenient representation of these proximal operators.

\begin{proposition}\label{prop:l2-prox-properties}
    Consider any lsc, proper, convex $f$.
    \begin{enumerate}[(a)]
        
        \item 
        $\mathsf{prox}[f]$ is 1-Lipschitz.

        \item 
        Assume $f$ is also symmetric (i.e.\ $f \in \cF$). A mapping $\frA: \cP_2(\reals) \rightarrow \cPR_1$ is an effective scalar representation of $f$ (see Definition \ref{def:effective-scalar-representation}) if and only if for all $Y \in L_2(0,1)$, we have
        \begin{equation}\label{effective-scalar-estimator}
            \mathsf{prox}[f](Y) = \frA(\mu_Y) \circ Y.  
        \end{equation}
        In particular, for all $Y \in L_2(0,1)$, we have $\mathsf{prox}[f](Y)$ is $\sigma(Y)$-measurable.
    \end{enumerate}
    
\end{proposition}

\noindent Note that the right-hand side of \eqref{effective-scalar-estimator} yields the symmetry of $\mathsf{prox}[f]$ we have promised.
Indeed, the function $\frA(\mu_Y) \in \cPR_1$ depends on $Y$ only via its distribution.
Then, the distribution of $(Y,\frA(\mu_Y) \circ Y)$ is the push-forward of the measure $\mu_Y$ through the measurable (in fact, continuous) mapping $y \mapsto (y,\frA(\mu_Y)(y))$, which also depends only on $\mu_Y$.

\begin{proof}[Proof of Proposition \ref{prop:l2-prox-properties}]
    
    Part (a) is standard \cite{Parikh2013ProximalAlgorithms}.

    Now part (b). 
    First assume $\frA$ is an effective scalar representation of $f$, and let $\frf \in \frF$ be related to $f$ via \eqref{eqdef:symmetric-functions}.
    For all $X \in L_2(0,1)$, we have
    \begin{align*}
        \frac12\E[(Y - X)^2] + f(X) &\geq \frac12 W_2(\mu_Y,\mu_X) + \frf(\mu_X) \\
        &\geq \frac12 W_2(\mu_Y,\mathsf{prox}[\frf](\mu_Y)) + \frf(\mathsf{prox}[\frf](\mu_Y)) \\
        & = \frac12 \E[(Y - \frA(\mu_Y)\circ Y))^2] + f(\frA(\mu_Y)\circ Y)),
    \end{align*}
    where the first inequality follows from the definition of the Wasserstein metric and \eqref{eqdef:symmetric-functions}, the second inequality follows from \eqref{eqdef:wass-prox}, and the final equality from \eqref{frA-def}.
    Because the minimizer of \eqref{eqdef:l2-prox} is unique, we have $\frA(\mu_Y)\circ Y = \mathsf{prox}[f](Y)$.
\end{proof}

\subsection{Subdifferentials of symmetric, convex functions}

We now study the subdifferential relations of a function $f \in \cF$ via the same strategy employed above to study their proximal operators. 
In particular, we study a suitably defined subdifferential of a function $\frf \in \frF$ and later ``lift'' the results to functions $f \in \cF$.
The motivation for pursuing this two-step strategy is, as above, to eliminate potential asymmetries arising from difficulties like that in Example \ref{counter-example}.

Recall the subderivative of a convex function $f$ evaluated at $X$, denoted by $\partial f(X)$, is defined by
\begin{equation}\label{subdiff-def}
    \partial f(X) = \{ G \in L_2(0,1) \mid  f(X') \geq f(X) + \E[G(X'-X)] \text{ for all } X' \in L_2(0,1)\}.
\end{equation}
The subdifferential $\partial f \subset L_2(0,1) \times L_2(0,1)$ is the relation defined by
\begin{equation}
    (X,G) \in \partial f \Longleftrightarrow G \in \partial f(X).
\end{equation}
First, we recall what is known about such relations without the assumption of symmetry.
\begin{definition}[Cyclic monotonicity on $L_2(0,1) \times L_2(0,1)$]\label{def:l2-cyc-mon}
    A relation $\cR \subset L_2(0,1) \times L_2(0,1)$ is \emph{cyclically-monotone} if for every finite set $\{(X_j,G_j)\}_{j=1}^n \subset \cR$ and every permutation $\sigma:[n]\rightarrow[n]$, we have
    \begin{equation}\label{eqdef:relation-cyclic-monotonicity}
        \sum_{j=1}^n \E[X_jG_j] \geq \sum_{j=1}^n\E[X_jG_{\sigma(j)}].
    \end{equation} 
    A relation is \emph{maximally cyclically-monotone} if it is cyclically monotone and not a proper subset of another cyclically monotone relation. 
\end{definition}
\noindent There is no clash of terminology between Definitions \ref{def:R-cyc-mon} and \ref{def:l2-cyc-mon}.
Indeed, cyclic monotonicity can be defined for subsets of $H \times H$ where $H$ is any Hilbert space. In \eqref{eqdef:cyclic-monotonicity-inner-product-form} the Hilbert space is $\reals$, and in \eqref{eqdef:relation-cyclic-monotonicity} the Hilbert space is $L_2(0,1)$.

In Theorem 1, the remark following Corollary 2, and Theorem 3 of \cite{rockafellar1966}, Rockafellar establishes that \emph{(i)} the subdifferential of an lsc, proper, convex function $f$ is maximally cyclically-monotone, \emph{(ii)} conversely, any maximally cyclically montone relation is the subdifferential of an lsc, proper, convex function which is unique up to an additive constant, and \emph{(iii)} every cyclically monotone relation is contained in the subdifferential of some lsc, proper, convex function.
Our objective is to make similar statements for the case where $f$ is also symmetric.
First, we observe that symmetry enables a more compact represention of the subdifferential relation.
\begin{proposition}\label{prop:symm-sub-diff}
    If $f \in \cF$, than $\partial f$ is symmetric in the sense that membership of $(X,G)$ in $\partial f$ is determined by the joint distribution $\mu_{(X,G)}$. 
    That is, there exists a subset of $\cP_2(\reals^2)$, which we will denote by $\frD f$, such that the following are equivalent
    \begin{subequations}\label{eqdef:wass-sub-diff}
        \begin{gather}
            \pi \in \frD f, \\
            (X,G) \in \partial f \text{ whenever } \mu_{(X,G)} = \pi,\\
            \text{there exists } (X,G) \in \partial f \text{ with } \mu_{(X,G)} = \pi.
        \end{gather}
    \end{subequations}
\end{proposition}

\begin{proof}[Proof of Proposition \ref{prop:symm-sub-diff}]
    By \cite[Proposition 12.26]{Bauschke2011}, we have that $(X,G) \in \partial f$ if and only if $X = \mathsf{prox}[f](X + G)$.
    But $\mathsf{prox}[f](X + G) = \frA_{\frf}(\mu_{X+G}) \circ (X + G)$. 
    Thus, $(X,G) \in \partial f$ if and only if $X = \frA_{\frf}(\mu_{X+G}) \circ (X + G)$.
    The latter condition only depends upon the joint-distribution $\mu_{(X,G)}$.
\end{proof}

\noindent We call the set $\frD f$ of Proposition \ref{prop:symm-sub-diff} the \emph{Wasserstein subdifferential} of the function $f$. 
For any $\frf \in \frF$, we define
\begin{equation}\label{wass-sub-diff-cf-to-frf}
    \frD \frf := \frD f,
\end{equation}
where $\frf$ is the unique element of $\frF$ identified with $f$ via \eqref{eqdef:symmetric-functions}.
Unsurprisingly, $\frD \frf$ has an equivalent definition intrinsic to the function $\frf$.
\begin{proposition}\label{prop:wass-sub-diff-intrinsic-form}
    For any $\frf \in \frF$, we have $\pi \in \frD \frf$ if and only if for all $(X,G,X')$ with $(X,G) \sim \pi$ and $X' \in L_2(0,1)$
    \begin{equation}\label{eq:wass-sub-diff-intrinsic-form}
        \frf(\mu_{X'}) \geq \frf(\mu_X) + \E[G(X'-X)].
    \end{equation}
\end{proposition}
\begin{proof}[Proof of Proposition \ref{prop:wass-sub-diff-intrinsic-form}]
    First, $\Rightarrow$.
    Let $f$ be related to $\frf$ via \eqref{eqdef:symmetric-functions}. 
    Then, by Proposition \ref{prop:symm-sub-diff}, $(X,G) \sim \pi$ implies $G \in \partial f (X)$, whence $f(X') \geq f(X) + \E[G(X'-X)]$ by \eqref{subdiff-def}. By \eqref{eqdef:symmetric-functions}, this is equivalent to \eqref{eq:wass-sub-diff-intrinsic-form}.
    Now, $\Leftarrow$.
    Assume \eqref{eq:wass-sub-diff-intrinsic-form} holds for all $(X,G,X')$ of the specified form.
    By the coupling lemma (Lemma \ref{lem:embeddings}.(a)), we may construct at least one $(X,G) \sim \pi$. 
    Then, for all $X' \in L_2(0,1)$, we have by \eqref{eq:wass-sub-diff-intrinsic-form} that $f(X') = \frf(\mu_{X'}) \geq \frf(\mu_X) + \E[G(X'-X)] = f(X) + \E[G(X'-X)]$. Thus, $(X,G) \in \partial f$. 
    By Proposition \ref{prop:symm-sub-diff}, $\mu_{(X,G)} \in \frD f = \frD \frf$.  
\end{proof}

\noindent Using the characterization of Proposition \ref{prop:wass-sub-diff-intrinsic-form}, we can, in the spirit of \cite{rockafellar1966}, identify which subsets of $\cP_2(\reals^2)$ are contained in the Wasserstein subdifferential of some function $\frf \in \frF$, or equivalently, of some function $f \in \cF$.
In fact, it is exactly those subsets $\frR \subset \cP_2(\reals^2)$ which are jointly cyclically-monotone, as defined in Definition \ref{def:joint-cyc-mon}.

In our first step towards showing that joint cyclic-monotonicity characterizes Wasserstein subdifferentials, we establish that joint cyclic monotonicity of $\frR$ imposes the same structure on the distributions $\pi \in \frR$ which was identified in Propositions \ref{prop:prox-properties} to hold for the distributions $\pi_{\mathrm{res}}$.

\begin{lemma}\label{lem:joint-cyc-mon-to-opt-coupling}
    If $\frR$ is jointly cyclically-monotone, then for all $\pi \in \frR$ we have
    \begin{enumerate}[(a)]
        
        \item 
        $\pi$ is the optimal coupling between its marginals.

        \item 
        If $(X,G) \sim \pi$, then $\mu_{(X+G,X)}$ is the optimal coupling between its marginals, and there exists function $\eta \in \cPR_1$ such that $X = \eta \circ (X + G)$.

    \end{enumerate}
\end{lemma}

\begin{proof}[Proof of Lemma \ref{lem:joint-cyc-mon-to-opt-coupling}]
    \begin{enumerate}[(a)]
    
    \item 

    By \cite[Theorem 5.10.(ii)]{Villani2008}, this is equivalent to showing that $\mathsf{spt}(\pi)$ is a cyclically monotone subset of $\reals^2$.
    Assume otherwise.
    Then take $(x,g),(x',g') \in \mathsf{spt}(\pi)$ such that $x < x'$ and $g > g'$. 
    Denote by $B_\epsilon(x,g)$ the ball of radius $\epsilon$ in $\reals^2$ centered at $(x,g)$.
    For some $\epsilon > 0 $ sufficiently small, we have $\tilde x < \tilde x'$ and $\tilde g > \tilde g'$ whenever $(\tilde x,\tilde g) \in B_\epsilon(x,g)$ and $(\tilde x',\tilde g') \in B_\epsilon(x',g')$.
    Moreover, $\pi(B_\epsilon(x,g)),\,\pi(B_\epsilon(x',g')) > 0$ by the definition of the support. 
    Define probability measure $\rho$ to be $\pi|_{B_\epsilon(x,g)}/\pi(B_\epsilon(x,g))$ and similarly for $\rho'$. 
    Take $0 < m < \min\{\pi(B_\epsilon(x,g)),\,\pi(B_\epsilon(x',g'))\}$ and take $(X,G) \sim \pi$ ,$U \sim \mathsf{Unif}([0,1])$, $(X_1,G_1) \sim \rho$, and $(X_2,G_2) \sim \rho'$, all independent. Define
    $$
    (X',G') = 
    \begin{cases}
    (X_2,G_2) \quad & \text{if } (X,G) \in B_\epsilon(x,g) \text{ and } U \leq m/\pi(B_\epsilon(x,g)), \\
    (X_1,G_1) \quad & \text{if } (X,G) \in B_\epsilon(x',g') \text{ and } U \leq m/\pi(B_\epsilon(x',g')), \\
    (X,G) \quad & \text{otherwise}.
    \end{cases}
    $$
    It is not hard to verify that $(X,G,X',G')$ couples $\pi$ to itself. Moreover, $(X - X')(G - G')$ is equal to zero except in two cases. First, if $(X,G) \in B_\epsilon(x,g) \text{ and } U \leq m/\pi(B_\epsilon(x,g))$, then it is equal to $(X - X_2)(G - G_2) < 0$ on this event. Second, if $(X,G) \in B_\epsilon(x',g') \text{ and } U \leq m/\pi(B_\epsilon(x',g'))$, then it is equal to $(X - X_1)(G - G_1) < 0$ on this event. 
    Thus, $\E[(X-X')(G-G')] < 0$, which rearranges to
    $$
    \E[XG] + \E[X'G'] < \E[XG'] + \E[X'G],
    $$
    a contradiction. 
    Thus, $\mathsf{spt}(\pi)$ is cyclically monotone.

    \item 
    Because $(x,g) \mapsto (x,x+g)$ is a homeomorphism, we have $\mathsf{spt}(\mu_{(X,X+G)}) = \{(x,x+g) \mid (x,g) \in \mathsf{spt}(\pi)\}$. 
    In particular, there do not exist $(x,y),(x',y') \in \mathsf{spt}(\mu_{(X,X+G)})$ for which $x < x'$ and $y > y'$ because otherwise there exists $(x,g),(x',g') \in \mathsf{spt}(\pi))$ for which $x < x'$ and $g > g'$, contradicting the cyclic monotonicity of $\mathsf{spt}(\pi)$.
    Thus, $\mathsf{spt}(\mu_{(X,X+G)})$ is cyclically monotone.
    The coupling $\mu_{(X,X+G)}$ thus has the same structure that allowed us to conclude it was implemented by a non-decreasing and 1-Lipschitz mapping applied to $X+G$ in the proof of Proposition \ref{prop:effective-scalar-estimators}, and the proof proceeds as there.
    \end{enumerate}
    \noindent The proof is complete.
\end{proof}

\noindent Lemma \ref{lem:joint-cyc-mon-to-opt-coupling} suggests that joint cyclic monotonicity correctly characterizes Wasserstein subdifferentials. 
The next proposition confirms this, providing a characterization analogous to that in \cite{rockafellar1966} for arbitrary lsc, proper, convex functions.
First, we extend Definition \ref{def:joint-cyc-mon} slightly.

\begin{definition}[Maximal joint cyclic monotonicity]
    A subset $\frR \subset \cP_2(\reals^2)$ is \emph{maximally jointly cyclically-monotone} if it is jointly cyclically monotone (see Definition \ref{def:joint-cyc-mon}) and not a proper subset of another jointly cyclically monotone set.
\end{definition}

\begin{proposition}\label{prop:wass-sub-diff-characterization}
    Consider any $\frf \in \frF$ and $f \in \cF$.
    Then $\frD \frf$ and $\frD f$ are maximally jointly cyclically monotone. 
    Conversely, if $\frR \subset \cP_2(\reals^2)$ is maximally jointly cyclically monotone, then it is the Wasserstein subdifferential of an $\frf \in \frF$  and an $f \in \cF$ which are unique up to an additive constant. 
    Further, if $\frR$ is jointly cyclically monotone, it is contained in the Wasserstein subdifferential of some $\frf \in \frF$ and $f \in \cF$. 
\end{proposition}

\noindent The proof of Proposition \ref{prop:wass-sub-diff-characterization} is provided in Appendix \ref{app:proof-of-prop-wass-sub-diff-characterization}.

\subsection{Finite dimensional penalties}\label{sec:finite-d-pen}

Having completed our development of symmetric functions on $L_2(0,1)$ and their relation to functions defined on Wasserstein space, 
we are ready to establish that any lsc, proper, symmetric, convex function $f_p:\reals^p \rightarrow \bar\reals$ can be viewed as the restriction of a function $f \in \cF$ to a certain set of discrete random variables
after embedding $\reals^p$ into $L_2(0,1)$ in a particular way.
This embedding idea has also appeared previously in \cite[Example 4.4]{day2973}.

We will denote the space of lsc, proper, symmetric, convex functions on $\reals^p$ by $\cF_p$.
Let $I_1,\ldots,I_p$ be a partition of $(0,1)$ such that each $I_j$ has Lebesgue measure $1/p$, and let $\cI_p$ be the $\sigma$-algebra generated by this partition. 
Consider the embedding $\iota:\reals^p \rightarrow L_2(0,1)$ defined by
\begin{equation}\label{iota-def}
    \iota(\bx) = \sum_{j=1}^p x_j \mathbf{1}_{I_j}.
\end{equation}
The embedding $\iota$ is a linear isomorphism between $\reals^p$ and the $\cI_p$-measurable random variables in $L_2(0,1)$.
Moreover, it is clear that 
\begin{equation}\label{embedding-preserves-dist}
\mu_{\bx} = \mu_{\iota(\bx)}.
\end{equation}
Under this embedding, the function $f \in \cF$ induces a function $f_p \in \cF_p$ defined by\footnote{We have selected this normalization to match the relationship between $\|\bx - \bx'\|^2$ and $d_2(\iota(\bx),\iota(\bx'))^2$. As a result, later formulas involving proximal operators will not involve annoying factors of $p$.}
\begin{equation}\label{embed-finite-penalty}
    f_p(\bx) = pf(\iota(\bx)) \text{ for all } \bx \in \reals^p.
\end{equation}
Because $\iota$ is linear, continuous, and bijective, $f_p$ is indeed lsc, proper, and convex. 
By \eqref{embedding-preserves-dist}, $f_p$ is also symmetric, whence $f_p \in \cF_p$ as claimed.
We observe that $f_p$ does not depend upon the particular embedding $\iota$ of $\reals^p$ into $L_2(0,1)$.
Indeed, by \eqref{embedding-preserves-dist} and \eqref{eqdef:symmetric-functions}, Eq.\ \eqref{embed-finite-prox} is equivalent to
\begin{equation}\label{embed-finite-penalty-wass-form}
    f_p(\bx) = p\frf(\mu_{\bx}) \text{ for all } \bx \in \reals^p
\end{equation}
for the unique $\frf \in \frF$ related to $f$ by \eqref{eqdef:symmetric-functions}. 
The right-hand side of \eqref{embed-finite-penalty-wass-form} does not depend on $\iota$.
\begin{definition}[$L_2$ and Wasserstein embeddings]
    The function $f \in \cF$ is an \emph{$L_2$ embedding} of $f_p \in \cF_p$ if \eqref{embed-finite-penalty} holds for all isomorphisms $\iota:\reals^p \rightarrow L_2(0,1)$ of the form \eqref{iota-def}.
    The function $\frf \in \frF$ is a \emph{Wasserstein embedding} of $f_p \in \cF_p$ if \eqref{embed-finite-penalty-wass-form} holds.
\end{definition}
\noindent The discussion up to this point establishes the following claim.
\begin{claim}\label{claim:L2-wass-embed}
    The function $f \in \cF$ is an $L_2$-embedding of $f_p \in \cF_p$ if and only if the function $\frf \in\frF$ is a Wasserstein embedding of $f_p$ for the unique $\frf \in \frF$ satisfying \eqref{eqdef:symmetric-functions}.
\end{claim}

\noindent The next proposition shows that $L_2$ and Wasserstein embeddings preserve the structure of proximal operators and their relation to effective scalar representations.
Moreover, such embeddings always exist.
\begin{proposition}\label{prop:prox-embedding}
    Consider $f_p \in \cF_p$. 
    \begin{enumerate}[(a)]
        \item 
        The following are equivalent.
        \begin{enumerate}[(i)]
            \item The function $f \in \cF$ is such that $f+c$ is an $L_2$ embedding of $f_p$ for some $c \in \reals$.
            \item For all isomporphisms $\iota:\reals^p \rightarrow L_2(0,1)$ of the form \eqref{iota-def},
            \begin{equation}\label{embed-finite-prox}
                \mathsf{prox}[f_p](\by) = \iota^{-1}(\mathsf{prox}[f](\iota(\by))) \text{ for all } \by \in \reals^p.
            \end{equation}
            \item There exists an isomporphism $\iota:\reals^p \rightarrow L_2(0,1)$ of the form \eqref{iota-def} for which \eqref{embed-finite-prox} holds.
            \item For all effective scalar representations $\frA$ of $f$,
            \begin{equation}\label{Rp-effective-scalar-estimator}
                \mathsf{prox}[f_p](\by) = \frA(\mu_{\by})(\by) \text{ for all } \by \in \reals^p,
            \end{equation} 
            where it is understood that $\frA(\mu_{\by})$ is applied coordinate-wise.
            \item There exists an effective scalar representation $\frA$ of $f$ such that \eqref{Rp-effective-scalar-estimator} holds.
        \end{enumerate}
        \item 
        The following are equivalent.
        \begin{enumerate}[(i)]
            \item 
            The function $\frf \in \cF$ is  such that $\frf + c$ is a Wasserstein embedding of $f_p$ for some $c \in \reals$.
            \item
            For all effective scalar representations $\frA$ of $\frf$, Eq.\ \eqref{Rp-effective-scalar-estimator} holds.
            \item 
            There exists an effective scalar representation $\frA$ of $\frf$ such that \eqref{Rp-effective-scalar-estimator} holds.
        \end{enumerate}
        \item 
        There exists an $L_2$ embedding $f$ and a Wasserstein embedding $\frf$ of $f_p$.
    \end{enumerate}
\end{proposition}
\noindent Note that \eqref{embed-finite-prox} makes sense because by Proposition \ref{prop:l2-prox-properties}.(b), $\mathsf{prox}[f_p](\iota(\by))$ is guaranteed to be $\cI_p$-measurable, so is in the domain of $\iota^{-1}$.

\begin{proof}[Proof of Proposition \ref{prop:prox-embedding}]
    \begin{enumerate}[(a)]
        \item 
        We prove a cycle of implications.

        \emph{(i)} $\Rightarrow$ \emph{(ii)}: Consider any isomorphism $\iota:\reals^p \rightarrow L_2(0,1)$ of the form \eqref{iota-def}. 
        By \eqref{embed-finite-penalty-wass-form}, for any $\bx \in \reals^p$
        \begin{align}
            \frac12 \|\by - \bx\|^2 + f_p(\bx) &= p\left(\frac12 d_2(\iota(\by),\iota(\bx))^2 + f(\iota(\bx)) + c \right)\nonumber \\
            &\geq p\left(\frac12 d_2(\iota(\by),\mathsf{prox}[f](\iota(\by)))^2 + f(\mathsf{prox}[f](\iota(\by))) + c \right)\nonumber \\
            &= \frac12 \|\by - \iota^{-1}(\mathsf{prox}[f](\iota(\by)))\|^2 + f_p(\iota^{-1}(\mathsf{prox}[f](\iota(\by)))).
        \end{align}
        By the uniqueness of the minimizer in \eqref{eqdef:finite-p-prox}, we have \eqref{Rp-effective-scalar-estimator}.

        \emph{(ii)} $\Rightarrow$ \emph{(iii)}: This is trivial.

        \emph{(iii)} $\Rightarrow$ \emph{(iv)}: Consider an effective scalar representation $\frA$ of $f$.
        Then for all $\by \in \reals^p$ and $i = 1,\ldots,p$,
        \begin{equation*}
            \iota^{-1}(\mathsf{prox}[f](\iota(\by)))_i \stackrel{\substack{\eqref{effective-scalar-estimator}\\\eqref{embedding-preserves-dist}}}= \iota^{-1}(\frA(\mu_{\by}) \circ \iota(\by))_i \stackrel{\eqref{iota-def}}= \iota^{-1}\left(\sum_{j=1}^p \frA(\mu_{\by})(y_j) \mathbf{1}_{I_j}\right)_i = \frA(\mu_{\by})(y_i).
        \end{equation*}

        \emph{(iv)} $\Rightarrow$ \emph{(v)}: We only need to verify the existence of an effective scalar representation $\frA$ of $f$. This holds by Proposition \ref{prop:effective-scalar-estimators}.

        \emph{(v)} $\Rightarrow$ \emph{(i)}: 
        Let $\frA$ be an effective scalar representation of $f$.
        Define $f_p' \in \cF_p$ to satisfy \eqref{embed-finite-penalty}, so that by definition $f$ is an $L_2$ embedding of $f_p'$. 
        Then, as we have already shown, this implies that \eqref{Rp-effective-scalar-estimator} holds with $f_p'$ in place of $f_p$.
        Thus, for all $\by \in \reals^p$, we have $\mathsf{prox}[f_p](\by) = \mathsf{prox}[f_p'](\by)$. 
        From the KKT conditions for the minimization \eqref{eqdef:finite-p-prox}, we have $(\bx,\bg) \in \partial f_p$ if and only if $\bx = \mathsf{prox}[f_p](\bx + \bg)$, and similarly for $f_p'$.
        Thus, the agreement of the proximal operators implies the agreement of the subdifferentials of $f_p$ and $f_p'$.
        By \cite[Theorem 3]{rockafellar1966}, this implies the $f_p$ and $f_p'$ agree up to an additive constant.

        \item 
        By Claim \ref{claim:L2-wass-embed}, \emph{(i)} is equivalent to the function $f$ which satisfies \eqref{eqdef:symmetric-functions} being an $L_2$ embedding of $f_p$. Because by Definition \ref{def:effective-scalar-representation} the effective scalar representations of $\frf$ are exactly the effective scalar representations of $f$, all claimed equivalences follow from part \emph{(a)}.

        \item 
        Consider the subdifferential $\partial f_p = \{(\bx,\bg) \in \reals^p \times \reals^p \mid \bg \in \partial f_p(\bx)\}$. 
        By symmetry, the membership of $(\bx,\bg)$ in $\partial f_p$ depends only upon the joint distribution $\frac1p \sum_{j=1}^p \delta_{(x_j,g_j)}$. (Note that all permutations of the coordinates are invertible, so that we do not face the difficulties of Example \ref{counter-example}).
        We claim that 
        \begin{equation}\label{finite-p-wass-subdiff}
        \frR = \left\{\frac1p \sum_{j=1}^p \delta_{(x_j,g_j)} \bigm\vert (\bx,\bg) \in \partial f_p\right\}
        \end{equation}
        is jointly cyclically monotone. 
        This claim is proved in Appendix \ref{app:prop-Rp-to-l2-embedding-omitted-parts}.

        By Proposition \ref{prop:wass-sub-diff-characterization} and \eqref{wass-sub-diff-cf-to-frf}, there exists $f \in \cF$ such that $\frR \subset \frD f$. 
        For any $\by \in \reals^p$, note that by the KKT conditions for minimization \eqref{eqdef:finite-p-prox}, we have $\by - \mathsf{prox}[f_p](\by) \in \partial f_p(\mathsf{prox}[f_p](\by))$.
        If $(X,Y) \sim \mu_{(\mathsf{prox}[f_p](\bx),\by)}$, then $\mathsf{prox}[f](Y) = X$ because $\mu_{(X,Y-X)} = \mu_{(\mathsf{prox}[f_p](\by),\by)} \in \frR \subset \frD f$.
        Thus, for any isophorphism $\iota:\reals^p\rightarrow L_2(0,1)$ of the form \eqref{iota-def}, we have \eqref{embed-finite-prox}.
        By part \emph{(a)}, $f$ is, up to a constant, an $L_2$ embedding of $f_p$.
        Subtracting the constant yields an $L_2$ embedding of $f_p$. 
        Defining $\frf$ via \eqref{eqdef:symmetric-functions} yields a Wasserstein embedding by Claim \ref{claim:L2-wass-embed}.
    \end{enumerate}
    \noindent The proof is complete.
\end{proof}

\section{Proofs of main results}\label{sec:proofs}

The theory developed in Section \ref{sec:symm-fnc-opt-transport} allows us to establish the results in Section \ref{sec:main-results} by constructing $L_2$ and Wasserstein embeddings and studying the convergence of empirical measures in Wasserstein space.
Here we provide proofs of the results in Section \ref{sec:main-results} using this theory. 
Sometimes we will prove more general results than those stated in Section \ref{sec:symm-fnc-opt-transport} and which may serve as more powerful starting points for future developments.
We defer some technical details to the appendix.

\subsection{Proof of Theorem \ref{thm:symm-like-sep}}\label{sec:proof-of-thm-symm-like-sep}

Theorem \ref{thm:symm-like-sep} is a particular case of a general result which can be used to study estimators in arbitrary statistical models on $\reals^p$.
\begin{proposition}\label{prop:Rp-to-effective-scalar-denoisre-approx}
    Consider any $f_p \in \cF_p$ and $\mu \in \cP_2(\reals)$.
    For any Wasserstein embedding $\frf$ (resp.\ $L_2$ embedding $f$) of $f_p$ and effective scalar representation $\frA$ of $\frf$ (resp.\ $f$), we have that 
    \begin{equation}
        \frac1p \|\mathsf{prox}[f_p](\by) - \frA(\mu)(\by)\|^2 \leq 4W_2(\mu_{\by},\mu)^2,
    \end{equation}
    where it is understood that $\frA(\mu)$ is applied to $\by$ coordinate-wise. 
\end{proposition}

\begin{proof}[Proof of Proposition \ref{prop:Rp-to-effective-scalar-denoisre-approx}]
    By Lemma \ref{lem:embeddings}, let $Y_{\mathrm{emp}},Y \in L_2(0,1)$ be such that $(Y_{\mathrm{emp}},Y) \sim \pi_{\mathrm{opt}}(\mu_{\by},\mu)$.
    Let $\iota:\reals^p \rightarrow L_2(0,1)$ be any isomorphism between $\reals^p$ and the $\sigma(Y_{\mathrm{emp}})$-measurable random variables in $L_2(0,1)$ of the form \eqref{iota-def}.
    Let $f$ be related to $\frf$ by \eqref{eqdef:symmetric-functions}.
    Then 
    \begin{align}
        \frac1p \|\mathsf{prox}[f_p](\by) - \frA(\mu)(\by)\|^2 
        &= d_2\left(\mathsf{prox}[f](Y_\mathrm{emp}),\frA(\mu)\circ Y_\mathrm{emp}\right)^2\nonumber\\
        &\leq \Big(d_2\big(\mathsf{prox}[f](Y_\mathrm{emp}),\mathsf{prox}[f](Y)\big) \nonumber\\
        &\qquad+ d_2\big(\mathsf{prox}[f](Y),\frA(\mu)\circ Y_\mathrm{emp}\big)\Big)^2\nonumber\\
        &\leq \Big(d_2(Y_\mathrm{emp},Y) + d_2\left(\frA(\mu)\circ Y,\frA(\mu)\circ Y_\mathrm{emp}\right)\Big)^2\nonumber\\
        &\leq 4d_2(Y_\mathrm{emp},Y)^2 = 4W_2(\mu_{\by},\mu)^2,
    \end{align}
    where in the first equality we have used \eqref{embed-finite-prox} and that $\iota$ is an isomorphism, in the first inequality we have used the triangle inequality, in the second inequality we have used that $\mathsf{prox}[f]$ is 1-Lipschitz (Proposition \ref{prop:l2-prox-properties}.(a)) and \eqref{effective-scalar-estimator}, and in the last inequality we have used that $\frA(\mu)$ is 1-Lipschitz.
\end{proof}

\noindent We are ready to prove Theorem \ref{thm:symm-like-sep}.

\begin{proof}[Proof of Theorem \ref{thm:symm-like-sep}]
By a known concentration inequality on the empirical distribution of Gaussian observations of a fixed parameter $\btheta$ \cite[Proposition F.2]{miolane2018distribution} (after rescaling by $\tau$), there exist universal functions $\mathsf{c},\mathsf{C}:\reals_{>0} \rightarrow \reals_{>0}$, non-increasing and non-decreasing respectively, such that for all $\epsilon \in (0,1/2]$,
\begin{equation}\label{wass-concentration}
    \P(W_2(\mu_{\by},\mu_{\btheta}^{*\tau})^2 > \tau^2 \epsilon) \leq C\epsilon^{-1}\exp\left(-cp \epsilon^3\log(1/\epsilon)^{-2}\right),
\end{equation}
where $C = \mathsf{C}(\mathsf{snr})$ and $c = \mathsf{c}(\mathsf{snr})$.
Moreover, by Proposition \ref{prop:prox-embedding}.(c), for each $p$ and $f_p \in \cF_p$ we may choose a Wasserstein embedding $\frf$ of $f_p$, and by Proposition \eqref{prop:effective-scalar-estimators}, we may then choose an effective scalar representation of $\frf$ which we denote by $\frA_{f_p}$.
We prove the theorem for these choices of $\mathsf{c},\mathsf{C}$, and $f_p$.

Observe that $W_2(\mu_{\btheta}^{*\tau},\mu^{*\tau}) \leq W_2(\mu_{\btheta},\mu)$ because we may consider the coupling between $\mu_{\btheta}^{*\tau}$ and $\mu^{*\tau}$ constructed as follows: by Lemma \ref{lem:embeddings}.(b), construct $(X,X',G)$ such that $(X,X') \sim \pi_{\mathrm{opt}}(\mu_{\btheta},\mu)$ and $G \sim \normal(0,1)$ independent of $(X,X')$. 
Then $(X + \tau G, X' + \tau G)$ couples $\mu_{\btheta}^{*\tau}$ and $\mu^{*\tau}$ in such a way that $\E[(X + \tau G - X' - \tau G)^2] = \E[(X-X')^2] = W_2(\mu_{\btheta},\mu)^2$. 
By Proposition \ref{prop:Rp-to-effective-scalar-denoisre-approx}, 
\begin{align}
    \frac1p \|\mathsf{prox}[f_p](\by) - \frA_{f_p}(\mu^{*\tau})(\by)\|^2 &\leq 4W_2(\mu_{\by},\mu^{*\tau})^2\nonumber\\
    &\leq 4(W_2(\mu_{\by},\mu_{\btheta}^{*\tau}) + W_2(\mu_{\btheta}^{*\tau},\mu^{*\tau}))^2\nonumber\\
    &\leq 4(W_2(\mu_{\by},\mu_{\btheta}^{*\tau}) + W_2(\mu_{\btheta},\mu))^2.\label{emp-pop-eff-bound}
\end{align}
Combining \eqref{wass-concentration} and \eqref{emp-pop-eff-bound} yields \eqref{prox-to-scalar-concentration}.
\end{proof}

\noindent Note that Proposition \ref{prop:Rp-to-effective-scalar-denoisre-approx} can generate results similar to Theorem \ref{thm:symm-like-sep} in alternative statistical models on $\reals^p$.
It quantifies the discrepancy between a symmetric and separable proximal operator in terms of the distance in Wasserstein space between $\mu_{\by}$ and some fixed $\mu$.
Thus, for any statistical model in which $\mu_{\by}$ concentrates in Wasserstein space, results like Theorem \ref{thm:symm-like-sep}, with different rates of convergence, will apply.
In particular, if $\by$ is a possibly random vector in $\reals^p$, then Proposition \ref{prop:Rp-to-effective-scalar-denoisre-approx} yields
\begin{equation}\label{generic-seq-concentration}
    \P\left(\frac1p \|\mathsf{prox}[f_p](\by) - \frA_{f_p}(\mu)(\by)\|^2 > \epsilon\right) \leq \P\left( 4W_2(\mu_{\by},\mu)^2 > \epsilon\right).
\end{equation}
To use \eqref{generic-seq-concentration}, the analyst only needs to bound the right-hand side; that is, she must study concentration rates in Wasserstein space, which depends only on her model and not he choice of penalty $f_p$.\footnote{Of course, even faster concentration is plausible and may depend upon the choice $f_p$.}
Conveniently, the concentration of empirical measures in Wasserstein space has been extensively studied, see e.g.\ \cite{Fournier2015}.

Moreover, the application of Proposition \ref{prop:Rp-to-effective-scalar-denoisre-approx} is not limited to concentration results. 
For example, it can just as easily generate bounds on the expected discrepancy between the symmetric and separable proximal operator if the expected value of $W_2(\mu_{\by},\mu)^2$ can be controlled in the statistical model of interest.
Even in models in which $\mu_{\by}$ does not concentrate in Wasserstein space, Proposition \ref{prop:Rp-to-effective-scalar-denoisre-approx} may generate stochastic information on the behavior of $\mathsf{prox}[f_p]$.

\subsection{Proof of Theorem \ref{thm:symm-lb-is-sep-lb}}\label{sec:proof-of-thm-symm-lb-is-sep-lb}

\begin{proof}[Proof of Theorem \ref{thm:symm-lb-is-sep-lb}]
    First, we will show that 
    \begin{equation}\label{R-symm-opt-lb}
        \mathsf{R_{symm}^{opt}}(\tau;\mu) \geq \mathsf{R_{sep}^{opt}}(\tau;\mu).
    \end{equation}
    By \eqref{R-symm-def}, we must show that for any $\{f_p\} \in \cC_p$,
    \begin{equation}\label{symm-liminf-lb}
        \liminf_{p\rightarrow\infty} \frac1p \E\left[\|\mathsf{prox}[f_p](\btheta + \tau \bz) - \btheta\|^2\right] \geq \mathsf{R_{sep}^{opt}}(\tau;\mu).
    \end{equation}
    Consider $\{f_p\} \in \cC_p$.
    Using that $\mathsf{prox}[f_p]$ is 1-Lipschitz and the triangle inequality, we have that $\|\mathsf{prox}[f_p](\btheta + \tau \bz) - \btheta\| \geq \|\mathsf{prox}[f_p](\bzero)\| - 2 \|\btheta\| - \tau \|\bz\|$.
    Then
    \begin{equation}
        \frac1p\|\mathsf{prox}[f_p](\btheta + \tau \bz) - \btheta\|^2 \geq \frac1p\|\mathsf{prox}[f_p](\bzero)\|^2 - \frac2p\left(2 \|\btheta\| + \tau \|\bz\|\right)\|\mathsf{prox}[f_p](\bzero)\| 
    \end{equation}
    Because $\frac1{\sqrt p}\E[(2\|\btheta\| + \tau \|\bz\|)] = O(1)$, if $\frac1p\|\mathsf{prox}[f_p](\bzero)\|^2 \rightarrow \infty$, then also 
    $\frac1p \E[\|\mathsf{prox}[f_p](\btheta + \tau \bz) - \btheta\|^2] \rightarrow \infty$.
    Thus, \eqref{symm-liminf-lb} is trivial if $\frac1p\|\mathsf{prox}[f_p](\bzero)\|^2 \rightarrow \infty$.
    Thus, without loss of generality, we may assume $\liminf_{p \rightarrow \infty} \frac1p\|\mathsf{prox}[f_p](\bzero)\| < \infty$.
    In fact, by the same consideration, we may restrict ourselves to the subsequence $\{p_\ell\}$ such that $\frac1p\|\mathsf{prox}[f_{p_\ell}](\bzero)\|^2 < M$ for some fixed $M$ sufficiently large that this subsequence is infinite.
    For simplicity and with some abuse of notation, we denote this subsequence as $\{p\}$. 

    For each $p$, let $\frf^{(p)}$ be a Wasserstein embedding of $f_p$ as permitted by Proposition \ref{prop:prox-embedding}.(c), and let $\frA^{(p)}$ be an effective scalar representation of $\frf^{(p)}$ as permitted by Proposition \ref{prop:effective-scalar-estimators}.
    Denote $\by = \btheta + \tau \bz$.
    An application of the triangle inequality and a few applications of Cauchy-Schwartz yields
    \begin{align}
        &\frac1p\E\left[\|\mathsf{prox}[f_p](\by) - \btheta\|^2\right] \geq \frac1p\E\left[\|\frA^{(p)}(\mu^{*\tau})(\by) - \btheta\|^2\right] \nonumber \\
        &\qquad\qquad - 2\left(\frac1p \E\left[\|\mathsf{prox}[f_p](\by) - \frA^{(p)}(\mu^{*\tau})(\by)\|^2\right]\right)^{1/2} \left(\frac1p \E\left[\|\frA^{(p)}(\mu^{*\tau})(\by) - \btheta\|^2\right]\right)^{1/2}.\label{prox-to-eff-scalar-bound}
    \end{align}
    We will show that 
    \begin{equation}\label{symm-to-sep-to-zero-in-expectation}
        \frac1p \E\left[\|\mathsf{prox}[f_p](\by) - \frA^{(p)}(\mu^{*\tau})(\by)\|^2\right] \rightarrow 0,
    \end{equation}
    whence it follows that
    \begin{align}
        \liminf_{p \rightarrow \infty}\frac1p\E\left[\|\mathsf{prox}[f_p](\by) - \btheta\|^2\right] &\geq \liminf_{p \rightarrow \infty}\frac1p\E\left[\|\frA^{(p)}(\mu^{*\tau})(\by) - \btheta\|^2\right]\nonumber\\
        &= \liminf_{p \rightarrow \infty}\E[(\frA^{(p)}(\mu^{*\tau})(y) - \theta)^2].\label{equivalent-liminfs}
    \end{align}

    To show \eqref{symm-to-sep-to-zero-in-expectation}, we may assume without loss of generality that the models are all defined on the same probability space and are independent for different values of $p$ because these assumptions do not affect the values of expectations.
    Because the coordinates of $\by$ are samplied iid from $\mu^{*\tau}$, by \cite[Lemma 8.4]{Bickel1981SomeBootstrap}, we have that
    $$
        \mu_{\by} \stackrel{\mathrm{W}}\rightarrow \mu^{*\tau} \text{ almost surely.}
    $$
    By Proposition \ref{prop:Rp-to-effective-scalar-denoisre-approx}, we have that 
    $$
        \frac1p\|\mathsf{prox}[f_p](\by) - \frA^{(p)}(\mu^{*\tau})(\by)\|^2 \stackrel{\mathrm{as}}\rightarrow 0.
    $$
    Eq.~\eqref{symm-to-sep-to-zero-in-expectation} will follow if we can establish that $\frac1p\|\mathsf{prox}[f_p](\by) - \frA^{(p)}(\mu^{*\tau})(\by)\|^2$ is uniformly integrable.
    By the triangle inequality and a standard bound on the square of a sum
    $$
        \frac1p\|\mathsf{prox}[f_p](\by) - \frA^{(p)}(\mu^{*\tau})(\by)\|^2 \leq \frac6p\|\by\|^2 + \frac3p \|\mathsf{prox}[f_p](\bzero)\|^2 + \frac3p \|\frA^{(p)}(\mu^{*\tau})(\bzero)\|^2.
    $$
    Because $ \frac6p\|\by\|^2 $ is the empirical mean of iid integrable random variables, it is uniformly integrable over $p$.
    Also, we established above that along the subsequence we have chosen, $\frac3p \|\mathsf{prox}[f_p](\bzero)\|^2$ is uniformly bounded.
    Finally, letting $s_2(\mu)$ denote the second moment of a distribution $\mu$, we have $\frac1{\sqrt p}\|\frA^{(p)}(\mu^{*\tau})(\bzero)\| \leq s_2(\mathsf{prox}[\frf^{(p)}](\mu^{*\tau})) + s_2(\mu^{*\tau}) \leq s_2(\mathsf{prox}[\frf^{(p)}](\mu_0)) + 2s_2(\mu^{*\tau}) = \frac1{\sqrt p}\|\mathsf{prox}[f_p](\bzero)\| + 2s_2(\mu^{*\tau})$, which we have already shown is uniformly bounded over $p$.
    We conclude that $\frac1p\|\mathsf{prox}[f_p](\by) - \frA^{(p)}(\mu^{*\tau})(\by)\|^2$ is uniformly integrable, as desired.
    Thus, we have \eqref{symm-to-sep-to-zero-in-expectation}, and hence by \eqref{prox-to-eff-scalar-bound} we conclude \eqref{equivalent-liminfs}.
    Because $\frA^{(p)}(\mu^{*\tau}) \in \cPR_1$, the right-hand side of \eqref{equivalent-liminfs} is bounded below by $\inf_{f_1 \in \cF_1} \E[(\mathsf{prox}[f_1](\theta + \tau z) - \theta)^2] = \mathsf{R_{sep}^{opt}}(\tau;\mu)$, as desired.
    Thus, we have \eqref{symm-liminf-lb}, hence \eqref{R-symm-opt-lb}.

    Now we must show that \eqref{R-symm-opt-lb} holds with equality.
    To do so, let $\{f_1^{(p)}\}$ be a sequence of functions in $\cF_1$ such that $\E[(\mathsf{prox}[f_1^{(p)}](\theta + \tau z) - \theta)^2] \rightarrow \inf_{f_1 \in \cF_1} \E[(\mathsf{prox}[f_1](\theta + \tau z) - \theta)^2]$.
    Then, for $f_p(\bx) = \sum_{j=1}^p f_1^{(p)}(x_j)$, we have $\frac1p \E\left[\|\mathsf{prox}[f_p](\btheta + \tau \bz) - \btheta\|^2\right] = \E[(\mathsf{prox}[f_1^{(p)}](\theta + \tau z) - \theta)^2]$, whence the infimum is attained. This completes the proof.
\end{proof}

\subsection{Proof of Theorem \ref{thm:finite-sample-lm-concentration}}\label{sec:proof-of-thm-finite-sample-lm-concentration}

\noindent The proof of Theorem \ref{thm:finite-sample-lm-concentration} depends upon Gaussian comparison inequalities, using techniques developed in \cite{stojnic2013framework,thrampoulidis2015regularized,Thrampoulidis2016PreciseHigh-dimensions,miolane2018distribution}.
Similar to the approach developed in Section \ref{sec:symm-fnc-opt-transport}, our analysis will rely on embedding a finite-dimensional optimization problem into an optimization problem on $L_2(0,1)$.

It is convenient to introduce the variable $\bv = \bb - \btheta$. 
We may solve \eqref{lm-cvx-estimator} by instead solving the reparameterized optimization problem
\begin{equation}\label{lm-cvx-estimator-v-param}
    \widehat \bv \in \arg\min_{\bv \in \reals^p} \left\{\frac1{2n} \|\bw - \bX\bv\|^2 + f_p(\bv)\right\},
\end{equation}
so that $\widehat \btheta = \btheta + \widehat \bv$.
We denote the objective in \eqref{lm-cvx-estimator-v-param} by $C(\bv)$.
Define 
\begin{equation}\label{gordon-objective-finite-p}
    L(\bv) = \frac12 \left(\sqrt{\frac{\|\bw\|^2}{n} + \frac{\|\bv\|^2}{n} \frac{\|\bh\|^2}{n} - \frac{\|\bv\|}{\sqrt n} \frac{\bh^\mathsf{T}\bw}{n} } - \frac1n \bg^\mathsf{T}\bv \right)_+^2 + f_p(\btheta + \bv),
\end{equation}
where $\bh \sim \mathsf{N}(\bzero,\bI_n)$ and $\bg \sim \mathsf{N}(\bzero,\bI_p)$, whenever the argument to the square-root is non-negative, and infinity otherwise.
We refer to minimizing $L(\bv)$ as \emph{Gordon's optimization problem} because it can be related to the optimization \eqref{lm-cvx-estimator-v-param} via an improvement of Gordon's theorem \cite{gordon1988} found in \cite{thrampoulidis2015regularized}.
Our main tool is the following lemma, which is a slight modification of Corollary 5.1 in \cite{miolane2018distribution}.
\begin{lemma}[Modification of Corollary 5.1 of \cite{miolane2018distribution}]\label{lem:gordon-thm}
    \hfill 
    \begin{enumerate}[(a)]

        \item 
        Let $D \subset \reals^p$ be a closed set. For all $t \in \reals$ we have
        \begin{equation}
            \P\left(\min_{\bv \in D} C(\bv) \leq t\right) \leq 2 \P\left(\min_{\bv \in D} L(\bv) \leq t\right).
        \end{equation}

        \item 
        Let $D \subset \reals^p$ be a convex closed set. We have for all $t \in \reals$
        \begin{equation}
            \P\left(\min_{\bv \in D} C(\bv) \geq t\right) \leq 2 \P\left(\min_{\bv \in D} L(\bv) \geq t\right).
        \end{equation}

    \end{enumerate}

\end{lemma}

\noindent Lemma \ref{lem:gordon-thm} differs from \cite[Corollary 5.1]{miolane2018distribution} only in that we do not assume Gaussian noise $\bw$ as those authors do, which impacts only the form of the objective $L(\bv)$.
In Appendix \ref{app:proof-of-lem-gordon-thm}, we describe how to derive this objective.
We refer the reader to \cite{gordon1988,thrampoulidis2015regularized,thrampoulidis2018precise,miolane2018distribution} for a detailed description of the Gaussian comparison techniques from which Lemma \ref{lem:gordon-thm} follows.
Lemma \ref{lem:gordon-thm} allows us to extend any concentration inequality characterizing the minimal value of $L(\bv)$ over a fixed set to a concentration inequality characterizing the minimal value of $C(\bv)$ over the same set.
With appropriate choices of the set $D$, these inequalities can generate detailed information regarding the minimizers of \eqref{lm-cvx-estimator}. 

First, we produce a probabilistic bounds for the minimal value of Gordon's optimization problem over two carefully chosen sets.
\begin{lemma}\label{lem:gordon-Rp-stability}
    There exist universal functions $\mathsf{c_1},\mathsf{c_2},\mathsf{c_3},\mathsf{c_4},\mathsf{C},\mathsf{L}:\reals_{>0}^4 \rightarrow \reals_{>0}$ such that the following is true.

    Consider any $f_p \in \cF_p$, $\btheta \in \reals^p$ with $\|\btheta\|^2/p \leq M$ for some $M$, and $\sigma^*,\tau^*,\lambda^*,\delta >0$.
    Assume $f$ is any $L_2$ embedding of $pf_p$ and $\frA$ is any effective scalar representation of $\lambda^*f$, and assume that
    \begin{subequations}\label{fixed-pt-eqns-proof-sec}
        \begin{align}
            {\tau^*}^2 &= {\sigma^*}^2 + \frac1\delta \E\left[(\frA(\mu^{*\tau^*}) \circ Y^* - \Theta^*)^2\right],\\
            \delta &= \lambda^* \left(1 - \frac1{\delta\tau^*} \E[G^*(\frA(\mu^{*\tau^*})\circ Y^*)]\right),
        \end{align}
    \end{subequations}
    where $Y^* = \Theta^* + \tau^* G^*$ with $\Theta^*,G^* \in L_2(0,1)$, $\Theta^* \sim \mu_{\btheta}$ and $G^* \sim \mathsf{N}(0,1)$ independent of $\Theta^*$.
    Define
    \begin{equation}\label{eqdef:L-star}
        L^* = \frac12 \left(\frac{\tau^*\delta}{\lambda^*}\right)^2 + f(\frA(\mu_{Y^*}) \circ Y^*).
    \end{equation}
    Define the random variable $\sigma^2 = \|\bw\|^2/n$ and constants $c_i = \mathsf{c_i}(M,\tau^*,\sigma^*,\delta)$, $C = \mathsf{C}(M,\tau^*,\sigma^*,\delta)$, and $L = \mathsf{L}(M,\tau^*,\sigma^*,\delta)$.
    Then in model \eqref{linear-model} with $n/p = \delta$ and $\bw$ independent of $\bX$, we have for $0 < \epsilon <  c_1$ that
    \begin{align}\label{gordon-concentration}
        &\P\left(\min_{\frac1{\sqrt p}\|\bv\| \leq c_4\tau^*} L(\bv) \geq L^* + 7L\epsilon \text{ or } \min_{\bv \in D_\epsilon^c,\, \frac1{\sqrt p}\|\bv\| \leq c_4\tau^*} L(\bv) \leq L^* + 10L\epsilon\right) \nonumber\\
        &\qquad\qquad\qquad\qquad\qquad \leq C\epsilon^{-1} \exp\left(-c_2p\epsilon^3 \log(1/\epsilon)^{-2}\right) + 2\P\left( \left|\frac{\sigma}{\sigma^*} -1 \right| > c_3\epsilon\right),
    \end{align}
    where
    \begin{equation}\label{wass-neighb}
        D_\epsilon := \left\{\bv \in \reals^p \mid W_2(\mu_{(\btheta + \bv,\btheta)},\mu_{(\frA(\mu^{*\tau^*}) \circ Y^*,\Theta^*)})^2 > {\tau^*}^2\epsilon\right\}.
    \end{equation}
\end{lemma}

\noindent Lemma \ref{lem:gordon-Rp-stability} is proved in Appendix \ref{app:proof-of-gordon-Rp-stability}. 
Using Lemma \ref{lem:gordon-Rp-stability} we may prove Theorem \ref{thm:finite-sample-lm-concentration}.

\begin{proof}[Proof of Theorem \ref{thm:finite-sample-lm-concentration}]
Choose $\mathsf{c_1},\mathsf{c_2},\mathsf{c_3},\mathsf{C}$ as in Lemma \ref{lem:gordon-Rp-stability} and for each $f_p'$ choose $\frA_{f_p'}$ to be any effective scalar representation of any $L_2$ embedding of $pf_p'$, which exists by Propositions \ref{prop:effective-scalar-estimators} and \ref{prop:prox-embedding}.(c).

Consider also $\mathsf{c_4},\mathsf{L}$ as in Lemma \ref{lem:gordon-Rp-stability}.
Define $c_i = \mathsf{c_i}(M,\tau^*,\sigma^*,\delta)$, $C = \mathsf{C}(M,\tau^*,\sigma^*,\delta)$, and $L = \mathsf{L}(M,\tau^*,\sigma^*,\delta)$.
By Lemmas \ref{lem:gordon-thm} and \ref{lem:gordon-Rp-stability}, we have for $0 < \epsilon < c_1$ that 
\begin{align}
\P&\left(\min_{\frac1{\sqrt p} \|\bv\| \leq c_4\tau^*} C(\bv) \geq L^* + 7L\epsilon \quad \text{or} \quad \min_{\bv \in D_\epsilon,\,\frac1{\sqrt p} \|\bv\| \leq c_4\tau^*} C(\bv) \leq L^* + 10L\epsilon\right)\nonumber\\
&\qquad\qquad \leq 2C\epsilon^{-1}\exp\left(-c_2p\epsilon^3\log(1/\epsilon)^{-2}\right) + 4\P\left( \left|\frac{\sigma}{\sigma^*} -1 \right| > c_3\epsilon\right).\label{primary-problem-stability}
\end{align}
Note that when the event in \eqref{primary-problem-stability} fails to occur, by convexity the minimizer of $C(\bv)$ over $\reals^p$ falls in $D_\epsilon$ (defined in \eqref{wass-neighb}).
On this event, we have $W_2(\mu_{(\widehat \btheta,\btheta)},\mu_{(\frA_{\lambda^*f_p}(\mu^{*\tau^*})\circ Y^*,\Theta^*)})^2 \leq {\tau^*}^2\epsilon$.
We conclude \eqref{linear-regression-concentration}.
\end{proof}

\subsection{Proof of Theorem \ref{thm:adaptive-char}}\label{sec:proof-of-thm-adaptive-char}

Fix $p$.
Because $\{\mu_{(\frA(\mu_Y)\circ Y,Y-\frA(\mu_Y)\circ Y)} \mid \mu_Y \in \cD\}$ is jointly cyclically monotone, we may find $f \in \cF$ such that $\{\mu_{(\frA(\mu_Y)\circ Y,Y-\frA(\mu_Y)\circ Y)} \mid \mu_Y \in \cD\} \subset \frD\frf$.
By Proposition \ref{prop:symm-sub-diff}, this implies that for all $Y \in L_2(0,1)$ with $\mu_Y \in \cD$, we have $Y - \frA(\mu_Y) \in \partial f(\frA(\mu_Y) \circ Y)$.
By the KKT conditions, this implies that 
\begin{equation}\label{effective-scalar-estimator-D-restriction}
    \frA(\mu_Y) \circ Y = \mathsf{prox}[f](Y) \text{ for all $Y$ with $\mu_Y \in \cD$ }.
\end{equation} 
Now consider $\frA'$ an effective scalar representation of $f$, which is exist by Proposition \ref{prop:effective-scalar-estimators}.
With $\frA'' = \frA|_{\cD} + \frA'|_{\cD^c}$, we have that \eqref{effective-scalar-estimator} holds for $\frA''$ for all $Y \in L_2(0,1)$ by applying \eqref{effective-scalar-estimator-D-restriction} to $\frA$ on $\cD$ and \eqref{effective-scalar-estimator} to $\frA'$ on $\cD^c$.
By Proposition \ref{prop:l2-prox-properties}.(b), $\frA''$ is an effective scalar representation of $f$.
Now define $f_p$ by \eqref{embed-finite-penalty}, so that $f$ is an $L_2$ embedding of $f_p$.
By considering the same universal functions $\mathsf{c},\mathsf{C}$ of Theorem \ref{thm:symm-like-sep}, the result follows by the proof of that theorem and using that $\frA''|_{\cD} = \frA|_{\cD}$.

\section*{Acknowledgements}

The author is grateful to Andrea Montanari for encouragement and insightful conversations and comments.
This material is based upon work supported by the National Science Foundation Graduate Research Fellowship under Grant No.\ DGE -- 1656518.

\bibliographystyle{alpha}

\begin{thebibliography}{BvdBS{\etalchar{+}}15}

\bibitem[ABDJ06]{abramovich2006}
Felix Abramovich, Yoav Benjamini, David~L. Donoho, and Iain~M. Johnstone.
\newblock Adapting to unknown sparsity by controlling the false discovery rate.
\newblock {\em Ann. Statist.}, 34(2):584--653, 04 2006.

\bibitem[BC11]{Bauschke2011}
Heinz~H. Bauschke and Patrick~L. Combettes.
\newblock {\em {Convex Analysis and Monotone Operator Theory in Hilbert
  Spaces}}.
\newblock Spring Science+businees Media, LLC, New York, NY, 2011.

\bibitem[BCW11]{belloniChernozhukovWang2011}
A.~Belloni, V.~Chernozhukov, and L.~Wang.
\newblock Square-root lasso: pivotal recovery of sparse signals via conic
  programming.
\newblock {\em Biometrika}, 98(4):791--806, 2011.

\bibitem[BF81]{Bickel1981SomeBootstrap}
Peter~J. Bickel and David~A. Freedman.
\newblock {Some Asymptotic Theory for the Bootstrap}.
\newblock {\em The Annals of Statistics}, 9(6):1196--1217, 11 1981.

\bibitem[BG09]{brown2009}
Lawrence~D. Brown and Eitan Greenshtein.
\newblock Nonparametric empirical bayes and compound decision approaches to
  estimation of a high-dimensional vector of normal means.
\newblock {\em Ann. Statist.}, 37(4):1685--1704, 08 2009.

\bibitem[BGT18]{bellec2018}
Pierre~C. Bellec, Lecu\'e Guillaume, and Alexandre~B. Tsybakov.
\newblock Slope meets lasso: Improved oracle bounds and optimality.
\newblock {\em Ann. Statist.}, 46(6B):3603--3642, 12 2018.

\bibitem[BM11]{BM-MPCS-2011}
Mohsen Bayati and Andrea Montanari.
\newblock {The dynamics of message passing on dense graphs, with applications
  to compressed sensing}.
\newblock {\em IEEE Trans. on Inform. Theory}, 57:764--785, 2011.

\bibitem[BM12]{BayatiMontanariLASSO}
Mohsen Bayati and Andrea Montanari.
\newblock {The LASSO risk for gaussian matrices}.
\newblock {\em IEEE Trans. on Inform. Theory}, 58:1997--2017, 2012.

\bibitem[BMN19]{Berthier2017StateFunctions}
Raphael Berthier, Andrea Montanari, and Phan-Minh Nguyen.
\newblock {State evolution for approximate message passing with non-separable
  functions}.
\newblock {\em {Information and Inference}}, 01 2019.

\bibitem[Bol14]{bolthausen2014iterative}
Erwin Bolthausen.
\newblock {An iterative construction of solutions of the TAP equations for the
  Sherrington--Kirkpatrick model}.
\newblock {\em Communications in Mathematical Physics}, 325(1):333--366, 2014.

\bibitem[BvdBS{\etalchar{+}}15]{Bogdan2015SLOPE---AdaptiveOptimization}
Ma{\l}gorzata Bogdan, Ewout van~den Berg, Chiara Sabatti, Weijie Su, and
  Emmanuel~J. Cand{\`{e}}s.
\newblock {SLOPE---Adaptive Variable Selection via Convex Optimization}.
\newblock {\em The Annals of Applied Statistics}, 9(3):1103--1140, 9 2015.

\bibitem[CM19]{celentanoMontanari2019}
Michael Celentano and Andrea Montanari.
\newblock Fundamental barriers to high-dimensional regression with convex
  pentalties.
\newblock {\em {\sf arXiv:1803.06964}}, 2019.

\bibitem[Day73]{day2973}
Peter~W. Day.
\newblock Decreasing rearrangements and doubly stochastic operators.
\newblock {\em Transactions of the American Mathematical Society},
  178:383--392, 1973.

\bibitem[DJ95]{donohoJohnstone1995}
David~L. Donoho and Iain~M. Johnstone.
\newblock Adapting to unknown smoothness via wavelet shrinkage.
\newblock {\em Journal of the American Statistical Association},
  90(432):1200--1224, 1995.

\bibitem[DM16]{Donoho2016HighPassing}
David Donoho and Andrea Montanari.
\newblock {High dimensional robust M-estimation: asymptotic variance via
  approximate message passing}.
\newblock {\em Probability Theory and Related Fields}, 166(3-4):935--969, 12
  2016.

\bibitem[Efr11]{Efron2011TweediesBias}
Bradley Efron.
\newblock {Tweedie's Formula and Selection Bias}.
\newblock {\em Journal of the American Statistical Association},
  106(496):1602--1614, 12 2011.

\bibitem[EG15]{Evans2015MeasureFunctions}
Lawrence~C. Evans and Ronald~F. Gariepy.
\newblock {\em {Measure Theory and Fine Properties of Functions}}.
\newblock CRC Press, Taylor {\&} Francis Group, Boca Raton, FL, revised
  edition, 2015.

\bibitem[EK13]{karoui2013asymptotic}
Noureddine El~Karoui.
\newblock Asymptotic behavior of unregularized and ridge-regularized
  high-dimensional robust regression estimators: rigorous results.
\newblock 2013.
\newblock {\sf arXiv:1311.2445}.

\bibitem[EKBB{\etalchar{+}}13]{ElKaroui2013OnPredictors.}
Noureddine El~Karoui, Derek Bean, Peter~J Bickel, Chinghway Lim, and Bin Yu.
\newblock {On robust regression with high-dimensional predictors.}
\newblock {\em Proceedings of the National Academy of Sciences of the United
  States of America}, 110(36):14557--62, 9 2013.

\bibitem[EM73]{efronMorris1973}
Bradley Efron and Carl Morris.
\newblock Stein's estimation rule and its competitors--an empirical bayes
  approach.
\newblock {\em Journal of the American Statistical Association},
  68(341):117--130, 1973.

\bibitem[FG15]{Fournier2015}
Nicolas Fournier and Arnaud Guillin.
\newblock On the rate of convergence in wasserstein distance of the empirical
  measure.
\newblock {\em Probability Theory and Related Fields}, 162(3):707--738, Aug
  2015.

\bibitem[Gor88]{gordon1988}
Y.~Gordon.
\newblock On milman's inequality and random subspaces which escape through a
  mesh in ℝn.
\newblock In Joram Lindenstrauss and Vitali~D. Milman, editors, {\em Geometric
  Aspects of Functional Analysis}, pages 84--106, Berlin, Heidelberg, 1988.
  Springer Berlin Heidelberg.

\bibitem[HL19]{huLu2019}
Hong Hu and Yue~M. Lu.
\newblock Asymptotics and optimal designs of slope for sparse linear
  regression.
\newblock 2019.

\bibitem[HW88]{horsley1988}
Anthony Horsley and Andrzej Wrobel.
\newblock Subdifferentials of convex symmetric functions: an application of the
  inequality of hardy, littlewood, and p\'olya.
\newblock {\em Journal of Mathematical Analysis and Applications},
  135:462--475, 1988.

\bibitem[JM13]{javanmard2013state}
Adel Javanmard and Andrea Montanari.
\newblock State evolution for general approximate message passing algorithms,
  with applications to spatial coupling.
\newblock {\em Information and Inference: A Journal of the IMA}, 2(2):115--144,
  2013.

\bibitem[JZ09]{jiang2009}
Wenhua Jiang and Cun-Hui Zhang.
\newblock General maximum likelihood empirical bayes estimation of normal
  means.
\newblock {\em Ann. Statist.}, 37(4):1647--1684, 08 2009.

\bibitem[Kal02]{Kallenberg2002}
Olav Kallenberg.
\newblock {\em {Foundations of Modern Probability}}.
\newblock Applied Probability Trust, New York, NY, 2002.

\bibitem[MM18]{miolane2018distribution}
L{\'e}o Miolane and Andrea Montanari.
\newblock The distribution of the lasso: Uniform control over sparse balls and
  adaptive parameter tuning.
\newblock {\em {\sf arXiv:1811.01212}}, 2018.

\bibitem[OB12]{obozinskiBach2012}
Guillaume Obozinski and Francis~R. Bach.
\newblock Convex relaxation for combinatorial penalties.
\newblock 2012.

\bibitem[PB13]{Parikh2013ProximalAlgorithms}
Neal Parikh and Stephen Boyd.
\newblock {Proximal Algorithms}.
\newblock {\em Foundations and Trends in Optimization}, 1(3):123--231, 2013.

\bibitem[Rob56]{robbins1956}
Herbert Robbins.
\newblock An empirical bayes approach to statistics.
\newblock In {\em Proceedings of the Third Berkeley Symposium on Mathematical
  Statistics and Probability, Volume 1: Contributions to the Theory of
  Statistics}, pages 157--163, Berkeley, Calif., 1956. University of California
  Press.

\bibitem[Roc66]{rockafellar1966}
R.~T. Rockafellar.
\newblock Characterization of the subdifferentials of convex functions.
\newblock {\em Pacific J. Math.}, 17(3):497--510, 1966.

\bibitem[SBB17]{sankaranBachBhattacharya2017}
Raman Sankaran, Francis Bach, and Chiranjib Bhattacharya.
\newblock {Identifying Groups of Strongly Correlated Variables through Smoothed
  Ordered Weighted $L_1$-norms}.
\newblock In Aarti Singh and Jerry Zhu, editors, {\em Proceedings of the 20th
  International Conference on Artificial Intelligence and Statistics},
  volume~54 of {\em Proceedings of Machine Learning Research}, pages
  1123--1131, Fort Lauderdale, FL, USA, 20--22 Apr 2017. PMLR.

\bibitem[SC16]{Su2016SLOPEMinimax}
Weijie Su and Emmanuel Cand{\`{e}}s.
\newblock {SLOPE is Adaptive to Unknown Sparsity and Asymptotically Minimax}.
\newblock {\em The Annals of Statistics}, 44(3):1038--1068, 6 2016.

\bibitem[SC18]{sur2018modern}
Pragya Sur and Emmanuel~J Cand{\`e}s.
\newblock A modern maximum-likelihood theory for high-dimensional logistic
  regression.
\newblock {\em {\sf arXiv:1803.06964}}, 2018.

\bibitem[Sto13]{stojnic2013framework}
Mihailo Stojnic.
\newblock A framework to characterize performance of lasso algorithms.
\newblock {\em {\sf arXiv:1303.7291}}, 2013.

\bibitem[TAH16]{Thrampoulidis2016PreciseHigh-dimensions}
Christos Thrampoulidis, Ehsan Abbasi, and Babak Hassibi.
\newblock {Precise Error Analysis of Regularized M-estimators in
  High-dimensions}.
\newblock Technical report, 2016.

\bibitem[TAH18]{thrampoulidis2018precise}
Christos Thrampoulidis, Ehsan Abbasi, and Babak Hassibi.
\newblock Precise error analysis of regularized m-estimators in
  high-dimensions.
\newblock {\em IEEE Transactions on Information Theory}, 2018.

\bibitem[TOH15]{thrampoulidis2015regularized}
Christos Thrampoulidis, Samet Oymak, and Babak Hassibi.
\newblock Regularized linear regression: A precise analysis of the estimation
  error.
\newblock In {\em Conference on Learning Theory}, pages 1683--1709, 2015.

\bibitem[Vil10]{Villani2008}
C{\`e}dric Villani.
\newblock {\em {Optimal Transport, old and new}}.
\newblock Springer-Verlag Berlin Heidelberg, New York, NY, 2010.

\bibitem[XKB12]{xieKouBrown2012}
Xianchao Xie, S.~C. Kou, and Lawrence~D. Brown.
\newblock Sure estimates for a heteroscedastic hierarchical model.
\newblock {\em Journal of the American Statistical Association},
  107(500):1465--1479, 2012.

\end{thebibliography}

\newcommand{\etalchar}[1]{$^{#1}$}

\newpage

\begin{appendices}

\section{Proof of Lemma \ref{lem:gordon-thm}}\label{app:proof-of-lem-gordon-thm}

We write the (random) value of the optimization in \eqref{lm-cvx-estimator-v-param} as
\begin{equation}\label{PO-as-minmax}
    \min_{\bv \in D} \max_{\bu \in \reals^p} \left\{\frac1n \bu^\mathsf{T} \left(\bw - \bX \bv\right) -\frac1{2n} \|\bu\|^2 + f_p(\bv)\right\}.
\end{equation}
Morover, we claim
\begin{equation}\label{SO-as-minmax}
    \min_{\bv \in D} L(\bx) = \min_{\bv \in D} \max_{\bu \in \reals^p} \left\{\frac1n \bu^\mathsf{T}\bw - \frac{\bh^\mathsf{T} \bu}{n} \frac{\|\bv\|}{\sqrt n} - \frac{\bg^\mathsf{T}\bv}{n} \frac{\|\bu\|}{\sqrt n}-\frac1{2n} \|\bu\|^2 + f_p(\bv)\right\}.
\end{equation}
Indeed, if we maximize first over the direction of $\bu$ to get
\begin{align}
    &\min_{\bv \in D} \max_{\beta \geq 0} \left\{\beta \left(\left\|\frac{\bw}{\sqrt n} - \frac{\bh}{\sqrt n} \frac{\|\bv\|}{\sqrt n}\right\| - \frac1n \bg^\mathsf{T}\bv \right) - \frac12 \beta^2 + f_p(\bv)\right\}\\
    &\qquad = \min_{\bv \in D} \max_{\beta \geq 0} \left\{\beta \left(\sqrt{\frac{\|\bw\|^2}{n} + \frac{\|\bv\|^2}{n} \frac{\|\bh\|^2}{n} - \frac{\|\bv\|}{\sqrt n} \frac{\bh^\mathsf{T}\bw}{n}} - \frac1n \bg^\mathsf{T}\bw \right) - \frac12 \beta^2 + f_p(\bv)\right\}
\end{align}
If we then maximize over $\beta$, we get $\min_{\bv} L(\bv)$.

Gordon's Theorem \cite{gordon1988}, in which a minmax problem involving a Gaussian matrix $\bX$ is related to a particular minmax problem involving Gaussian vectors $\bg,\bh$, permits the comparison of the min-max problems \eqref{PO-as-minmax} and \eqref{SO-as-minmax}.
Specifically, if we condition on $\bw$ (so that it may be viewed as deterministic), we may follow exactly the proof of \cite[Corollary 5.1]{miolane2018distribution} --with the only difference being that because $\bw$ is not Gaussian we have a different form for Gordon's optimization problem-- to achieve the probability bounds of Lemma \ref{lem:gordon-thm} conditional on $\bw$.
Taking expectations over $\bw$, we also have the probability bounds of Lemma \ref{lem:gordon-thm} unconditionally.

We refer the reader to \cite{gordon1988,thrampoulidis2015regularized,thrampoulidis2018precise,miolane2018distribution} for a detailed description of the Gaussian comparison techniques involved in the above argument.
The most recent of these, \cite{miolane2018distribution}, includes a mostly self-contained presentation of the relevant results.

\section{Proofs of additional technical results}

\subsection{Proof of Fact \ref{fact:scalar-prox-characterization}}\label{app:proof-of-fact-scalar-prox-characterization}

For $\eta:\reals \rightarrow \reals$, there exists lsc, proper, convex $\rho$ such that $\eta = \mathsf{prox}[\rho]$ if and only if $\{(\eta(y),y-\eta(y))\mid y \in \reals\}$ is cyclically monotone \cite[Theorem 1]{rockafellar1966}. 
This is equivalent to the following: if $\eta(y) < \eta(y')$ then $y - \eta(y) \leq y' - \eta(y')$, and if  $y - \eta(y) < y' - \eta(y')$, then $\eta(y) \leq \eta(y')$. 
This, in turn, is equivalent to $\eta$'s being non-decreasing and 1-Lipschitz, as we now show.

First, consider that $\eta$ is non-decreasing and 1-Lipschitz.
Then, if $\eta(y) < \eta(y')$, it must be the case that $y < y'$, whence $y - \eta(y) = y' - \eta(y') - ( (y - y') - (\eta(y) + \eta(y')) \leq y' - \eta(y')$.
Note that $\eta$ being non-decreasing and 1-Lipschitz is equivalent to $\eta$ and $\mathrm{Id} - \eta$ both being non-decreasing, whence, by the same argument, if  $y - \eta(y) < y' - \eta(y')$, then $\eta(y) \leq \eta(y')$.

Next, we prove the converse.
First, if $y \leq y'$, then $\eta(y) \leq \eta(y')$ because otherwise we would have $\eta(y) > \eta(y')$ and $y - \eta(y) < y' - \eta(y')$, a contradiction.
Thus, $\eta$ is non-decreasing.
By symmetry, $\mathrm{Id}-\eta$ is also non-decreasing, whence $\eta$ is 1-Lipschitz as well.

\subsection{Proof of Lemma \ref{lem:embeddings}}\label{sec:proof-of-lem-embeddings}

\begin{proof}[Proof of Lemma \ref{lem:embeddings}]
    By \cite[Lemma 3.21]{Kallenberg2002}, there exists $U,U_1,U_2,\ldots$ independent, uniformly distributed random variables in $L_2(0,1)$.

    First, we prove part (a).
    Let $A_j$ be the coordinate functions on the probability space $(\reals^k,\pi,\cB)$.
    By \cite[Theorem 3.19]{Kallenberg2002}, there exists $X_1 \in L_2(0,1)$ with $X_1 \stackrel{\mathrm{d}}= A_1$, whence also $X_1 \circ U_1 \stackrel{\mathrm{d}}= A_1$. 
    By \cite[Theorem 6.3]{Kallenberg2002}, the distribution of $A_j$ given $A_1,\ldots,A_{j-1}$ admits a regular conditional probability distribution for each $j=2,\ldots,k$.
    Thus, by \cite[Lemma 3.22]{Kallenberg2002}, there exists for each such $j$ a measurable function $\zeta_j:\reals^{j-1} \times [0,1] \rightarrow \reals$ such that 
    $$
    (A_1,\ldots,A_{j-1}, \zeta_j(A_1,\ldots,A_{j-1},U_j)) \stackrel{\mathrm{d}}= (A_1,\ldots,A_{j-1},A_j).
    $$
    In particular, if we define $X_j$ inductively by $X_j = \zeta_j(X_1,\ldots,X_{j-1},\zeta_j(X_1,\ldots,X_{j-1},U_j))$ for $j \geq 2$, we have the $(X_1,\ldots,X_k) \stackrel{\mathrm{d}}= (A_1,\ldots,A_k) \sim \pi$. Hence, part (a).

    Next, we prove part (b). 
    As above, by \cite[Theorem 3.19]{Kallenberg2002}, there exists $X \in L_2(0,1)$ with $X \sim \mu$, whence also $X \circ U \sim \mu$. 
    By \cite[Theorem 6.3]{Kallenberg2002}, each coupling $\pi_p$ admits regular conditional probability distributions, whence by \cite[Lemma 3.22]{Kallenberg2002}, there exists for each $p$ a measurable function $\zeta_p:\reals \times [0,1] \rightarrow \reals$ such that $(X\circ U,\zeta_p(X\circ U,U_p)) \sim \pi_p$. Defining $X_p = \zeta_p(X \circ U, U_p)$ yields part (b).
    
    Part (c) follows by applying part (b) to the couplings $\pi_p = \pi_{\mathrm{opt}}(\mu,\mu_p)$ for each $p$ and observing that if $(X,X_p) \sim \pi_{\mathrm{opt}}(\mu,\mu_p)$ then $d_2(X,X_p) = W_2(\mu,\mu_p)$.
\end{proof}

\subsection{Proximal operators on Wasserstein space are well-defined}\label{app:prox-well-defined}

In this Appendix, we show that when $\frf$ is lsc, proper, and convex, the minimizer on the right-hand side of \eqref{eqdef:wass-prox} exists and is unique.

First, we show that 
\begin{equation}\label{wass-moreau}
    \inf_{\nu \in \cP_2(\reals)} \left\{\frac12 W_2(\mu,\nu)^2 + \frf(\nu)\right\}
\end{equation}
is finite.
Becuase $\frf$ is proper, there exists $\nu$ for which the objective is less than $\infty$, whence we must only show the infimum is not $-\infty$.
Take any $\nu$ such that $\frf(\nu) < \infty$.
By lower semi-continuity, there exists $\epsilon > 0$ such that $\frf(\nu') > \frf(\nu) - 1$ whenever $W_2(\nu',\nu) \leq \epsilon$.
Fix such an $\epsilon$, and consider $\nu'$ with $W_2(\nu',\nu) > \epsilon$.
Observe that
\begin{equation}
    W_2\left(\frac{\epsilon}{W_2(\nu',\nu)} \nu' \oplus_{\pi_{\mathrm{opt}}(\nu',\nu)} \frac{W_2(\nu',\nu) - \epsilon}{W_2(\nu',\nu)} \nu, \nu \right) \leq \epsilon.
\end{equation}
Thus,
\begin{equation}
    \frf(\nu) - 1 \leq \frf\left(\frac{\epsilon}{W_2(\nu',\nu)} \nu' \oplus_{\pi_{\mathrm{opt}}(\nu',\nu)} \frac{W_2(\nu',\nu) - \epsilon}{W_2(\nu',\nu)} \nu\right) \leq \frac{\epsilon}{W_2(\nu',\nu)} \frf(\nu') + \frac{W_2(\nu',\nu) - \epsilon}{W_2(\nu',\nu)} \frf(\nu),
\end{equation}
where the first inequality holds by the definition of $\epsilon$ and the second by convexity of $\frf$ (see Definition \ref{def:wasserstein-convexity}).
Rearranging, we have
\begin{equation}
    \frf(\nu') \geq \frf(\nu) - \frac{W_2(\nu',\nu)}{\epsilon}.
\end{equation}
Thus,
\begin{align}
    &\inf_{\nu' \in \cP_2(\reals)} \left\{\frac12W_2(\mu,\nu')^2 + \frf(\nu') \right\} \nonumber \\ 
    &\qquad \geq \min\left\{\frf(\nu) - 1 , \inf_{\nu' \in \cP_2(\reals)} \left\{ \frac12 W_2(\mu,\nu')^2 + \frf(\nu) - \frac{W_2(\nu',\nu)}{\epsilon}\right\}\right\}\nonumber\\
    &\qquad \geq \min\left\{\frf(\nu) - 1 , \inf_{\nu' \in \cP_2(\reals)} \left\{ \frac12 (W_2(\mu,\nu) - W_2(\nu',\nu))^2 + \frf(\nu) - \frac{W_2(\nu',\nu)}{\epsilon}\right\}\right\}\nonumber \\
    &\qquad = \min\left\{\frf(\nu) - 1 , \frac12 W_2(\mu,\nu)^2 + \frf(\nu) + \inf_{\nu' \in \cP_2(\reals)} \left\{ \frac12 W_2(\nu',\nu)^2  - \left(\frac{1}{\epsilon} +  W_2(\mu,\nu)\right)W_2(\nu',\nu)\right\}\right\}\nonumber\\
    &\qquad \geq \min\left\{\frf(\nu)-1, \frac12 W_2(\mu,\nu)^2 + \frf(\nu) - \frac12 \left(\frac{1}{\epsilon} +  W_2(\mu,\nu)\right)^2\right\} > -\infty,
\end{align}
where the first inequality holds by dividing into cases $W_2(\nu',\nu)\leq\epsilon$ and $> \epsilon$, the second inequality holds by the triangle inequality, and the remaining (in)equalities hold by algebrabraic manipulation.
Thus, \eqref{wass-moreau} is finite.

Now consider any $\epsilon \geq 0$ and $\nu',\nu''$ which achieve objective values within $\epsilon$ of the infimum (here we use finiteness of \eqref{wass-moreau}) .
By Lemma \ref{lem:embeddings}.(b), there exists $X,X',X'' \in L_2(0,1)$ with $(X,X') \sim \pi_{\mathrm{opt}}(\mu,\nu')$ and $(X,X'') \sim \pi_{\mathrm{opt}}(\mu,\nu'')$.
Then 
\begin{align}
    \frac12 \left(\frac12 W_2(\mu,\nu')^2 + \frf(\nu')\right) &+ \frac12 \left(\frac12 W_2(\mu,\nu'')^2 + \frf(\nu'')\right) - \epsilon \nonumber\\
    &\leq \inf_{\nu\in\cP_2(\reals)} \left\{\frac12 W_2(\mu,\nu)^2 + \frf(\nu)\right\}\nonumber\\
    &\leq \frac12 W_2\left(\mu,\mu_{(X'+X'')/2}\right) + \frf\left(\mu_{(X+X'')/2}\right)\nonumber\\
    &\leq \frac12 W_2\left(\mu,\mu_{(X'+X'')/2}\right) + \frac12\frf(\nu') + \frac12 \frf(\nu'')\nonumber\\
    &\leq \frac12 d_2\left(X, (X' + X'')/2\right)^2 + \frac12\frf(\nu') + \frac12 \frf(\nu'').\label{wass-prox-midpt-identity}
\end{align}
The parallelogram identity gives
\begin{align}
    \frac14 d_2(X',X'')^2 &= \frac12 d_2(X,X')^2 + \frac12 d_2(X,X'')^2 - d_2(X,(X'+X'')/2)^2\nonumber\\
    &= \frac12 W_2(\mu,\nu')^2 + \frac12 W_2(\mu,\nu'')^2 - d_2(X,(X'+X'')/2)^2\nonumber\\
    &\leq 2\epsilon,\label{almost-minimizers-close}
\end{align}
where the last inequality follows from \eqref{wass-prox-midpt-identity}. Thus, $W_2(\nu',\nu'')^2 \leq 8 \epsilon$.

The argument in the preceding paragraph implies that if $\{\nu_j\}$ is a sequence along which the objective value in \eqref{wass-moreau} approaches its infimum, the sequence is Cauchy. 
Because Wasserstein space is complete \cite[Theorem 6.18]{Villani2008}, this sequence has limit $\nu^*$. 
Because the objective is lower-semicontinuous, the infimum is achieved at $\nu^*$.

That minimizer of the objective in \eqref{wass-moreau} is unique follows by taking $\epsilon = 0$ in \eqref{almost-minimizers-close}.

\subsection{Proof of Proposition \ref{prop:wass-sub-diff-characterization}}\label{app:proof-of-prop-wass-sub-diff-characterization}

\begin{proof}[Proof of Proposition \ref{prop:wass-sub-diff-characterization}]
    Because \eqref{wass-sub-diff-cf-to-frf} whenever $f$ and $\frf$ are related by \eqref{eqdef:symmetric-functions}, it is enough to prove each part of the proposition for either $f$ or $\frf$.
    At each point of the proof, we work with whichever is most convenient.\\

    \noindent {\bf Step 1: $\frD f$ (and hence $\frD f$) are maximally jointly cyclically monotone.}
    First we show that if $f \in \cF$, then $\frD f$ is maximally jointly cyclically monotone. 
    Indeed, for all finite $n$, all sets $\{\pi_j\}_{j=1}^n \subset \frD f$, and all random vectors $(X_1,G_1,\ldots,X_n,G_n)$ with $(X_j,G_j) \sim \pi_j$ for all $j$, we have by Proposition \ref{prop:symm-sub-diff} that $G_j \in \partial f(X_j)$.
    Then, for all permutations $\sigma:[n]\rightarrow [n]$, we have by the definition of subdifferential that for all $j \in [n]$ 
    $$
    f(X_{\sigma(j+1)}) \geq f(X_{\sigma(j)}) + \E[G_{\sigma(j)}(X_{\sigma(j+1)}-X_{\sigma(j)})],
    $$
    where by convention we denote $\sigma(n+1) = \sigma(1)$.
    Summing over $j \in [n]$ and rearranging gives \eqref{join-cycl-mono-characterization}, whence $\frD f$ is jointly cyclically monotone.
    To see that it is maximal, assume for contradiction there were $\frR \subset \cP_2(\reals^2)$ strictly containing $\frD f$ which was jointly cyclically monotone.
    Then by Lemma \ref{lem:embeddings}.(a), we have $\{(X,G) \mid X,G \in L_2(0,1),\, \mu_{(X,G)} \in \frR\}$ is a strict super-set of $\partial f$.
    Moreover, by \eqref{join-cycl-mono-characterization}, it is cyclically monotone.
    This contradicts Theorem 3 of \cite{rockafellar1966}, which states that the subdifferential relation $\partial f$ is always maximally cyclically monotone.
    Thus, $\frD f$ is maximally jointly cyclically monotone.\\

    \noindent {\bf Step 2: If $\frR$ is jointly cyclically monotone, it is contained in the Wasserstein sub-differential of an $\frf \in \frF$ (and hence $f \in \cF$).}
    We imitate the proof of \cite[Theorem 1]{rockafellar1966}, but apply the proof technique to $\frf$ instead of $f$.

    Pick some $\pi_0 \in \frR$. 
    For all $\mu \in \cP_2(\reals)$, define 
    \begin{align}\label{construct-frf-f-from-subdiff}
        \frf(\mu) &= \sup\{  \E[G_n(X - X_n)] + \E[G_{n-1}(X_n-X_{n-1}) +\cdots + \E[G_0(X_1-X_0)] \},
    \end{align}
    where the supremum is taken over all positive integers $n$ and all vectors $(X,X_0,G_0,\ldots,X_n,G_n)$ satsifying $\mu_X = \mu$ and $\mu_{(X_j,G_j)} = \pi_j$ for all $0\leq j \leq n$.
    We confirm that $\frf$ so defined is lsc, proper, and convex, and then will show that $\frR \subset \frD \frf$.

    First, we show $\frf$ is proper.
    Let $\mu_0$ be the first marginal of $\pi_0$. We show that $\frf(\mu_0) < \infty$.
    Let $\{\pi_j\}_{j=1}^n \subset \frD f$ and $(X,X_0,G_0,\ldots,X_n,G_n)$ be such that $\mu_{X} = \mu_0$ and $\mu_{(X_j,G_j)} = \pi_j$.
    Let $F_{X_j}$ and $F_{G_j}$ denote the cdfs of $X_j,G_j$, respectively.
    Consider $\tilde X_j,\tilde G_j \in L_2(0,1)$ defined by $\tilde X_j(t) = F_{X_j}^{-1}(t)$ and $\tilde G_j = F_{G_j}^{-1}(t)$ for $t \in (0,1)$. 
    By \cite[Theorem 5.10.(ii)]{Villani2008}, for any $i,j$ we have that $\mu_{(\tilde X_i, \tilde G_j)}$ is the optimal coupling between its marginals because $F_{X_j}$ and $F_{G_j}$ have the same ordering with respect to $t$, so that the support of $\mathsf{spt}(\mu_{(\tilde X_i, \tilde G_j)})$ so constructed is cyclically monotone (see Definition \ref{def:R-cyc-mon}).
    Thus,
    $$
        \E[\tilde G_n \tilde X_0] + \E[\tilde G_{n-1} \tilde X_n ] + \cdots + \E[\tilde G_0\tilde X_1] \geq \E[G_n X] + \E[ G_{n-1} X_n ] + \cdots + \E[ G_0X_1].
    $$
    Moreover, by Lemma \ref{lem:joint-cyc-mon-to-opt-coupling}.(a), we have that $\mu_{(X_j,G_j)}$ is also the optimal coupling between its marginals, whence $\mu_{(\tilde X_j,\tilde G_j)} = \mu_{(X_j,G_j)}$ for all $j$.
    Thus,
    $$
        \E[\tilde G_n \tilde X_n] + \E[\tilde G_{n-1} \tilde X_{n-1} ] + \cdots + \E[\tilde G_0\tilde X_0] = \E [G_n X_n] + \E[ G_{n-1} X_{n-1} ] + \cdots + \E[ G_0 X_0].
    $$
    Combining the previous two displays,
    \begin{align*}
        \E[\tilde G_n (\tilde X_0 - \tilde X_n)] &+ \E[\tilde G_{n-1}(\tilde X_n - \tilde X_{n-1}) ] + \cdots + \E[\tilde G_0(\tilde X_1 - \tilde X_0)] \\
        &\qquad\geq \E[G_n (X -X_n)] + \E[ G_{n-1} (X_n - X_{n-1}) ] + \cdots + \E[ G_0 ( X_1 - X_0 ) ].
    \end{align*}
    Moreover, the left-hand side is bounded above by $0$ by Definition \ref{def:joint-cyc-mon} because $(\tilde X_j,\tilde G_j) \sim \pi_j$ for all $j$ and $\frR$ is jointly cyclically monotone.
    Thus, we conclude that $\frf(\mu_0) \leq 0$.

    Second, we show $\frf$ is lower semi-continuous. 
    Consider $\mu \in \cP_2(\reals)$ and $\frf(\mu) < \infty$.
    Let $(X,X_0,G_0,\ldots,X_n,G_n)$ be as above with 
    \begin{equation}\label{wass-lsc-eps}
        \E[G_n(X - X_n)] + \E[G_{n-1}(X_n-X_{n-1}) +\cdots + \E[G_0(X_1-X_0)] > \frf(\mu) - \epsilon.
    \end{equation}
    Now consider any $\mu'$ such that $W_2(\mu,\mu') < \epsilon/\sqrt{\E[G_n^2]}$.
    Then, by the gluing lemma \cite[pg.~11]{Villani2008} and Lemma \ref{lem:embeddings}.(a), we may construct $(X',\tilde X, \tilde X_0, \tilde G_0,\ldots,\tilde X_n,\tilde G_n)$ such that $\mu_{(X',\tilde X)}$ is the optimal coupling between $\mu'$ and $\mu$, and $(\tilde X, \tilde X_0, \tilde G_0,\ldots,\tilde X_n,\tilde G_n) \stackrel{\mathrm{d}}= ( X,  X_0,  G_0 , \ldots , X_n, G_n)$.
    By \eqref{construct-frf-f-from-subdiff},
    \begin{align*}
    \frf(\mu') &\geq \E[\tilde G_n(X' - \tilde X_n)] + \E[\tilde G_{n-1}(\tilde X_n- \tilde X_{n-1})] +\cdots + \E[\tilde G_0(\tilde X_1 - \tilde X_0)]\\
    &= \E[\tilde G_n(X - \tilde X_n)] + \E[\tilde G_{n-1}(\tilde X_n- \tilde X_{n-1})] +\cdots + \E[\tilde G_0(\tilde X_1 - \tilde X_0)] + \E[\tilde G_n(X' - X)]\\
    & \geq \frf(\mu) - \epsilon - \sqrt{\E[\tilde G_n^2]\E[(X'-X)^2]} \\
    &= \frf(\mu) - \epsilon - W_2(\mu,\mu')\sqrt{\E[G_n^2]}\\
    &> \frf(\mu) - 2\epsilon.
    \end{align*}
    Because $\epsilon$ was arbitrary, we get lower semi-continuity at any $\mu$ such that $\frf(\mu) < \infty$.
    For $\frf(\mu) = \infty$, we instead take the right-hand side of \eqref{wass-lsc-eps} to be $N$ and take $N \rightarrow \infty$.

    Third, we show $\frf$ is convex.
    Consider any $\mu,\mu' \in \cP_2(\reals)$, any coupling $\pi$ between them, and any $\alpha \in [0,1]$.
    If $\frf(\alpha \mu \oplus_\pi (1-\alpha)\mu') < \infty$, consider any $\epsilon > 0$ and take $(X'',X_0,G_0,\ldots,X_n,G_n)$ such that $X'' \sim \alpha \mu \oplus_\pi (1-\alpha)\mu'$, 
    $(X_j,G_j) \sim \pi_j$ for $1 \leq j \leq n$ where $\pi_j \in \frR$ and $j \geq 1$, $(X_0,G_0)\sim \pi_0$, and 
    \begin{equation}\label{convexity-eps-close}
        \frf(\alpha \mu \oplus_\pi (1-\alpha)\mu') \leq \E[G_n(X'' - X_n)] + \cdots + \E[G_0(X_1 - X_0)] + \epsilon.
    \end{equation}
    By the gluing lemma \cite[pg.~11]{Villani2008} and Lemma \ref{lem:embeddings}.(a), we can construct $(\tilde X,\tilde X',\tilde X'',\tilde X_0,\tilde G_0,\ldots,\tilde X_n,\tilde G_n)$ such that $\tilde X'' = \alpha \tilde X + (1-\alpha)\tilde X'$ (almost surely) and $(\tilde X'',\tilde X_0,\tilde G_0,\ldots,\tilde X_n,\tilde G_n)$ has the same distribution as $(X'',X_0,G_0,\ldots,X_n,G_n)$.
    Then 
    \begin{align*}
        \frf(\alpha \mu \oplus_\pi (1-\alpha)\mu') &\leq \E[\tilde G_n(\alpha \tilde X + (1-\alpha)\tilde X' - \tilde X_n)] + \cdots + \E[\tilde G_0(\tilde X_1 - \tilde X_0)] + \epsilon\\
        &\leq \alpha \frf(\mu) + (1-\alpha)\frf(\mu') + \epsilon.
    \end{align*}
    Taking $\epsilon \rightarrow 0$ yields the result. If $\frf(\alpha \mu \oplus_\pi (1-\alpha)\mu') = \infty$, we instead take $\E[G_n(X'' - X_n)] + \cdots + \E[G_0(X_1 - X_0)] \geq N$ in place of \eqref{convexity-eps-close} and take $N \rightarrow \infty$ to conclude that either $\frf(\mu)$ or $\frf(\mu')$ is infinite.

    Finally, we check that $\frR \subset \frD \frf$.
    We do this by checking the condition in Proposition \ref{prop:wass-sub-diff-intrinsic-form}.
    Consider $\pi \in \frR$, and let $\mu$ be its first marginal.
    Let $\mu' \in \cP_2(\reals)$.
    Consider $(X,G,X')$ with $(X,G) \sim \pi$ and $X' \sim \mu$.
    Fix $\alpha < \frf(\mu)$.
    By \eqref{construct-frf-f-from-subdiff}, the gluing lemma \cite[pg.~11]{Villani2008} and Lemma \ref{lem:embeddings}.(a), we may assume that we in fact have $(X,G,X',X_0,G_0,\ldots,X_n,G_n)$ such that
    $$
        \E[G_n(X-X_n)] + \cdots + \E[G_0(X_1-X_0)] > \alpha.
    $$
    By \eqref{construct-frf-f-from-subdiff},
    $$
        \frf(\mu') \geq \E[G(X'-X)] + \E[G_n(X-X_n)] + \cdots + \E[G_0(X_1 - X_0)] > \alpha + \E[G(X'-X)].
    $$
    Taking $\alpha \uparrow \frf(\mu_X)$ gives \eqref{eq:wass-sub-diff-intrinsic-form}.\\

    \noindent {\bf Step 3: If $\frR$ is maximally jointly cyclically monotone, then it is the Wasserstein subdifferential of an $\frf \in \frF$ (and hence $f \in \cF$).} By step 2, $\frR \subset \frD \frf$ for some $\frf \in \frF$. 
    By step 1, $\frD \frf$ is jointly cyclically monotone.
    The maximality of $\frR$ implies $\frR = \frD \frf$.\\

    \noindent {\bf Step 4: Uniqueness up to an additive constant.}
    The Wasserstein subdifferential of $f$ uniquely determines the subdifferential of $f$ by Proposition \ref{prop:symm-sub-diff}, and the subdifferential of $f$ uniquely determines $f$ up to an additive constant by \cite[Theorem 3]{rockafellar1966}.
\end{proof}

\subsection{Omitted part of proof of Proposition \ref{prop:prox-embedding}.(c)}\label{app:prop-Rp-to-l2-embedding-omitted-parts}

Here we show that $\frR$ defined in \eqref{finite-p-wass-subdiff} is jointly cyclically monotone, as claimed.

First, we show that $\mathsf{spt}\left(\frac1p \sum_{i=1}^p \delta_{(x_i,g_i)}\right)$ is cyclically monotone for any $(\bx,\bg) \in \partial f$.
By \cite[Theorem 1]{rockafellar1966}, we have that for all $(\bx_1,\bg_1),\ldots,(\bx_n,\bg_n) \in \partial f_p$ and all permutation $\sigma:[n]\rightarrow [n]$ that 
\begin{equation}\label{Rp-sub-diff-cyc-mon}
    \sum_{j=1}^n \langle \bx_j,\bg_j\rangle \geq \sum_{j=1}^n \langle \bx_j, \bg_{\sigma(j)}\rangle.
\end{equation}
For any permutation $\tau:[p]\rightarrow [p]$, denote $\bx^{\tau}\in \reals^p$ the vector with $(\bx^{\tau})_i = \bx_{\tau(i)}$.
Then, if we let $\tau$ be the transposition $(i,k)$ for $i \neq k$ and consider any $(\bx,\bg) \in \partial f_p$, we also have $(\bx^{\tau},\bg^{\tau}) \in \partial f_p$ by symmetry.
Thus,
$$
    \langle \bx , \bg \rangle + \langle \bx^{\tau} , \bg^{\tau} \rangle \geq \langle \bx , \bg^{\tau} \rangle + \langle \bx^{\tau},\bg\rangle,
$$
which is exactly the inequality $(x_i - x_k)(g_i-g_k)\geq0$. 
This is equivalent to the statement that for no $i,k$ is it the case that $x_i > x_k$ and $g_i < g_k$.
By Definition \ref{def:R-cyc-mon}, $\mathsf{spt}\left(\frac1p \sum_{i=1}^p \delta_{(x_i,g_i)}\right)$ is cyclically-monotone, as desired.

Now consider any $\pi_1,\ldots,\pi_n \in \frR$ and $(X_1,G_1,\ldots,X_n,G_n)$ with $(X_j,G_j) \sim \pi_j$ for all $j$. 
We show \eqref{eqdef:relation-cyclic-monotonicity} for any permutation $\sigma:[n]\rightarrow[n]$.
Consider any permutation $\sigma:[n]\rightarrow [n]$.
We may write $\mu_{X_j} = \frac1p \sum_{i=1}^p \delta_{x_{j(i)}}$ where $x_{j(1)} \leq x_{j(2)} \leq \cdots \leq x_{j(p)}$ is the ordered enumeration of the support points of $\mu_{X_j}$, and similarly for $G_j$. 
Note that $\{(x_{j(i)},g_{\sigma(j)(i)})\mid i = 1,\ldots,p\}$ is cyclically monotone because $x_{j(i)}$ and $g_{\sigma(j)(i)}$ share the same ordering with respect to $i$, whence by \cite[Theorem 5.10.(ii)]{Villani2008}, we have $\frac1p \sum_{i=1}^p \delta_{(x_{j(i)},g_{\sigma(j)(i)})}$ is the optimal coupling between $\mu_{X_j}$ and $\mu_{G_{\sigma(j)}}$.
Thus, 
\begin{equation}\label{finite-p-ordering-inequality}
    \E[X_jG_{\sigma(j)}] \leq \frac1p \sum_{i=1}^p x_{j(i)}g_{\sigma(j)(i)}.
\end{equation}
Also,
\begin{equation}\label{sum-of-inner-product-ineq}
    \sum_{j=1}^n\frac1p \sum_{i=1}^p x_{j(i)}g_{\sigma(j)(i)} \leq \sum_{j=1}^n\frac1p \sum_{i=1}^p x_{j(i)}g_{j(i)},
\end{equation}
because if we define $\bx_j,\bg_j$ to have coordinates given by $x_{j(i)}$ and $g_{j(i)}$, respectively, then $(\bx_j,\bg_j) \in \partial f$, whence \eqref{sum-of-inner-product-ineq} follows from \eqref{Rp-sub-diff-cyc-mon}.
Finally, we identify the right-hand side of \eqref{sum-of-inner-product-ineq} as $\sum_{j=1}^n \E[X_jG_j]$ because the support of $\mathsf{spt}(\mu_{(X_j,G_j)})$ is cyclically monotone.
Compining the previous three displays, we conclude \eqref{eqdef:relation-cyclic-monotonicity}.

Thus, $\frR$ is jointly-cyclically monotone.

\subsection{Proof of Lemma \ref{lem:gordon-Rp-stability}}\label{app:proof-of-gordon-Rp-stability}

Our main idea is to study instead an embedding of Gordon's optimization problem into $L_2(0,1)$.
For $B,\Theta,G \in L_2(0,1)$, $\sigma \in \reals_{\geq0}$, and $\kappa,\xi \in \reals$, define
\begin{equation}\label{gordon-residual-obj}
    g(B;\Theta,G,\sigma,\kappa,\xi) = \sqrt{\sigma^2 + \frac{\kappa}{\delta}\E[(B-\Theta)^2] - \xi \sqrt{\frac{\E[(B-\Theta)^2]}{\delta}}} - \frac1\delta\E[G(B-\Theta)],
\end{equation}
whenever the arguments to the square-roots are non-negative, and infinite otherwise.
We refer the optimization problem
\begin{equation}\label{approx-Gordon-l2-pert}
    \min_{B} \left\{\frac12 g(B;\Theta,G,\sigma,\kappa,\xi)_+^2 + f(B)\right\}.
\end{equation}
as the \emph{Gordon's $L_2$ optimization}.
We write the objective in \eqref{approx-Gordon-l2-pert} as $\cL(B;\Theta,G,\sigma^2,\kappa,\xi)$. 
The proof proceeds in several steps.\\

\noindent {\bf Step 1: Convexity of Gordon's $L_2$ objective when $\xi = 0$.}
When $\xi = 0$, the arguments to the square-roots in \eqref{gordon-residual-obj} are non-negative, whence the objective is defined everywhere.
Moreover, as we now show, the objective is also convex in $B$, and locally strongly convex. 
\begin{lemma}\label{lem:strong-convexity}
    Consider $\Theta \in L_2(0,1)$.
    Let $h(B) = \sqrt{\sigma^2 + \frac1\delta \E[(B-\Theta)^2]}$. Then $h$ is  $\frac{\sigma^2/\delta}{(\sigma^2 + R^2/\delta)^{3/2}}$-strongly convex on $d_2(B,\Theta) \leq R$.
\end{lemma}

\begin{proof}[Proof of Lemma \ref{lem:strong-convexity}]
    Let $B_t = B + t\Delta$ and $V_t = B_t - \Theta$.
    Then we have
    \begin{align}
        \frac{\mathrm{d}}{\mathrm{d}t} \sqrt{\sigma^2 + \frac1\delta \E[V_t^2]} = \frac{\E[V_t\Delta]/\delta}{\sqrt{\sigma^2 + \frac1\delta \E[V_t^2]}}.
    \end{align}
    Then
    \begin{align}
        \frac{\mathrm{d}^2}{\mathrm{d}t^2} \sqrt{\sigma^2 + \frac1\delta \E[V_t^2]} &= \frac{\E[\Delta^2]/\delta}{\sqrt{\sigma^2 + \frac1\delta \E[V_t^2]}} - \frac{(\E[V_t\Delta]/\delta)^2}{\left(\sigma^2 + \frac1\delta \E[V_t^2]\right)^{3/2}} \\
        &= \left(\sigma^2 + \frac1\delta \E[V_t^2]\right)^{-3/2} \left(\frac{\E[\Delta^2]}{\delta}\left(\sigma^2 + \frac1\delta \E[V_t^2]\right) - \left(\frac{\E[V_t\Delta]}{\delta}\right)^2\right)\\
        &\geq \frac{\sigma^2/\delta}{\left(\sigma^2 + \E[V_t^2]/\delta\right)^{3/2}} \E[\Delta^2],
    \end{align}
    where in the inequality, we have used Cauchy-Schwartz.
    We conclude that $h(B)$ is $\frac{\sigma^2/\delta}{(\sigma^2 + R^2/\delta)^{3/2}}$-strongly convex on $d_2(B,\Theta) \leq R$.
\end{proof}

\noindent {\bf Step 2: Solution to Gordon's $L_2$ population optimization.} We study first the simple case that $\sigma={\sigma^*},\kappa=1,\xi = 0$, and $\Theta = \Theta^*,G=G^*$ for any $\Theta^* \sim \mu_{\btheta},G^* \sim \normal(0,1)$ independent, which we call \emph{Gordon's $L_2$ population optimization.}
For simplicity of notation, we denote $\eta = \frA(\mu^{*{\tau^*}})$.
Fixing such $\Theta^*,G^*$, we denote $B^* = \eta \circ Y^*$. 
By \eqref{fixed-pt-eqns-proof-sec},
\begin{align}\label{gordons-population-value}
    g(B^*;\Theta^*,G^*,\sigma^*,1,0) &= \sqrt{{\sigma^*}  + \frac1\delta\E[(B^* - \Theta^*)^2]} - \frac1\delta\E[G^*(B^* - \Theta^*)]\nonumber\\
    &= \tau^* - \tau^* \left(1 - \frac\delta{\lambda^*}\right) = \frac{\tau^*\delta}{\lambda^*} > 0.
\end{align}
By \eqref{eqdef:L-star},
\begin{equation}\label{L-star-is-minimal}
    \cL(B^*;\Theta^*,G^*,\sigma^*,1,0) = \frac12 \left(\frac{\tau^*\delta}{\lambda*}\right)^2 + f(B^*) = L^*.
\end{equation}
Because $g$ is positive at $B^*,\Theta^*,G^*,{\sigma^*}$, the Fr\`echet derivative of $g_+^2$ with respect to its first argument, denoted by $\nabla_B$, evaluated at $B^*,\Theta^*,G^*,\sigma^*$ is
\begin{align}
    \nabla_B \frac12 g(B;\Theta^*,G^*,{\sigma^*},1,0)_+^2\Big|_{B = B^*} &= g(B^*;\Theta^*,G^*,\sigma^*,1,0)\left(\frac{(B^*-\Theta^*)/\delta}{\sqrt{{\sigma^*}^2 + \frac1\delta \E[(B^* - \Theta^*)^2]}} - \frac1\delta G^*\right)\nonumber \\
    &= \frac{\tau^*\delta}{\lambda^*}\left(\frac{B^* - \Theta^* - \tau^* G^*}{\delta}\right)\nonumber\\
    & = \frac{B^* - \Theta^* - \tau^* G^*}{\lambda^*},
\end{align}
where we have used \eqref{gordons-population-value}.
By \eqref{effective-scalar-estimator}, $B^* = \mathsf{prox}[\lambda^* f](\Theta^* + \tau^* G^*)$, whence
\begin{equation}
    \frac{\Theta^* + \tau G^* - B^*}{\lambda^*} \in \partial f(B^*).
\end{equation}
Combining the previous two displays, $0 \in \partial \left(\frac12 g(B;\Theta^*,G^*,{\sigma^*},1,0)_+^2 + f(B)\right)\Big|_{B = B^*}$, and by convexity
\begin{equation}
    B^* \in \arg\min_B\left\{\frac12 g(B;\Theta^*,G^*,{\sigma^*},1,0)_+^2 + f(B)\right\}.
\end{equation}

\noindent {\bf Step 3: Solution to Gordon's $L_2$ partially perturbed optimization.}
If we only perturb $\Theta$, $G$, and $\sigma$, so that still $\kappa = 1$ and $\xi = 0$, the objective remains convex.
Moreover, in this case the solution to \eqref{approx-Gordon-l2-pert} satisfies the following.
\begin{lemma}\label{lem:meas-gordon-soln}
    If $\kappa = 1,\xi = 0$ and a solution $B_{\mathrm{opt}}$ to \eqref{approx-Gordon-l2-pert} is such that $g(B_{\mathrm{opt}};\Theta,G,\sigma,1,0) > 0$, then $B_{\mathrm{opt}} = \mathsf{prox}\left[\lambda f\right](\Theta + \tau G)$,
    where $\tau^2 = \sigma^2 + \frac1\delta \E[(B_{\mathrm{opt}} - \Theta)^2]$ and $\lambda^{-1} = \frac 1\delta \left(1 - \frac1{\delta \tau} \E[G(B_{\mathrm{opt}} - \Theta)]\right)$.
    In particular, $B_{\mathrm{opt}}$ is $\sigma(\Theta + \tau G)$-measurable.
\end{lemma}

\begin{proof}
    We have 
    \begin{align}
        \nabla_B \frac12 g(B;\Theta,G,\sigma,1,0)_+^2\Big|_{B = B_{\mathrm{opt}}} &= g(B_{\mathrm{opt}};\Theta,G,\sigma,1,0)\left(\frac{(B_{\mathrm{opt}}-\Theta)/\delta}{\sqrt{\sigma^2 + \frac1\delta \E[(B_{\mathrm{opt}} - \Theta)^2]}} - \frac1\delta G\right)\nonumber \\
        &= \left(\tau - \frac1\delta \E[G(B_{\mathrm{opt}} - \Theta)]\right)\left(\frac{B_{\mathrm{opt}} - \Theta - \tau G}{\tau \delta}\right)\nonumber\\
        &= \frac{\tau \delta}{\lambda}\left(\frac{B_{\mathrm{opt}} - \Theta - \tau G}{\tau \delta}\right)\nonumber \\
        & = \frac{B_{\mathrm{opt}} - \Theta - \tau G}{\lambda},
    \end{align}
    where in the second equality we have used the definition of $\tau$, and in the third equality we have used the definition of $\lambda$.
    The KKT conditions give
    \begin{equation}
        \frac{\Theta + \tau G - B_{\mathrm{opt}}}{\lambda} \in \partial f(B_{\mathrm{opt}}),
    \end{equation}
    whence $B_{\mathrm{opt}} = \mathsf{prox}[\lambda f](\Theta + \tau G)$.
    Measurability follows from Proposition \ref{prop:l2-prox-properties}.
\end{proof}

\noindent {\bf Step 4: Derivatives of $\tau$.} 
We will control solutions to Gordon's $L_2$ optimization problem by viewing Gordon's $L_2$ optimization as a perturbation of Gordon's $L_2$ population optimization and controlling the sensitivity of the solutions to perturbations in the problem.
This requires the study of several derivates.
Define 
\begin{equation}\label{eqdef:tau}
    \tau(B;\Theta,\sigma,\kappa,\xi) = \sqrt{\sigma^2 + \frac{\kappa}{\delta}\E[(B-\Theta)^2] - \xi \sqrt{\frac{\E[(B-\Theta)^2]}{\delta}}},
\end{equation}
whenever the argument to the square-root is positive.
For any variable $A$, we denote the Fr\`echet derivative with respect to $A$ by $\nabla_A$.
We have
\begin{gather}
    \nabla_B \tau(B;\Theta,\sigma,\kappa,\xi) = -\nabla_\Theta \tau(B;\Theta,\sigma,\kappa,\xi) =  \frac{(B - \Theta)/\delta}{\tau(B;\Theta,\sigma,\kappa,\xi)},\label{deriv-B-Theta}\\
    \frac{\mathrm{d}}{\mathrm{d\sigma}} \tau(B;\Theta,\sigma,\kappa,\xi) = \frac{\sigma}{\tau(B;\Theta,\sigma,\kappa,\xi)},\label{deriv-sigma}\\
    \frac{\mathrm{d}}{\mathrm{d\kappa}} \tau(B;\Theta,\sigma,\kappa,\xi) = \frac{d_2(B,\Theta)^2/\delta}{2\tau(B;\Theta,\sigma,\kappa,\xi) },\label{deriv-kappa}\\
    \frac{\mathrm{d}}{\mathrm{d\xi}} \tau(B;\Theta,\sigma,\kappa,\xi) = - \frac{d_2(B,\Theta)/\sqrt{\delta}}{2\tau(B;\Theta,\sigma,\kappa,\xi)}.\label{deriv-xi}
\end{gather}

\noindent {\bf Step 5: Lipschitz continuity and strong convexity parameters.}
We define several local Lipschitz and strong converxity parameters which are entirely explicit in $\tau^*,\lambda^*,\delta$.
Let 
\begin{align}\label{big-R1}
    R_1 = \sqrt{\delta({\tau^*}^2 - {\sigma^*}^2)}.
\end{align}
By \eqref{fixed-pt-eqns-proof-sec}, 
\begin{equation}\label{R1-meaning}
    R_1 = d_2(B^*,\Theta^*).
\end{equation}
The following will serve as local Lipschitz constants:
\begin{align}\label{lipschitz-1}
    L_1^{(\Theta)} = &L_1^{(B)} = \frac{4R_1\tau^*}{\sigma^*\delta}, & L_1^{(\sigma)} &= \sigma^*, & L_1^{(\kappa)} &= \frac{4R_1^2}{\delta\sigma^*}, & L_1^{(\xi)} &= \frac{4R_1}{\sqrt{\delta}},
\end{align}
and
\begin{align}\label{lipschitz-2}
    L_2^{(\Theta)} = &L_2^{(B)} = L_1^{(\Theta)} + 2\tau^*, & L_2^{(\sigma)}, &= L_1^{(\sigma)}, 
    & L_2^{(\kappa)} &= L_1^{(\kappa)}, & L_2^{(\xi)} &= L_1^{(\xi)}, &
    L_2^{(G)} &= \frac{2R_1}{\delta},
\end{align}
and
\begin{equation}\label{lipschitz-3}
\begin{gathered}
    L_3^{(\Theta)} = \frac{R_1 + \tau^*}{\delta\tau^*}L_2^{(\Theta)} + \frac{6(R_1+\tau^*)}{\lambda^*\tau^*} L_1^{(\Theta)} + \frac{3(\tau^* + 2 L_1^{(\Theta)})}{\lambda^*}, \\
    L_3^{(G)} = \frac{R_1 + \tau^*}{\delta\tau^*} L_2^{(G)} + \frac{3\tau^*}{\lambda^*}, \qquad\qquad
    L_3^{(\sigma)} = \frac{R_1 + \tau^*}{\delta\tau^*} L_2^{(\sigma)} + \frac{6(R_1 + \tau^*)}{\lambda^* \tau^*} L_1^{(\sigma)} + \frac{6}{\lambda^*} L_1^{(\sigma)}.
\end{gathered}
\end{equation}
The following will enter our definition of ``local'':
\begin{align}\label{radius-1}
    r_1^{(\Theta)} &= r_1^{(B)} = \frac{R_1}{3\tau^*}, &  r_1^{(G)} &= 1 & r_1^{(\sigma)} &= 1/2, &
    r_1^{(\kappa)} &= 1, & r_1^{(\xi)} &= \frac{\sigma^*\sqrt{\delta}}{16R_1} ,
\end{align}
and for $A \in \{\Theta,B,G,\sigma,\kappa,\xi\}$ (viewed as a symbol and a collection of symbols), let 
\begin{align}\label{radius-2}
    r_2^{(A)} = 
    \min\left\{ r_1^{(A)},\, \frac{\tau^*}{10L_1^{(A)}},\,\frac{\tau^*\delta}{12L_2^{(A)}\lambda^*}\right\}.
\end{align}
For $\br = (r^{(\Theta)},r^{(B)},r^{(G)},r^{(\sigma)},r^{(\kappa)},r^{(\xi)})$, we say $(\Theta,B,G,\sigma,\kappa,\xi)\in N_{\br}$ if
\begin{subequations}\label{neighborhood}
\begin{align}
    d_2\left(\frac{\Theta}{\tau^*},\frac{\Theta^*}{\tau^*}\right) &< r^{(\Theta)}, & d_2\left(\frac{B}{\tau^*},\frac{B^*}{\tau^*}\right) &< r^{(B)},&
    d_2(G,G^*) &< r^{(G)}, \\
    \left|\frac{\sigma}{\sigma^*} - 1\right| &<r^{(\sigma)}, 
     & |\kappa - 1| &< r^{(\kappa)}, & \frac{|\xi|}{\sigma^*} &< r^{(\xi)}.
\end{align}
\end{subequations}
With some abuse of notation, we say an ordered tuple of a subset of these variables is in $N_{\br}$ if the relevant inequalities apply to this subset.
For example, $(\kappa,\xi) \in N_{\br}$ means $|\kappa - 1| < r^{(\kappa)}$ and $\frac{|\xi|}{\sigma^*} < r^{(\xi)}$.
The following will be a local strong-convexity parameter
\begin{equation}\label{strong-convexity}
    K = \frac{{\tau^*}^3\delta}{2\lambda^*} \frac{{\sigma^*}^2/(4\delta)}{({\sigma^*}^2/4 + 4R_1^2/\delta)^{3/2}}.
\end{equation}
Finally, we define one more set of Lipschitz constants.
Let
\begin{equation}\label{M-def}
H = \sqrt{\frac{5{\sigma^*}^2}2 + \frac{4(1 + r_2^{(\kappa)})R_1^2}{\delta}} + \frac{4R_1}{\delta}, 
\end{equation}
and let
\begin{align}\label{lipschitz-4}
L_4^{(\kappa)} &= H L_1^{(\kappa)}, & L_4^{(\xi)} &= HL_1^{(\xi)}.
\end{align}

\noindent {\bf Step 6: Lipschitz continuity and strong convexity properties.}
Using the parameters we defined in \eqref{radius-1} and \eqref{radius-2}, we let $\br_i = (r^{(\Theta)}_i,r^{(B)}_i,r^{(G)}_i,r^{(\sigma)}_i,r^{(\kappa)}_i,r^{(\xi)}_i)$.
Our next lemma establishes the regularity properties we will need.
\begin{lemma}\label{lem:regularity}
    If $(\Theta,B,G,\sigma,\kappa,\xi) \in N_{\br_1}$, then
    \begin{enumerate}[(i)]
        \item ($\tau$ is locally Lipchitz)
            \begin{align}\label{tau-lipschitz-1}
                &|\tau(B;\Theta,\sigma,\kappa,\xi) - \tau(B^*;\Theta^*,\sigma^*,1,0)|\nonumber\\
                &\quad\leq L_1^{(\Theta)} d_2\left(\frac{\Theta}{\tau^*},\frac{\Theta^*}{\tau^*}\right) + L_1^{(B)} d_2\left(\frac{B}{\tau^*},\frac{B^*}{\tau^*}\right) + L_1^{(\sigma)}\left|\frac{\sigma}{\sigma^*} - 1\right| + L_1^{(\kappa)} |\kappa - 1| + L_1^{(\xi)} \frac{|\xi|}{\sigma^*}.
            \end{align}
        \item ($g$ is locally Lipschitz)
        \begin{align}\label{g-lipschitz-1}
            &|g(B;\Theta,G,\sigma,\kappa,\xi) - g(B^*;\Theta^*,G^*,\sigma^*,1,0)| \nonumber\\
            &\quad \leq L_2^{(\Theta)} d_2\left(\frac{\Theta}{\tau^*},\frac{\Theta^*}{\tau^*}\right) + L_2^{(B)} d_2\left(\frac{B}{\tau^*},\frac{B^*}{\tau^*}\right) + L_2^{(G)} d_2(G,G^*) \nonumber\\
            &\qquad\qquad\qquad\qquad\qquad\qquad+ L_2^{(\sigma)}\left|\frac{\sigma}{\sigma^*} - 1\right| + L_2^{(\kappa)} |\kappa - 1| + L_2^{(\xi)} \frac{|\xi|}{\sigma^*}.
        \end{align}
    \end{enumerate}  
    If $(\Theta,B,G,\sigma,\kappa,\xi) \in N_{\br_2}$, then 
    \begin{enumerate}[(i)]
        \setcounter{enumi}{2}
        \item ($\tau$ and $g$ are locally bounded)
        \begin{align}\label{tau-g-locally-bounded}
            \frac{3\tau^*}2 > |\tau(B;\Theta,\sigma,\kappa,\xi)| &> \frac{\tau^*}2,  & \frac{3\tau^*\delta}{2\lambda^*} > |g(B;\Theta,G;\sigma,\kappa,\xi)| &> \frac{\tau^*\delta}{2\lambda^*}.
        \end{align}
        \item (The subgradient at $B^*$ is locally Lipshitz)
        There exists $D \in \partial \cL(B;\Theta,G,\sigma,1,0)\Big|_{B = B^*}$ such that
        \begin{equation}\label{subgrad-locally-lipschitz}
            \|D\|_2 < L_3^{(\Theta)} d_2\left(\frac{\Theta}{\tau^*},\frac{\Theta^*}{\tau^*}\right) + L_3^{(G)}d_2(G,G^*) + L_3^{(\sigma)}\left|\frac{\sigma}{\sigma^*} - 1\right|.
        \end{equation}
        \item (The objective is locally strongly convex in $B$)
        \begin{equation}\label{L-locally-strongly-convex}
            \text{$B/\tau^* \mapsto \frac12 g(B;\Theta;G,\sigma,1,0)_+^2 = \frac12 g(B;\Theta;G,\sigma,1,0)^2$ is $K$-strongly convex};
        \end{equation}
        \item ($\cL$ is locally Lipschitz in $\Theta,G,\sigma,\kappa,\xi$)
        \begin{align}\label{L-locally-lipschitz}
            &|\cL(B^*;\Theta,G,\sigma,\kappa,\xi) - \cL(B^*;\Theta^*,G^*,\sigma^*,1,0)| \nonumber\\
            &\quad \leq \frac{3\tau^*\delta}{2\lambda^*}\left(L_2^{(\Theta)} d_2\left(\frac{\Theta}{\tau^*},\frac{\Theta^*}{\tau^*}\right) + L_2^{(G)} d_2(G,G^*) + L_2^{(\sigma)}\left|\frac{\sigma}{\sigma^*} - 1\right| + L_2^{(\kappa)} |\kappa - 1| + L_2^{(\xi)} \frac{|\xi|}{\sigma^*}\right).
        \end{align}
    \end{enumerate}
    If $(\Theta,G,\sigma,\kappa,\xi) \in N_{\br_1}$ and $d_2\left(\frac{B}{\tau^*},\frac{\Theta^*}{\tau^*}\right) < \frac{4R_1}{3\tau^*}$, then
    \begin{enumerate}[(i)]
        \setcounter{enumi}{6}
        \item 
        ($\cL$ locally Lipschitz in $\kappa,\xi$ in expanded $B$ neighborhood)
        \begin{equation}\label{L-locally-Lipshitz-big-B-neighborhood}
            |\cL(B;\Theta,G,\sigma,\kappa,\xi) - \cL(B;\Theta,G,\sigma,1,0)| \leq L_4^{(\kappa)}|\kappa - 1| + L_4^{(\xi)}\frac{|\xi|}{\sigma^*}.
        \end{equation}  
    \end{enumerate}
\end{lemma}

\begin{proof}[Proof of Lemma \ref{lem:gordon-local-stability}]
    We prove each result in order.
    \begin{enumerate}[(i)]
        \item 
        First, we consider $(\Theta,B,G,\sigma,\kappa,\xi) \in N_{\br_1}$.
        On this neighborhood, we have 
        \begin{equation}\label{tau-lower-bound-1}
        \tau(B,\Theta,\sigma,\kappa,\xi) \geq \sigma^* /2
        \end{equation}
        because $\left|\frac{\sigma}{\sigma^*} - 1\right| < \frac{1}2$ by \eqref{radius-1}, \eqref{neighborhood}.
        Also
        \begin{equation}\label{R-to-Theta-bound-neighborhood-1}
            d_2(B,\Theta) \leq 2R_1,
        \end{equation}
        by \eqref{R1-meaning}, \eqref{radius-1}, \eqref{neighborhood}, and the triangle inequality.
        Then 
        \begin{align}
            |\tau(B^*;\Theta,\sigma,\kappa,\xi) - \tau(B^*;\Theta^*,\sigma^*,1,0)| &\leq |\tau(B;\Theta^*,\sigma^*,1,0) - \tau(B^*;\Theta^*,\sigma^*,1,0)|\nonumber \\
            &\quad +|\tau(B;\Theta,\sigma^*,1,0) - \tau(B;\Theta^*,\sigma^*,1,0)|\nonumber\\
            &\quad + |\tau(B;\Theta,\sigma,1,0) - \tau(B;\Theta,\sigma^*,1,0) |\nonumber\\
            &\quad + |\tau(B;\Theta,\sigma,\kappa,0) - \tau(B;\Theta,\sigma,1,0) |\nonumber\\
            &\quad + |\tau(B;\Theta,\sigma,\kappa,\xi) - \tau(B;\Theta,\sigma,\kappa,0) |.
        \end{align}
        The first two lines are bounded by 
        \begin{equation}
            \frac{4R_1 \tau^*}{\sigma^*\delta} \left(d_2\left(\frac{\Theta}{\tau^*},\frac{\Theta^*}{\tau^*}\right) + d_2\left(\frac{B}{\tau^*},\frac{B^*}{\tau^*}\right)\right) = L_1^{(\Theta)} d_2\left(\frac{\Theta}{\tau^*},\frac{\Theta^*}{\tau^*}\right) + L_1^{(B)} d_2\left(\frac{B}{\tau^*},\frac{B^*}{\tau^*}\right)
        \end{equation}
        because the derivative of $\tau$ with respect to $\Theta$ or $B$ on $N_{\br_1}$ has $\ell_2$ norm bounded by $\frac{d_2(\Theta,B)}{\sigma^*\delta/2}<\frac{4R_1}{\sigma^*\delta}$ by \eqref{deriv-B-Theta}, \eqref{tau-lower-bound-1}, and \eqref{R-to-Theta-bound-neighborhood-1}, and we may substitute the Lipschitz constants by \eqref{lipschitz-1}.
        The third line is bounded by
        \begin{equation}
            |\sigma - \sigma^*| = L_1^{(\sigma)}\left|\frac{\sigma}{\sigma^*} - 1\right|
        \end{equation}
        because the derivative of $\tau$ with respect to $\sigma$ is bounded by 1 by \eqref{deriv-sigma} (because $\tau(B,\Theta,\sigma,\kappa,\xi)  > \sigma$ always), and we may substitute the Lipschitz constants by \eqref{lipschitz-1}.
        The fourth line is bounded by
        \begin{equation}
            \frac{4R_1^2}{\delta \sigma^*}|\kappa - 1| = L_1^{(\kappa)} |\kappa - 1|
        \end{equation}
        because the derivative of $\tau$ with respect to $\kappa$ is bounded by
        \begin{equation}\label{deriv-tau-by-kappa}
            \frac{4R_1^2/\delta}{2\sigma^*/2} = \frac{4R_1^2}{\delta \sigma^*}
        \end{equation}
        by \eqref{deriv-kappa}, \eqref{tau-lower-bound-1}, and \eqref{R-to-Theta-bound-neighborhood-1}, and we may substitute the Lipschitz constants by \eqref{lipschitz-1}.
        Finally, the fifth line is bounded by  
        \begin{equation}
            \frac{4R_1}{\sqrt{\delta}} \frac{|\xi|}{\sigma^*} = L_1^{(\xi)} \frac{|\xi|}{\sigma^*}
        \end{equation}
        because the derivative of $\tau$ with respect to $\xi$ is bounded in absolute value by 
        \begin{equation}
            \frac{2R_1/\sqrt{\delta}}{2\sqrt{{\sigma^*}^2/4 - \frac{\sigma^*\sqrt{\delta}}{16 R_1} \cdot \frac{2R_1}{\sqrt{\delta}}  }} < \frac{4R_1}{\sqrt{\delta}\sigma^*}, 
        \end{equation}
        by \eqref{eqdef:tau}, \eqref{deriv-xi}, \eqref{radius-1}, \eqref{neighborhood}, and \eqref{R-to-Theta-bound-neighborhood-1}, and we may substitute the Lipschitz constants by \eqref{lipschitz-1}.
        Summing the preceding five bounds, we conclude \eqref{tau-lipschitz-1}.

        \item 
        By \eqref{radius-1} and \eqref{neighborhood} that on $N_{\br_1}$ we have $\|G\|_2 \leq 2$, whence by \eqref{R-to-Theta-bound-neighborhood-1}, we have
        \begin{equation}
            |\E[G(B-\Theta)]/\delta - \E[G^*(B^*-\Theta^*)]/\delta| \leq \frac{2R_1}{\delta} d_2(G,G^*) + 2\tau^* d_2\left(\frac{B}{\tau^*},\frac{B^*}{\tau^*}\right) + 2\tau^* d_2\left(\frac{\Theta}{\tau^*},\frac{\Theta^*}{\tau^*}\right).
        \end{equation}
        Combined with  \eqref{lipschitz-2} and \eqref{tau-lipschitz-1}, we get \eqref{g-lipschitz-1}.

        \item 
        Observe that $(B,\Theta,G,\sigma,\kappa,\xi) \in N_{\br_2}$ implies $(B,\Theta,G,\sigma,\kappa,\xi) \in N_{\br_1}$ because $r_2^{(A)} \leq r_1^{(A)}$ for $A \in \{\Theta,B,G,\sigma,\kappa,\xi\}$.
        We may combine \eqref{tau-lipschitz-1} and \eqref{g-lipschitz-1} with \eqref{radius-2} to get \eqref{tau-g-locally-bounded}.

        \item 
         By \eqref{radius-1} and \eqref{neighborhood}, on $N_{\br_1} \subset N_{\br_2}$ we have $\|G\|_2 \leq 2$. Thus,
        \begin{equation*}
            d_2(B-\Theta-\tau G,B^*-\Theta^* - \tau^*G^*) \leq \tau^* d_2\left(\frac{B}{\tau^*},\frac{B^*}{\tau^*}\right) + \tau^* d_2\left(\frac{\Theta}{\tau^*},\frac{\Theta^*}{\tau^*}\right) + \tau^* d_2(G,G^*) + 2|\tau - \tau^*|.
        \end{equation*}
        Also, by \eqref{tau-g-locally-bounded} both $g$ and $\tau$ are positive on $N_{\br_2}$, so that
        \begin{equation}
            \nabla_B \frac12 g(B;\Theta,G,\sigma,1,0)_+^2\Bigg|_{B = B^*} =  \frac{g(B^*;\Theta,G,\sigma,1,0)}{\delta \tau(B^*;\Theta,\sigma,1,0)} ( B - \Theta - \tau(B^*;\Theta,\sigma,1,0) G).
        \end{equation}
        Denoting $g^* = g(B^*;\Theta^*,G^*,\sigma^*,1,0)$ and recalling $\tau^* = \tau(B^*;\Theta^*,G^*,\sigma^*,1,0)$, observe that 
        \begin{align}
            &\left\| \frac{g^*}{\delta \tau^*} (B^* - \Theta^* - \tau^* G^*) - \frac{g(B^*;\Theta,G,\sigma,1,0)}{\delta \tau(B^*;\Theta,\sigma,1,0)} (B^* - \Theta - \tau(B^*;\Theta,\sigma,1,0) G) \right\|\nonumber\\
            &\quad \leq |g^* - g(B^*;\Theta,G,\sigma,1,0)| \frac{\|B^* - \Theta^* - \tau^* G^*\|}{\delta \tau^*} \nonumber\\
            &\quad\quad + |g(B^*;\Theta,G,\sigma,1,0)|\frac{\|B^* - \Theta^* - \tau^* G^*\|}{\delta}\left|\frac1{\tau^*} - \frac1{\tau(B^*;\Theta,\sigma,1,0)}\right| \nonumber\\
            &\quad\quad + \frac{|g(B^*;\Theta,G,\sigma,1,0)|}{\delta\tau(B^*;\Theta,\sigma,1,0)} \Big(d_2(\Theta,\Theta^*) + \tau^*d_2(G,G^*) + \|G\| |\tau^* - \tau(B^*;\Theta,\sigma,1,0)|\Big)\nonumber\\
            &\quad\leq  \left( L_2^{(\Theta)} d_2\left(\frac{\Theta}{\tau^*},\frac{\Theta^*}{\tau^*}\right) + L_2^{(G)} d_2(G,G^*) + L_2^{(\sigma)}\left|\frac{\sigma}{\sigma^*} - 1\right| \right)\frac{R_1 + \tau^*}{\delta \tau^*}\nonumber\\
            &\quad\quad +\frac{3\tau^*\delta}{2\lambda^*} \frac{R_1+\tau^*}{\delta} \frac{4}{{\tau^*}^2} \left(L_1^{(\Theta)} d_2\left(\frac{\Theta}{\tau^*},\frac{\Theta^*}{\tau^*}\right) + L_1^{(\sigma)}\left|\frac{\sigma}{\sigma^*} - 1\right|\right) \nonumber\\
            &\quad\quad+ \frac{3\tau^*\delta}{2\lambda^*} \frac{2}{\delta\tau^*}\left(\left(\tau^* + 2L_1^{(\Theta)}\right)d_2\left(\frac{\Theta}{\tau^*},\frac{\Theta^*}{\tau^*}\right) + \tau^*d_2(G,G^*) + 2L_1^{(\sigma)} \left|\frac{\sigma}{\sigma^*}-1\right|\right)\nonumber\\
            &\quad \leq L_3^{(\Theta)} d_2\left(\frac{\Theta}{\tau^*},\frac{\Theta^*}{\tau^*}\right) + L_3^{(G)}d_2(G,G^*) + L_3^{(\sigma)}\left|\frac{\sigma}{\sigma^*} - 1\right|,
        \end{align}
        where in the first inequality we have used the triangle inequality; in the first line of the second inequality we have used \eqref{R1-meaning} and \eqref{g-lipschitz-1}; in the second line of the second inequality we have used \eqref{tau-g-locally-bounded}, \eqref{R1-meaning}, \eqref{tau-lipschitz-1}, and the fact that the derivative of $\tau \mapsto 1/\tau$ is $1/\tau^2$ which is bounded by $4/{\tau^*}^2$ on $N_{\br^2}$; in the third line of the second inequality we have used \eqref{tau-g-locally-bounded} and \eqref{tau-lipschitz-1}; and in the last equality we have used \eqref{lipschitz-3}.
        Because $0 \in \partial \cL(B;\Theta^*,G^*,{\sigma^*}^2,1,0)\Big|_{B = B^*}$ and the only part of the objective $\cL$ which depends upon $\Theta,G,\sigma$ is given by $\frac12 g_+^2$, we conclude \eqref{subgrad-locally-lipschitz}.

        \item 
        Combining Lemma \ref{lem:strong-convexity}, \eqref{tau-g-locally-bounded}, and \eqref{strong-convexity}, we conclude \eqref{L-locally-strongly-convex}.

        \item 
        Because the derivative of $x \mapsto \frac12 x^2$ is $x$, we have by \eqref{g-lipschitz-1}, \eqref{tau-g-locally-bounded}, the chain rule, and the fact that the only part of $\cL$ which depends upon $\Theta,G,\sigma,\kappa,\xi$ is $\frac12 g_+^2$, that \eqref{L-locally-lipschitz}.

        \item 
        Because $\left|\frac{\sigma}{\sigma^*} -1 \right| < r_1^{(\sigma)} = \frac12$, $|\kappa - 1| < r_1^{(\kappa)}$, $d_2\left(\frac{B}{\tau^*},\frac{\Theta}{\tau^*}\right) < d_2\left(\frac{B}{\tau^*},\frac{\Theta^*}{\tau^*}\right) + d_2\left(\frac{\Theta^*}{\tau^*},\frac{\Theta}{\tau^*}\right) < \frac{4R_1}{3\tau^*} + r_1^{(\Theta)} < \frac{2R_1}{\tau^*}$, $\|G\|_2 < 1 + r_1^{(G)} =  2$, and $\frac{|\xi|}{\sigma^*} < r_1^{(\xi)} = \frac{\sigma^*\sqrt{\delta}}{16R_1}$, we have by \eqref{gordon-residual-obj} that 
        \begin{equation}
            |g(B;\Theta,G,\sigma^2,\kappa,\xi)| < \sqrt{\frac{9{\sigma^*}^2}{4} + (1+r_1^{(\kappa)})\frac{4R_1^2}{\delta} + \frac{{\sigma^*}^2\sqrt{\delta}}{16R_1}\frac{2R_1}{\sqrt{\delta}}} + 2 \frac{2R_1}{\delta} < H,
        \end{equation}
        where in the last inequality we apply \eqref{M-def}.
        Because the derivative of $x \mapsto \frac12 x^2$ is $x$, we get by \eqref{g-lipschitz-1}
        \begin{align}
            |\cL(B;\Theta,G,\sigma,\kappa,\xi) - \cL(B;\Theta,G,\sigma,1,0)| &\leq  HL_1^{(\kappa)}|\kappa - 1| + HL_1^{(\xi)} \frac{|\xi|}{\sigma^*} \nonumber\\
            &\leq L_4^{(\kappa)}|\kappa - 1| + L_4^{(\xi)} \frac{|\xi|}{\sigma^*},
        \end{align}
        where in the last line we have used \eqref{lipschitz-4}.
    \end{enumerate}
    The proof is complete.
\end{proof}

\noindent {\bf Step 7: Locally stable solution to Gordon's $L_2$ optimization.}
We return to studying Gordon's $L_2$ optimization \eqref{approx-Gordon-l2-pert}.
By viewing it as a perturbation of Gordon's $L_2$ population optimization (Step 2), we can bound the values of $\cL$ close to and far from $B^*$.

\begin{lemma}\label{lem:gordon-local-stability}
    There exists $c_1',c_2',c_3',L > 0$ depending only on $\tau^*,\lambda^*,\delta$ such that for
    \begin{equation}\label{choose-epsilon-small}
        0 < \epsilon < c_1',
    \end{equation}
    if
    \begin{equation}\label{gordon-stability-conds}
        \max\left\{d_2\left(\frac{\Theta}{\tau^*},\frac{\Theta^*}{\tau^*}\right),\,d_2\left( G,G^*\right),\,\left|\frac{\sigma}{\sigma^*} - 1\right|,\,|\kappa-1|,\,\frac{|\xi|}{\sigma^*}\right\} < \epsilon
    \end{equation}
    then 
    \begin{equation}\label{gordon-V*-stability}
        \cL(B^*;\Theta,G,\sigma,\kappa,\xi) < L^* + 5L\epsilon,
    \end{equation}
    and
    \begin{equation}\label{gordon-far-away-stability}
        \inf_{\substack{d_2\left(\frac{B}{\tau^*},\frac{\Theta^*}{\tau^*}\right)\leq c_2'\\ d_2\left(\frac{B}{\tau^*},\frac{B^*}{\tau^*}\right) \geq c_3'\sqrt{\epsilon}}}\cL(B;\Theta,G,\sigma,\kappa,\xi) > L^* + 10L\epsilon.
    \end{equation}  
\end{lemma}

\noindent In fact, the constants $c_1',c_2',c_3',L$ are explicitly
\begin{gather}
    L = \max\left\{ \frac{3\tau^*\delta}{2\lambda^*} \max\left\{L_2^{(\Theta)},\,L_2^{(G)},\,L_2^{(\sigma)},\,L_2^{(\kappa)},\,L_2^{(\xi)}\right\},\, \tau^* \max\left\{L_3^{(\Theta)}, L_3^{(G)}, L_3^{(\sigma)}\right\},\,\max\left\{L_4^{(\kappa)},L_4^{(\xi)}\right\}\right\},\nonumber\\
    r_{\mathrm{min}} = \min\left\{r_j^{(A)}\mid j \in \{1,2\},\, A \in \{\Theta,B,G,\sigma,\kappa,\xi\}\right\},\nonumber\\
    c_1' = \min\left\{r_{\mathrm{min}} , \frac{Kr_{\mathrm{min}}^2}{60L} , \frac{5K}{12L} \right\},\qquad\qquad
    c_2' = \frac{4R_1}{3\tau^*},\qquad\qquad
    c_3' = \sqrt{\frac{60L}{K}}.\label{L-rmin-c-defs}
\end{gather}
Note that with these choices, we have that whenever $\epsilon < c_1'$
\begin{equation}\label{sqrt-epsilon-bound}
    c_3'\sqrt{\epsilon} < c_1',
\end{equation}
because $c_1' \leq \frac{Kr_{\mathrm{min}}^2}{60 L}$.

\begin{proof}[Proof of Lemma \ref{lem:gordon-local-stability}]
    Note that because $\epsilon < c_1' \leq r_{\mathrm{min}}$, Eq.~\eqref{gordon-stability-conds} implies $(B^*,\Theta,G,\sigma,\kappa,\xi) \in N_{\br_2}$, whence we have by \eqref{L-locally-lipschitz} and substituting \eqref{L-star-is-minimal} that \eqref{gordon-V*-stability} holds.

    Moreover, for $(B,\Theta,G,\sigma,\kappa,\xi) \in N_{\br_2}$, we have by \eqref{subgrad-locally-lipschitz} and \eqref{L-locally-strongly-convex} that there exists $D \in \partial \cL(B;\Theta,G,\sigma,1,0)\Big|_{B = B^*}$ such that
    \begin{align}
        \cL(B;\Theta,G,\sigma,1,0) - \cL(B^*;\Theta,G,\sigma,1,0) &\geq \E[(B-B^*)D] + \frac{K}2 d_2\left(\frac{B}{\tau^*},\frac{B^*}{\tau^*}\right)^2\nonumber\\
        & \geq - 3L\epsilon d_2\left(\frac{B}{\tau^*},\frac{B^*}{\tau^*}\right) + \frac{K}2 d_2\left(\frac{B}{\tau^*},\frac{B^*}{\tau^*}\right)^2,
    \end{align}
    where in the second inequality we have applied the definition of $L$ \eqref{L-rmin-c-defs} and Cauchy-Schwartz.
    By \eqref{sqrt-epsilon-bound}, we may apply the previous display to all $B$ with 
    $d_2\left(\frac{B}{\tau^*},\frac{B^*}{\tau^*}\right) = c_3'\sqrt{\epsilon}$, whence for such $B$
    \begin{equation}\label{gordon-far-away-cvx-stability}
        \cL(B;\Theta,G,\sigma,1,0) - \cL(B^*;\Theta,G,\sigma,1,0) \geq -3L\epsilon\sqrt{\frac{60L}{K}\epsilon} + 30 L \epsilon \geq 15L\epsilon,
    \end{equation}
    where the last inequality holds because $3L\epsilon\sqrt{\frac{60L}{K}\epsilon} < 15L\epsilon$ because $\epsilon < c_1' < \frac{5K}{12L}$.
    By convexity in $B$ (which holds because $\xi = 0$), \myeqref{gordon-far-away-cvx-stability} holds also for any $B$ with $d_2\left(\frac{B}{\tau^*},\frac{B^*}{\tau^*}\right) \geq  c_3'\sqrt{\epsilon}$.

    Now consider any $(\Theta,G,\sigma,\kappa,\xi)$ satisfying \eqref{gordon-stability-conds} and $B$ satisfying both $c_2' \geq d_2\left(\frac{B}{\tau^*},\frac{\Theta^*}{\tau^*}\right)$ and $d_2\left(\frac{B}{\tau^*},\frac{B^*}{\tau^*}\right) \geq c_3'\sqrt{\epsilon}$.
    Then
    \begin{align*}
        \cL(B;\Theta,G,\sigma,\kappa,\xi) - \cL(B^*;\Theta^*,G^*,\sigma^*,1,0) &\geq \cL(B;\Theta,G,\sigma,1,0) - \cL(B^*;\Theta,G,\sigma,1,0) \\
        &\quad- L_4^{(\kappa)}|\kappa - 1| - L_4\frac{|\xi|}{\sigma^*}\\
        &\quad- \frac{3\tau^*\delta}{2\lambda^*} \left(L_2^{(\Theta)}d_2\left(\frac{\Theta}{\tau^*},\frac{\Theta^*}{\tau^*}\right) + L_2^{(G)}d_2(G,G^*) + L_2^{(\sigma)}\left|\frac{\sigma}{\sigma^*}-1\right|\right)\\
        &\geq \cL(B;\Theta,G,\sigma,1,0) - \cL(B^*;\Theta,G,\sigma,1,0) - 5 L\epsilon \\ 
        &\geq 15L\epsilon - 5L\epsilon = 10L\epsilon,
    \end{align*}
    where in the first inequality we have applied \eqref{L-locally-Lipshitz-big-B-neighborhood}, which is permitted because $d_2\left(\frac{B}{\tau^*},\frac{\Theta^*}{\tau^*}\right) \leq c_2' = \frac{4R_1}{3\tau^*}$, and \eqref{L-locally-lipschitz}, which is permitted because $(B^*,\Theta,G,\sigma,\kappa,\xi) \in N_{\br_2}$; in the second inequality we have used the definition of $L$ in \eqref{L-rmin-c-defs}; and in the third inequality we have used \eqref{gordon-far-away-cvx-stability}, which is permited becuase $d_2\left(\frac{B}{\tau^*},\frac{B^*}{\tau^*}\right) \geq c_3'\sqrt{\epsilon}$.
    Substituting \eqref{L-star-is-minimal} yields \eqref{gordon-far-away-stability}.
\end{proof}

\noindent {\bf Step 7: Control of Gordon's $\reals^p$ perturbed optimization.}
We complete the proof of Lemma \ref{lem:gordon-Rp-stability} by viewing the objective \eqref{gordon-objective-finite-p} as a restricted version of the objective $\cL$ and using the control over the latter objective established in Lemma \ref{lem:gordon-local-stability}.
Throughout this section, we take $L,r_{\mathrm{min}},c_1',c_2',c_3'$ as defined by \eqref{L-rmin-c-defs}.

Let $\kappa = \frac{\|\bh\|^2}{n}$, $\sigma^2 = \frac{\|\bw\|^2}{n}$, $\xi = \left\langle \frac{\bh}{\sqrt n} , \frac{\bw}{\sqrt n} \right \rangle$, and $\delta = n/p$.
We may then write 
\begin{equation}
    L(\bv) = \frac12 \left(\sqrt{\sigma^2 + \frac\kappa\delta \|\bv\|^2 - \xi \sqrt{\frac{\|\bv\|^2}{\delta p}}} - \frac1{\delta}\frac{\bg^\mathsf{T}\bv}{p}\right)_+^2 + f_p(\btheta + \bv).
\end{equation}
Let $\bb = \btheta + \bv$, and we denote a reparametrized version of the objective
\begin{equation}\label{gordon-Rp-b-form}
    L'(\bb) = \frac12 \left(\sqrt{\sigma^2 + \frac\kappa\delta \frac{\|\bb - \btheta\|^2}{p} - \frac{\xi}{\sqrt \delta}\|\bb - \btheta\|} - \frac1{\delta}\frac{\bg^\mathsf{T}(\bb - \btheta)}{p}\right)_+^2 + f_p(\bb).
\end{equation}
By Lemma \ref{lem:embeddings}.(a), there exists $\Theta^*,G^*,\Theta,G \in L_2(0,1)$ such that $\Theta^* \sim \mu_{\btheta}$, 
$G^* \sim \normal(0,1)$ independent, 
$(\Theta,\tau^* G) \sim \mu_{(\btheta,\tau \bg)}$, 
and $d_2\left(\frac{\Theta}{\tau^*},\frac{\Theta^*}{\tau^*}\right)^2 + d_2( G, G^*)^2 = W_2(\mu_{(\btheta/\tau^*, \bg)},\mu_{\btheta/\tau^*} \otimes \normal(0,1))^2$.
Because $\btheta,\bg \in \reals^p$, we have that $\Theta,G$ are measurable with respect to the $\sigma$-algebra generated by a partition consisting of $p$ atomic sets each with probability $1/p$.
We denote this $\sigma$-algebra by $\cI_p$, and we denote the $\cI_p$-measurable random variables in $L_2(0,1)$ by $m\cI_p$.
Let $\iota: \reals^p \rightarrow L_2(0,1)$ be the linear isomorphism of the form \eqref{iota-def} between $\reals^p$ and $m\cI_p$ which takes $\btheta$ to $\Theta$ and $\bg$ to $G$.
Observe that for all $\bb$,
\begin{equation}\label{gordon-embedding}
    L'(\bb) = \cL(\iota(\bb);\iota(\btheta),\iota(\bg),\sigma,\kappa,\xi),
\end{equation}
where we have used that $f_p(\bb) = f(\iota(B))$ because $f$ is an $L_2$ embedding of $pf_p$ (see \eqref{embed-finite-penalty}).


Denote $B^* = \eta \circ (\Theta^* + \tau^* G^*)$.
With $c_1'$ as in Lemma \ref{lem:gordon-local-stability}, assume there exists $0 < \epsilon < c_1'$ such that 
\begin{equation}\label{choice-of-epsilon}
    W_2(\mu_{(\btheta/\tau^*, \bg)},\mu_{\btheta/\tau^*} \times \normal(0,1)) < \epsilon \quad \text{and} \quad \left|\frac{\sigma}{\sigma^*}-1\right|, |\kappa - 1|,\frac{|\xi|}{\sigma^*} < \epsilon.
\end{equation}
In particular, then $d_2\left(\frac{\Theta}{\tau^*},\frac{\Theta^*}{\tau^*}\right) < \epsilon$ and $d_2(G,G^*) < \epsilon$.
By Lemma \ref{lem:gordon-local-stability}, Eq.~\eqref{gordon-V*-stability}, we have
\begin{equation}\label{perturbed-value-at-B}
    \min_{B \in L_2(0,1)} \cL(B;\Theta,G,\sigma,1,0) \leq \cL(B^*;\Theta,G,\sigma,1,0) < L^* + 5L\epsilon.
\end{equation}
By Lemma \ref{lem:regularity} and \ref{lem:gordon-local-stability}, the objective $\cL(\cdot;\Theta,G,\sigma,1,0)$ is strongly convex and lower semi-continuous on $d_2\left(\frac{B}{\tau^*},\frac{B^*}{\tau^*}\right) < c_3'\sqrt{\epsilon}$, a set outside of which it is uniformly sub-optimal.
Thus, its minimum is uniquely achieved at some $B_{\mathrm{opt}}$ with $d_2\left(\frac{B_{\mathrm{opt}}}{\tau^*},\frac{B^*}{\tau^*}\right) < c_3'\sqrt{\epsilon} < c_1' \leq r_{\mathrm{min}} \leq  r_1^{(B)} = \frac{R_1}{3\tau^*}$, where we have used \eqref{sqrt-epsilon-bound} in the second inequality.
Combining this with \eqref{R1-meaning}, we get $d_2\left(\frac{B_{\mathrm{opt}}}{\tau^*},\frac{\Theta^*}{\tau^*}\right) < \frac{R_1}{3\tau^*} + \frac{R_1}{\tau^*}  = c_2'$.
Then, because by \eqref{choice-of-epsilon} also $d_2\left(\frac{\Theta}{\tau^*},\frac{\Theta^*}{\tau^*}\right) < \epsilon < c_1'  \leq r_{\mathrm{min}} \leq r_2^{(\Theta)}$ and similarly $\left|\frac{\sigma}{\sigma^*} - 1\right| < r_2^{(\sigma)}$, $|\kappa - 1| < r_2^{(\kappa)}$, and $\frac{|\xi|}{\sigma^*} < r_2^{(\xi)}$, by \eqref{L-locally-Lipshitz-big-B-neighborhood},
\begin{align}
    \cL(B_{\mathrm{opt}};\Theta,G,\sigma,\kappa,\xi) &< \cL(B_{\mathrm{opt}};\Theta,G,\sigma,1,0) + 2L\epsilon \\
    &< L^* + 7L\epsilon,
\end{align}
where the second inequality holds by \eqref{perturbed-value-at-B} and optimality.
By Lemma \ref{lem:meas-gordon-soln}, $B_{\mathrm{opt}}$ is $\sigma(\Theta + \tau G)$ measurable, whence 
\begin{align}
    \inf_{\substack{\bb \in \reals^p\\\frac1{\sqrt p}\|\bb - \btheta\| \leq c_2'\tau^*}} L'(\bb)  = \inf_{\substack{B \in \iota(\reals^p)\\ d_2(B,\Theta^*) \leq c_2'\tau^*}} \cL(B;\iota(\btheta),\iota(\bg),\sigma,\kappa,\xi) &\leq \cL(B_{\mathrm{opt}};\Theta,G,\sigma,\kappa,\xi)\nonumber\\
    & < L^* + 7L\epsilon.\label{Rp-to-L2-upper-bound}
\end{align}
Next, let 
\begin{equation}
A = \left\{\bb \in \reals^p \Biggm\vert W_2(\mu_{(\bb,\btheta)/\tau^*},\mu_{(B^*,\Theta^*)/\tau^*}) \geq c_3'\sqrt{\epsilon},\,\frac1{\sqrt p}\|\bb - \btheta\|\leq c_2'\tau^*\right\}.
\end{equation}
Then
\begin{align}
    \inf_{\bb \in A} L'(\bb) &= \inf_{\bb\in A} \cL(\iota(\bb);\iota(\btheta),\iota(\bg),\sigma,\kappa,\xi) \nonumber \\
    &\geq \inf_{\substack{d_2\left(\frac{B}{\tau^*},\frac{\Theta^*}{\tau^*}\right)\leq c_2'\\ d_2\left(\frac{B}{\tau^*},\frac{B^*}{\tau^*}\right) \geq c_3'\sqrt{\epsilon}}} \cL(B;\Theta,G,\sigma,\kappa,\xi)\nonumber\\
    &\geq \cL(B^*;\Theta^*,G^*,\sigma^*,1,0) + 10L\epsilon,\label{Rp-to-L2-lower-bound}
\end{align}
where the first inequality holds because $\iota(A)$ is contained in the set over which the infimum in the second line is taken, and the second inequality holds by Lemma \ref{lem:gordon-local-stability}.

By \cite[Proposition F.2]{miolane2018distribution}, there exist $\mathsf{C_{conc}},\mathsf{c_{conc}}:\reals_{>0} \rightarrow \reals_{>0}$, non-increasing and non-decreasing respectively, such that for $0 < \epsilon < 1/2$,
\begin{equation}
    \P\left(W_2(\mu_{(\btheta/\tau^*, \bg)},\mu_{\btheta/\tau^*} \otimes \normal(0,1)) \geq \epsilon\right) \leq C\epsilon^{-1}\exp\left(-cp\epsilon^3\log(1/\epsilon)^{-2}\right),
\end{equation}
where $C = \mathsf{C_{conc}}(\mathsf{snr})$, $c = \mathsf{C_{conc}}(\mathsf{snr})$, and $\mathsf{snr} = \|\btheta\|^2/(p\tau^2)$.
Now note that $\E[\exp(\lambda(h_i^2 - 1))] = \frac{e^{-\lambda}}{\sqrt{1-2\lambda}} \leq e^{-\lambda}e^{\lambda + 4\lambda^2} = e^{4\lambda^2}$ for $\lambda \leq 1/4$ because the second derivative of $\log (1 - 2\lambda)^{-1/2}$ is bounded in absolute by 8 on $[-1/4,1/4]$.
Thus, taking $\lambda = \epsilon/8 \in [-1/4,1/4]$ for $0 < \epsilon < 1/2$, we have
\begin{equation}
    \P\left(|\kappa - 1| > \epsilon\right) = \P\left(\left|\sum_{i=1}^n (h_i^2 - 1)\right| > n\epsilon\right) \leq 2\frac{\exp(4n\lambda^2)}{\exp(\lambda n \epsilon)} = 2 \exp\left(- \frac{p\delta}8 \epsilon^2\right).
\end{equation}
By standard Gaussian concentration
\begin{align}
    \P\left(\frac{|\xi|}{\sigma^*} > \epsilon\right) &= \P\left(\left|\left\langle \frac{\bh}{\sqrt n},\frac{\bw/\sigma^*}{\sqrt n}\right\rangle\right| > \epsilon\right) \nonumber \\
    &\leq \P\left(\frac{\sigma}{\sigma^*} > 2 \right) + \P\left( 2\left|\left\langle \frac{\bh}{\sqrt n}, \frac{\bw}{\|\bw\|}\right\rangle\right| > \epsilon \right)\nonumber \\
    &\leq \P\left( \left|\frac{\sigma}{\sigma^*} - 1\right| > \epsilon\right) +  2\exp\left(-\frac{p\delta}{8}\epsilon^2\right).
\end{align}
Because \eqref{choice-of-epsilon} implies \eqref{Rp-to-L2-upper-bound} and \eqref{Rp-to-L2-lower-bound}, we conclude that for $\epsilon < \min\left\{c_1',\frac12\right\}$, we have 
\begin{align*}
    &\P\left( \min_{\frac1{\sqrt{p}}\|\bv\| \leq c_2'\tau^*} L(\bv) > L^* + 7 L\epsilon \text{ or } \min_{\substack{\frac1{\sqrt{p}}\|\bv\| \leq c_2'\tau^*\\\bv \in D_{c_3'^2\epsilon}}} L(\bv) < L^* + 10 L\epsilon \right) \\
    &\qquad\qquad\leq C\epsilon^{-1}\exp\left(-cp\epsilon^3\log(1/\epsilon)^{-2}\right) + 4\exp\left(-\frac{p\delta}{8}\epsilon^2\right) + 2\P\left(\left|\frac{\sigma}{\sigma^*} -1\right|>\epsilon\right) \\
    &\qquad\qquad\leq  \tilde C\epsilon^{-1}\exp\left(-\tilde cp\epsilon^3\log(1/\epsilon)^{-2}\right) + 2\P\left(\left|\frac{\sigma}{\sigma^*} -1\right|>\epsilon\right)
\end{align*}
where 
\begin{equation}\label{eqdef:tilde-Cs}
\begin{gathered}
\tilde C = C + 2 \leq \mathsf{C_{conc}}(\mathsf{snr}) + 2 \\
\tilde c = \min\left\{\mathsf{c_{conc}}(M/\tau^2),\frac{\delta (\log 2)^2}{4}\right\} \leq \min\left\{c, \inf_{\epsilon\in(0,1/2)} \frac{\delta\log(1/\epsilon)^2}{8\epsilon}\right\},
\end{gathered}
\end{equation}
where we have used the monotonicity of $\mathsf{C_{conc}}$ and $\mathsf{c_{conc}}$, replaced $\bb$ with $\btheta + \bv$, and used \eqref{wass-neighb}.
Now \eqref{gordon-concentration} follows after the change of variables $c_3'^2 \epsilon \rightarrow \epsilon$ with the functions
\begin{align*}
    \mathsf{c_1}(M,\tau^*,\sigma^*,\delta) &= {c_3'}^2\min\left\{c_1',\frac12\right\}, & \mathsf{c_2}(M,\tau^*,\sigma^*,\delta) &= \frac{\tilde c (\log2)^2}{c_3'^6\log(2c_3'^2)^2},\\
    \mathsf{c_3}(M,\tau^*,\sigma^*,\delta) &= c_3'^{-2}, & \mathsf{c_4}(M,\tau^*,\sigma^*,\delta) &= c_2',\\
    \mathsf{C}(M,\tau^*,\sigma^*,\delta) &= \mathsf{C_{conc}}(M/\tau^2) + 2, & \mathsf{L}(M,\tau^*,\sigma^*,\delta) &= L,
\end{align*}
where the dependence of all parameters which appear in this display on $M,\tau^*,\sigma^*,\delta$ is specified in Step 5, \eqref{L-rmin-c-defs}, and \eqref{eqdef:tilde-Cs}.
Note that the choice of $\mathsf{c_2}$ is made to guarantee that for all $0 < \epsilon < 1/2$ we have
\begin{equation*}
    \mathsf{c_2} \epsilon^3\log(1/\epsilon)^{-2} \leq \tilde c (\epsilon / c_3'^2)^3 \log (1 / (\epsilon/c_3'^2))^{-2}.
\end{equation*}

\section{Additional simulations}\label{app:simulations}

We now present additional simulation results which demonstrate our results empirically.
We consider three penalties.\\

\noindent {\bf Power of $\ell_2$-norm.}
First we consider the penalty
\begin{align}
    f_p(\bx)&= p^{1-\alpha/2}\|\bx\|_2^\alpha\label{l2-power}
\end{align}
where $\alpha \geq 1$.
When $\alpha = 2$, this penalty is separable.
For this choice of exponent, using this penalty in linear regression implements ridge regression.
In the sequence model, the separable proximal operator \eqref{eqdef:finite-p-prox} applies a linear estimator to each coordinate with a slope which is constant across coordinates and does not depend upon the distribution of the data.
For choices of $\alpha \neq 2$, a linear estimator is applied to each coordinate with a slope which is constant across coordinates but which does depend upon the distribution of the data.
To demonstrate this fact, we plot in the first row of Figure \ref{fig:Simulations} theory and simulation results for three choices of expoenent $\alpha=1,2$, and $4$.
We take $p = 1000$ and consider estimating $\btheta \in \reals^{1000}$ whose $j^\text{th}$ entry is set to the $j/(p+1)^\text{st}$ quantile of the $\mathsf{N}(0,1)$ distribution.
For each of $\tau = .25,1$, and $5$, we once generate observations $\by = \btheta + \tau \bz$  for 
$\bz \sim \mathsf{N}(0,I_p)$ and compute $\widehat \btheta$ as in \eqref{eqdef:finite-p-prox} in $\mathsf{R}$ using the $\mathsf{CVXR}$ package.
We plot the estimated value $\widehat \theta_j$ against the observation $y_j$ for 100 randomly sampled coordinates across the three values of $\tau$.
We also plot theory curves of $\frA_{f_p}(\mu_{\btheta}^{*\tau})$.
The agreement between theory and experiment is extremely good.
We see that for $\alpha = 2$, $\frA_{f_p}(\mu_{\btheta}^{*\tau})$ does not depend upon $\tau$, as we expect for separable, and hence non-adaptive, penalties.
For $\alpha = 1$, we see that shrinkage decreases with noise-level and for $\alpha = 4$ shrinkage increases with noise-level.
In this context, we want greater shrinkage with higher-noise level, suggesting $\alpha > 2$ is a good choice.\\

\noindent {\bf Power of $\ell_1$ norm.}
Next we consider the penalty
\begin{align}
    f_p(\bx)&= p^{1-\alpha}\|\bx\|_1^\alpha,\label{l1-power}
\end{align}
where $\alpha \geq 1$.
When $\alpha = 1$, this penalty is separable.
For this choice of exponent, using this penalty in linear regression implements the LASSO.
In the sequence model, the separable proximal operator \eqref{eqdef:finite-p-prox} applies soft-thresholding to each coordinate with a threshold which is constant across coordiantes and does not depend upon the distribution of the data.
For choices $\alpha > 1$, soft thresholding is applied to each coordinate with a threshold which is constant across coordinates but which does depend upon the distribution of the data.
In the second row of Figure \ref{fig:Simulations}, we plot theory and experiments for $p=1000$ as above.
Now, we take $\btheta$ with has 50 coodinates equal to ${-1}$, 50 coordinates equal to 1, and the remaining 900 coordinates equal to 0. 
We plot the estimated value $\widehat \theta_j$ against the observation $y_j$ for 100 randomly sampled coordinates across the three values of $\tau$. We also plot theory curves of $\frA_{f_p}(\mu_{\btheta}^{*\tau})$.
We see that the threshold increases with noise-level, as it should. 
All computations are done in $\mathsf{R}$ with pacakge $\mathsf{CVXR}$.\\

\noindent {\bf Smoothed ordered weighted $\ell_1$-norms.}
Finally, we display some additional results for the penalty \eqref{sowl} discussed in Secton \eqref{sowl}.

\begin{figure}[h!]
\begin{tabular}{cc}
\centerline{\phantom{A}\hspace{-1cm}\includegraphics[width=0.85\columnwidth]{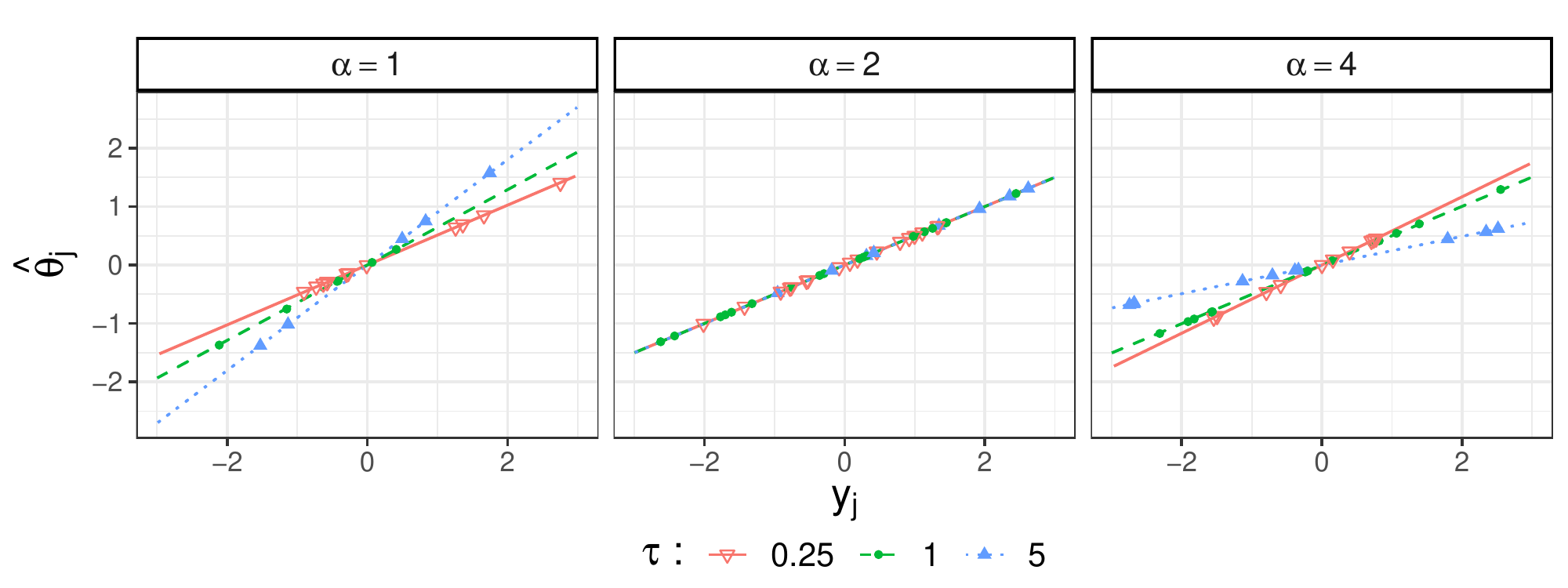}}\\
\centerline{\phantom{A}\hspace{-1cm}\includegraphics[width=0.85\columnwidth]{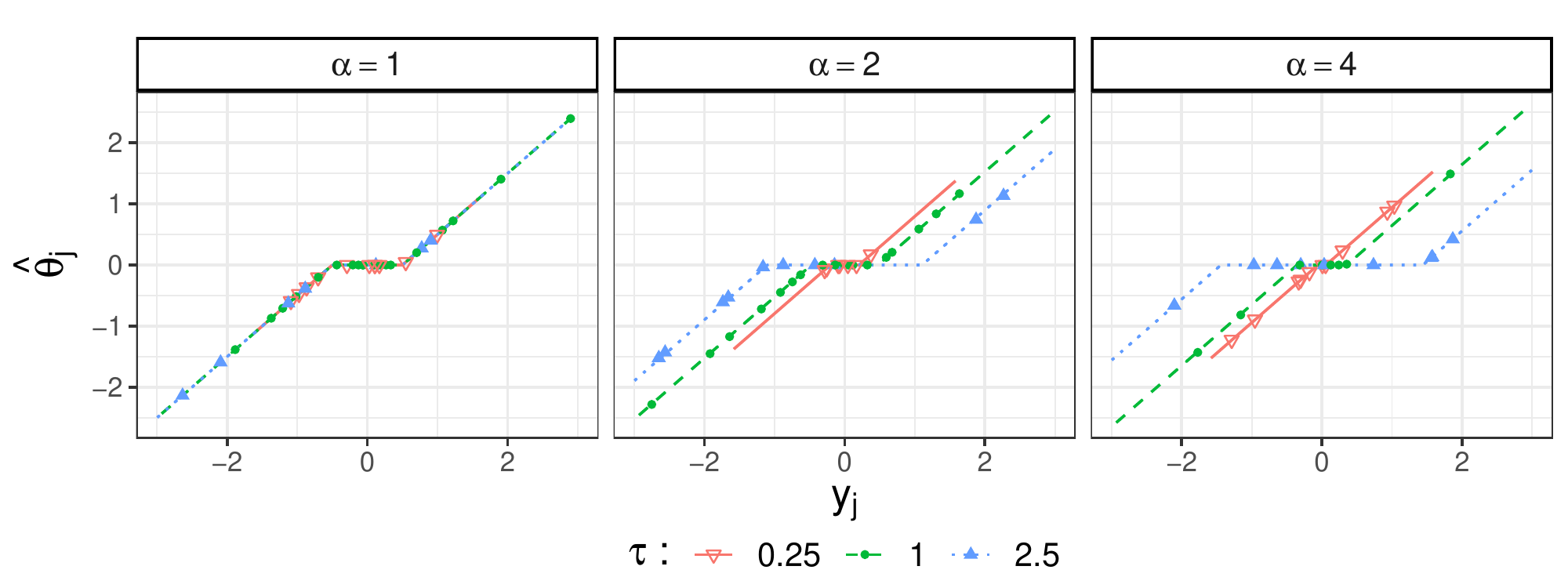}}\\
\centerline{\phantom{A}\hspace{-1cm}\includegraphics[width=0.85\columnwidth]{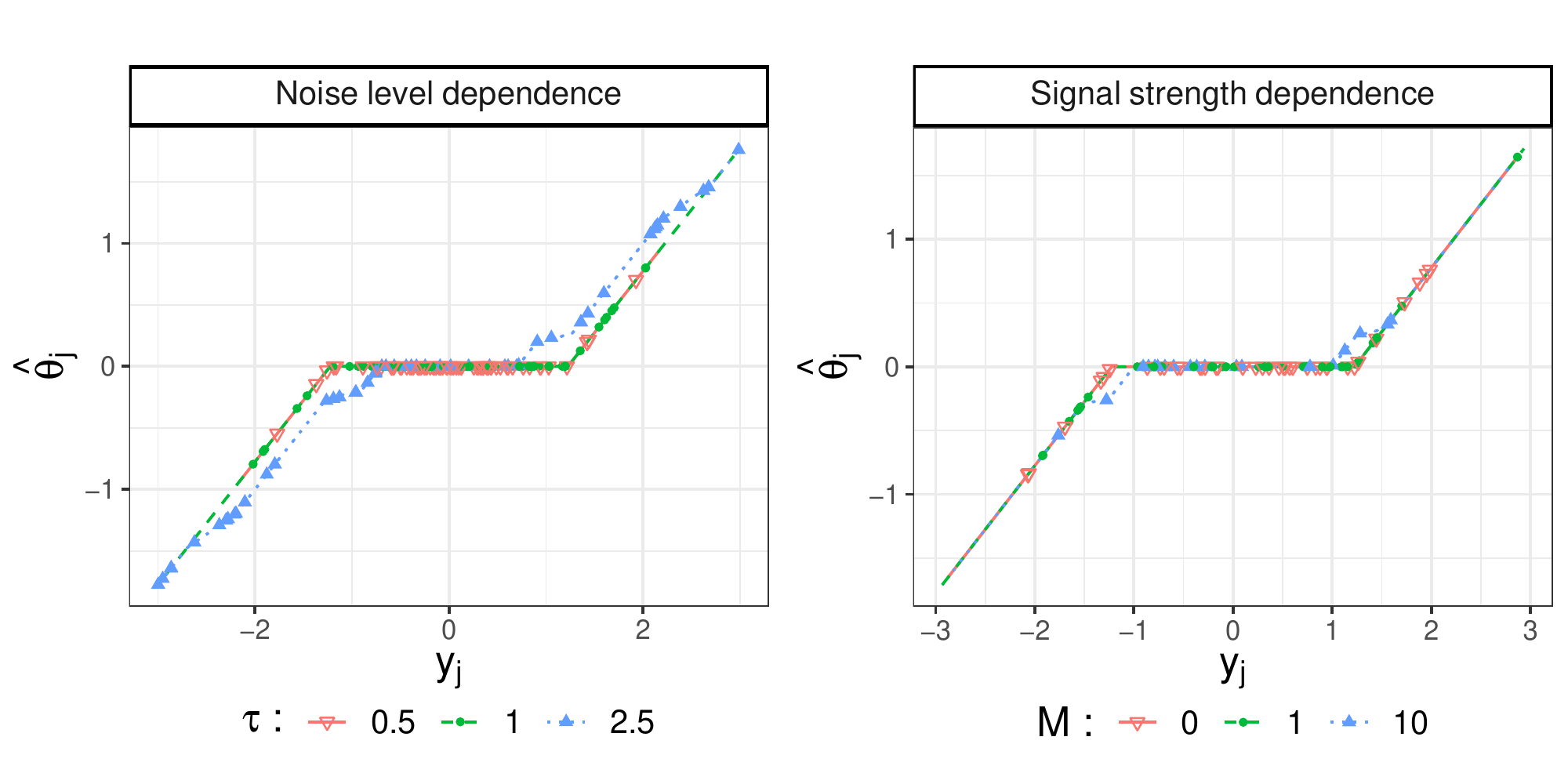}}
\end{tabular}
\caption{Plots of $\frA_{f_p}(y_j)$ vs.\ $y_j$ (theory) and $\widehat \theta_j$ vs.\ $y_j$ (simulation) for various choices of penalty $f_p$ in proximal operator \eqref{eqdef:finite-p-prox}. In all plots, $\by = \btheta + \tau \bz$ with $\bz \sim \mathsf{N}(\bzero,\bI_p)$. Top row: $f_p(\bx) = p^{1 - \alpha/2}\|\bx\|_2^\alpha$, $\mu_{\btheta} \approx \mathsf{N}(0,1)$. Middle row: $f_p(\bx) = p^{1-\alpha/2}\|\bx\|_1^\alpha$, $\mu_{\btheta} = .05 \delta_{-1} + .9 \delta_1 + .05 \delta_1$. Bottom row: $f_p(\bx) = \frac12 \min_{\eta \in \reals^p_+} \sum_{j=1}^p \left(\frac{w_j^2}{\eta_j} + \lambda_j \eta_{(j)}\right)$, $\mu_{\btheta} = .05 \delta_{-M} + .9 \delta_1 + .05 \delta_M$, $\mu_{\blambda} = \frac13 \delta_2 + \frac13 \delta_1 + \frac13 \delta_{.5}$. Bottom left: $M = 1$. Bottom right: $\tau = 1$.} 
\label{fig:Simulations}
\end{figure}

In the third row of Figure \ref{fig:Simulations}, we plot theory and experiments for $p = 1000$.
The left plot repeats Figure \ref{fig:smoothed-OWL-sim}, and we refer the reader to Section \ref{sec:penalized-ls-in-seq-model} for a description of that plot.
The right plot displays the effect of varying $M$ instead of $\tau$.
For a positive number $M$, we take $\btheta$ with 50 coordinates equal to $-M$, 50 coordinates equal to $M$, and 900 coordinates equal to $0$.
We consider three settings of the signal strength $M = 0,1,10$ at noise level $\tau = 1$.

\end{appendices}

\end{document}